\documentclass[11pt,reqno,oneside]{amsart}
\usepackage{graphicx}
\usepackage{amssymb}
\usepackage{epstopdf}
\usepackage[asymmetric,top=3.5cm,bottom=4.3cm,left=3.1cm,right=3.1cm]{geometry}
\usepackage{booktabs} 
\usepackage{array} 
\usepackage{paralist} 
\usepackage{verbatim} 
\usepackage{subfig} 
\usepackage{tabularx}
\usepackage{amsmath,amsfonts,amsthm,mathrsfs,amssymb,cite}
\usepackage[usenames]{color}
\usepackage{bm}
\allowdisplaybreaks \allowdisplaybreaks[4]

\newtheorem{thm}{Theorem}[section]
\newtheorem{cor}[thm]{Corollary}
\newtheorem{lem}{Lemma}[section]
\newtheorem{prop}{Proposition}[section]
\theoremstyle{definition}
\newtheorem{defn}{Definition}[section]
\theoremstyle{remark}

\newtheorem{rem}{Remark}[section]
\numberwithin{equation}{section}

\newcommand{\bmf}[1]{{\mathbf{#1}}}

\newcommand{\sfr}{{\sf{R}}}
\newcommand{\sfi}{{\sf{I}}}

\usepackage{hyperref}
\usepackage{xcolor}
\hypersetup{
  colorlinks,
  linkcolor={blue!80!black},
  urlcolor={blue!80!black},
  citecolor={blue!80!black}
}
\usepackage[capitalize,noabbrev]{cleveref}

\title[Geometric structures of generalized elastic transmission eigenfunctions]{On a local geometric property of the generalized elastic transmission eigenfunctions and application}

\author{Huaian Diao}
\address{School of Mathematics and Statistics, Northeast Normal University,
Changchun, Jilin 130024, China.}
\email{hadiao@gmail.com}

\author{Hongyu Liu}
\address{Department of Mathematics, City University of Hong Kong, Kowloon, Hong Kong SAR, China.}
\email{hongyu.liuip@gmail.com; hongyliu@cityu.edu.hk}

\author{Baiyi Sun}
\address{School of Mathematics and Statistics, Northeast Normal University,
Changchun, Jilin 130024, China.}
\email{1412102726@qq.com}

\date{} 
\begin{document}
\maketitle

\begin{abstract}
Consider the nonlinear and completely continuous scattering map
\[
\mathcal{S}\big((\Omega; \lambda, \mu, V), \mathbf{u}^i\big)=\mathbf{u}_t^\infty(\hat{\mathbf{x}}), \quad \hat{\mathbf{x}}\in\mathbb{S}^{n-1}, 
\]
which sends an inhomogeneous elastic scatterer $(\Omega; \lambda, \mu, V)$ to its far-field pattern $\mathbf{u}_t^\infty$ due to an incident wave field $\mathbf{u}^i$ via the Lam\'e system. Here, $(\lambda, \mu, V)$ signifies the medium configuration of an elastic scatterer that is compactly supported in $\Omega$. In this paper, we are concerned with the intrinsic geometric structure of the kernel space of $\mathcal{S}$, which is of fundamental importance to the theory of inverse scattering and invisibility cloaking for elastic waves and has received considerable attention recently. It turns out that the study is contained in analysing the geometric properties of a certain non-selfadjoint and non-elliptic transmission eigenvalue problem. We propose a generalized elastic transmission eigenvalue problem and prove that the transmission eigenfunctions vanish locally around a corner of $\partial\Omega$ under generic regularity criteria. The regularity criteria are characerized by the H\"older continuity or a certain Fourier extension property of the transmission eigenfunctions. As an interesting and significant application, we apply the local geometric property to derive several novel unique identifiability results for a longstanding inverse elastic problem by a single far-field measurement.

\medskip


\noindent{\bf AMS Class (2010):}~~35Q60, 78A46 (primary); 35P25, 78A05, 81U40 (secondary).

\end{abstract}

\maketitle

\section{Introduction}

\subsection{Background and motivation}

We first introduce the elastic scattering due to an embedded inhomogeneous medium and an incident wave field, which is the physical origin of our study. Let $\lambda, \mu$ be real constants satisfying the following strong convexity condition
\begin{equation}\notag 
 \mu>0,\ \ n\lambda+2\mu>0\ \mbox{ for }\ n=2,3.
\end{equation}
Let $\Omega\subset\mathbb{R}^n$, $n=2,3$, be a bounded Lipschitz domain with a connected complement $\mathbb{R}^n\backslash\overline{\Omega}$. Suppose that $V\in L^\infty(\mathbb{R}^n)$ is a real-valued function with $\mathrm{supp}(V)\subset\Omega$. The parameters $\lambda, \mu$ and $V$ characterize the elastic medium configuration of the space $\mathbb{R}^n$, with $\lambda, \mu$ and $1+V$ respectively denoting the bulk moduli and density. Throughout, we assume that $V$ is nontrivial, which is also referred to as a scattering potential. $(\Omega; \lambda, \mu, V)$ signifies an inhomogeneous scatterer embedded in the uniformly homogeneous space $\mathbb{R}^n$. Let $\mathbf{u}^i$ be an incident field which is a $\mathbb{C}^n$-valued entire solution to the following Lam\'e system:
\begin{equation}\label{eq:in1}
\mathcal{L}\mathbf{u}^i+\omega^2\mathbf{u}^i=0\ \ \mbox{in}\ \mathbb{R}^n,\quad \mathcal{L}\mathbf{u}^i:=\lambda \Delta \mathbf{u}^i+(\lambda+\mu) \nabla \nabla \cdot {\mathbf{u}^i}, 
\end{equation}
where $\omega\in\mathbb{R}_+$ signifies the angular frequency of the time-harmonic wave propagation. The interaction between the incident field $\mathbf{u}^i$ and the scatterer $(\Omega;\lambda, \mu, V)$ generates the elastic scattering, which is governed by the following Lam\'e system:
\begin{equation}\label{eq:scat1}
\mathcal{L}\bmf{u}+\omega^2(1+V)\mathbf{u}=0\ \ \mbox{in}\ \mathbb{R}^n;\ \ \mathbf{u}=\mathbf{u}^i+\mathbf{u}^{\mathrm{sc}};\ \ \mathbf{u}^{\mathrm{sc}}\ \mbox{is radiating}. 
\end{equation}
Here by radiating, we mean that $\mathbf{u}^\mathrm{sc}$ satisfies the following Kupradze radiation condition
\begin{equation}\label{eq:kp1}
\lim_{r\rightarrow\infty}r^{\frac{n-1}{2}}\left(\frac{\partial \bmf{u}_\beta^{\mathrm{sc} }}{\partial r}-\mathrm{i}k_\beta \bmf{u}_\beta^{\mathrm{sc} }\right) =\,\mathbf 0,\quad r:=|\mathbf{x}|,\quad \beta=p, s,
\end{equation}
where 
\begin{equation}\label{eq:decomp1}
\bmf{u}^{\mathrm{s c} }=\bmf{u}_{p}^{\mathrm {s c} }+\bmf{u}_{s}^{\mathrm{s c} },\ \	 \bmf{u}_{p}^{\mathrm {s c} }:=-\frac{1}{k_{p}^{2}} \nabla\left( \nabla \cdot \bmf{ u}^{\mathrm {s c} }\right ), \quad \bmf{ u}_{s}^{\mathrm{s c}}:=\begin{cases}  \frac{1}{k_{s}^{2}}\nabla \times \nabla \times  {\mathbf u}^{\mathrm {s c} }\ \ & (\mbox{3D})\\ 
 \frac{1}{k_{s}^{2}} \bf{curl} \operatorname{curl} u^{\mathrm {s c} }\ \ & (\mbox{2D})	 
\end{cases},
\end{equation}
and 
\begin{equation}\label{eq:kpks}
	k_p :=\frac{\omega}{\sqrt{ 2\mu+\lambda }}, \ \ k_s:=\frac{\omega}{\sqrt{ \mu}}. 
\end{equation}
In \eqref{eq:decomp1}, the two-dimensional operators $\bf{curl}$  and $\operatorname{curl}$ are defined respectively by
\[
 {\rm curl}\, \bmf{ u}=\partial_1 u_2-\partial_2 u_1, \quad {\bf
curl}\, {u}=(\partial_2 u, -\partial_1 u)^\top,
\]
with $\bmf{u}=(u_1, u_2)$ and $u$ being vector-valued and scalar functions, respectively. It is noted that \eqref{eq:decomp1} is the Helmholtz decomposition of the vector field $\bmf{u}^\mathrm{sc}$, where $\bmf{u}_p^{\mathrm sc}$ and $\bmf{u}_s^{\mathrm sc}$ are respectively referred to as the compressional and shear parts of $\bmf{u}^{\mathrm sc}$. This decomposition shall also play an important role in our subsequent analysis. The Kupradze radiation condition characterizes the outward propagating nature of the scattered field $\bmf{u}^\mathrm{sc}$. The well-posedness of the scattering system \eqref{eq:scat1} is known \cite{Hahner98}, and in particular there exists a unique solution $\bmf{u}\in H_{loc}^1(\mathbb{R}^n)$ which admits the following asymptotic expansions:
\begin{equation}\label{eq:far-field}
\begin{split}
	\bmf{u}_{\beta}^{\mathrm{sc}}(\bmf{x}) =& \frac{{e}^{\mathrm{i} k_{p} r}}{{r}^{\frac{
	n-1}{2} }}\left\{u_{p}^{\infty}(\hat{\bmf{x}}) \hat{\bmf{x}}+O\left(\frac{1}{r}\right)\right\},\ \ \hat{\bmf{x}}:=\bmf{x}/|\bmf{x}|,\ \ \beta=p, s, \\
	\bmf{u}^{\mathrm{sc}}(\bmf{x})=& \frac{{e}^{\mathrm{i} k_{p} r}}{{r}^{\frac{
n-1}{2} }} u_{p}^{\infty}(\hat{\bmf{x} } ) \hat{\bmf{x}}+\frac{{e}^{\mathrm{i} k_{s} r}}{{r}^{\frac{
n-1}{2} }} u_{s}^{\infty}(\hat{\bmf{x} }) \hat{\bmf{x} }^{\perp}+O\left(\frac{1}{r^{(n+1) / 2}}\right) ,
\end{split}
\end{equation}
as $r =|\bmf{x} | \rightarrow \infty$, where $u_{p}^{\infty}$ and $u_{s}^{\infty}$ are both scalar functions defined on the unit sphere $\mathbb S^{n-1}:=\{\hat{\bmf{x}}\in\mathbb{R}^n \big| |\hat{\bmf{x}}|=1\}$. Define the far-field pattern $\bmf{u}_t^\infty$ of $\bmf{u}^{\mathrm{sc}} $ as $\bmf{u}_t^{\infty}(\hat{\bmf{x}} ) :=u_{p}^{\infty}(\hat{\bmf{x}}) \hat{\bmf{x}}+u_{s}^{\infty}(\hat{\bmf{x}}) \hat{\bmf{x}}^{\perp}.$ One clearly has the unique correspondences:
$\bmf{u}_{p}^{\infty}(\hat{\mathbf{x}} )=\bmf{u}_t^{\infty}(\hat{\bmf{x}}) \cdot \hat{\bmf{x}}$ and $\bmf{u}_{s}^{\infty}(\hat{\bmf{x}} )=\bmf{u}_t^{\infty}(\hat{\bmf{x}}) \cdot \hat{\bmf{x}}^{\perp}$. Moreover, due to the Rellich Theorem \cite{CK}, the correspondence between $\bmf{u}_t^\infty$ and $\mathbf{u}^{\mathrm sc}$ is also one-to-one. 

Define the scattering operator $\mathcal{S}$ as
\begin{equation}\label{eq:so1}
\mathcal{S}\big((\Omega; \lambda, \mu, V), \mathbf{u}^i\big)=\mathbf{u}_t^\infty(\hat{\mathbf{x}}), \quad \hat{\mathbf{x}}\in\mathbb{S}^{n-1}, 
\end{equation}
which is implicitly defined by the scattering system \eqref{eq:scat1}. It can be directly verified that $\mathcal{S}$ is nonlinear with respect to the argument $(\Omega;\lambda,\mu, V)$, whereas it is linear with respect to the argument $\mathbf{u}^i$.  An inverse problem of industrial importance in geophysical exploration and medical imaging is to recover $(\Omega; \lambda,\mu, V)$ by knowledge of $\mathbf{u}_t^\infty$. In this paper, we are mainly concerned with the kernel space of $\mathcal{S}$, namely
\begin{equation}\label{eq:so2}
\mathcal{S}\big((\Omega; \lambda, \mu, V), \mathbf{u}^i\big)= \mathbf{0}, 
\end{equation}
which corresponds to the physical scenario that no scattering, a.k.a. invisibility occurs. In particular, we consider the geometric structures of the kernel space $\mathrm{Ker}(\mathcal{S})$, namely, the quantitative relationships between the configurations in $\mathrm{Ker}(\mathcal{S})$ and the intrinsic geometries of $\Omega$. Intuitively, if non-scattering/invisibility occurs, 
one has nil scattering information from the far-field measurement and hence the inverse problem described above fails. However, it is much surprising that the geometric understanding of $\mathrm{Ker}(\mathcal{S})$ can not only provide quantitative understanding of how the waves behave inside the scattering object when invisibility occurs with respect to exterior observations, but can also provide a completely new perspective of the inverse problem \eqref{eq:so1} for both the derivation of new theoretical uniqueness results and the development of novel numerical reconstruction algorithms. Before discussing more about these aspects, we present the so-called elastic transmission eigenvalue problem, which shall provide a broader spectral perspective of studying the geometric structure of $\mathrm{Ker}(\mathcal{S})$. 

Consider the following PDE system for ${\mathbf v}=(v_\ell)_{\ell=1}^n\in H^1(\Omega)^n$ and ${\mathbf w}=(w_\ell)_{\ell=1}^n\in H^{1}(\Omega)^n$:
\begin{align}\label{eq:lame1n1}
\left\{
\begin{array}{l}
{\lambda \Delta {\mathbf v}+(\lambda+\mu) \nabla \nabla \cdot {\mathbf v}+\omega^{2} {\mathbf v}={\mathbf 0} }\hspace*{2.04cm}\ \mbox{ in } \Omega, \\[5pt]
{\lambda \Delta {\mathbf w}+(\lambda+\mu) \nabla \nabla \cdot {\mathbf w}+\omega^{2}(1+V) {\mathbf w}={\mathbf 0} }\hspace*{0.5cm}\ \mbox{ in } \Omega, \\[5pt]
{\mathbf w}={\mathbf v},\quad T_{\nu}{\mathbf v}=T_{\nu}{\mathbf w} \hspace*{4.1cm}\mbox{ on } \partial \Omega,
  \end{array}
\right.
\end{align}
with $\nu$ signifying the outward unit normal to $\partial \Omega$, and the {\it boundary traction} operator $T_{\nu }{\mathbf v}$ defined as 
\begin{equation}\label{eq:Tu def}
T_{\nu }{\mathbf v}=	\begin{cases}
		2\mu\partial_{\nu }{\mathbf v}+\lambda {\nu} \left(\nabla \cdot {\mathbf v}\right)+\mu (\partial_{2}v_{1}-\partial_{1}v_{2})  \nu^\perp ,\quad \mbox{ for } n=2,\\[5pt]
		2 \mu \partial_{\nu }{\mathbf v}+\lambda {\nu} \left(\nabla \cdot {\mathbf v}\right)+\mu {\nu} \times(\nabla \times {\mathbf v}),\hspace*{1.1cm} \mbox{ for } n=3,
	\end{cases}
\end{equation}
where for $n=2$, $\nu^\perp \in \mathbb R^2$ denotes the unit vector obtained via rotating $\nu$ anti-clockwise by $\pi/2$. It is clear that $\mathbf{v}=\mathbf{w}\equiv \mathbf{0}$ are a pair of trivial solutions to \eqref{eq:lame1n1}.
If for a certain $\omega \in\mathbb{R}_+$, there exists a pair of nontrivial solutions $({\mathbf v}, {\mathbf w} )\in H^1(\Omega)^n\times H^1(\Omega)^n$ to \eqref{eq:lame1n1}, then $\omega$ is called an elastic transmission eigenvalue and $({\mathbf v}, {\mathbf w})$ is referred to as the corresponding pair of elastic transmission eigenfunctions.
The connection between $\mathrm{Ker}(\mathcal{S})$ and the transmission eigenvalue problem \eqref{eq:lame1n1} can be described as follows. If $\mathbf{u}_t^\infty\equiv 0$, one clearly has $\mathbf{u}^{\mathrm{sc}}=0$ in $\mathbb{R}^n\backslash\overline{\Omega}$. This in turn yields that $\mathbf{u}=\mathbf{u}^i$ in $\mathbb{R}^n\backslash\overline{\Omega}$ for the scattering system \eqref{eq:scat1}. In such a case, one can show that $\mathbf{v}=\mathbf{u}^i|_{\Omega}$ and $\mathbf{w}=\mathbf{u}|_{\Omega}$ are a pair of transmission eigenfunctions. However, if $(\mathbf{v},\mathbf{w})$ are a pair of transmission eigenfunctions, one cannot directly have the non-scattering/invisibility unless $\mathbf{v}$ can be (analytically) extended to an entire solution to \eqref{eq:in1} to generate an incident field. Nevertheless, $(\mathbf{v}, \mathbf{w})$ is located in any small neighbourhood of $\mathrm{Ker}(\mathcal{S})$ in the sense that $\mathbf{v}$ and $\mathbf{w}$ can respectively approximate $\mathbf{u}^i$ and $\mathbf{u}$ in $\Omega$ within $\varepsilon$-accuracy for any $\varepsilon>0$ such that the corresponding scattering amplitude is also of order $\varepsilon$. This viewpoint has been verified in \cite{BL2016,BL2017b,JL} for the acoustic scattering governed by the Helmholtz equation, and one should be able to show similar results for the elastic scattering by following a similar spirit. However, this is beyond the scope the current study. 

Due to its physical significance, the transmission eigenvalue problems have been extensively and intensively investigated in the literature, especially associated with the acoustic and electromagnetic scattering. It turns out that the transmission eigenvalue problems are non-elliptic and non-selfadjoint, and the corresponding mathematical study is highly challenging and intriguing, which is of significant interest in its own sake for the spectral theory of partial differential operators. We refer to \cite{CKreview,CHreview,Liureview} for historical accounts and surveys on the existing developments of the transmission eigenvalue problems. It is pointed out that the spectral study is less touched for the elastic transmission eigenvalue problems than that for the acoustic and electromagnetic transmission eigenvalue problems. Moreover, the existing results are concerned more about the spectral properties of the transmission eigenvalues and there are much fewer results on the intrinsic properties of the transmission eigenfunctions. Recently, there are considerable efforts in the literature in unveiling the distinct geometric structures of the transmission eigenfunctions. 
In \cite{BL2017b}, a local geometric structure was discovered for the acoustic transmission eigenfunctions, showing that they are generically vanishing around a corner on $\partial\Omega$. The discovery is motivated by the relevant study \cite{BL2016,SPV,BPS} which verified that if a scattering potential possesses a corner, it scatters every incident field non-trivially and stably. The vanishing property strongly depends on the regularity of the transmission eigenfunctions, and has been established under two regularity conditions. The first one is the H\"older continuity of the transmission eigenfunctions locally around the corner. This means that the transmission eigenfunctions should be more regular than $H^1$, say by the standard Sobolev embedding, $W^{1,\frac{n}{n-\alpha}}$ or $W^{2,\frac{n}{2-\alpha}}$ with $\alpha\in (0, 1)$. On the other hand, as evidenced by the numerical study \cite{BLLW}, there are indeed cases where the transmission eigenfunctions are not vanishing and instead they are localizing, especially for certain concave corners. In order to gain more insightful understanding of the regularity effect, another criterion was introduced in terms of the growth rate of the Herglotz densities which arise from the Fourier extension of the transmission eigenfunctions \cite{BL2017b,DCL,DDL}. There are several further studies on the locally vanishing property of the transmission eigenfunctions in different physical scenarios \cite{Bsource,EBL,BL2018,BXL,CX}. In \cite{CDHLW}, a global rigidity property is discovered, showing that the transmission eigenfunctions tend to localize on $\partial\Omega$. Both the local and global geometric structures of transmission eigenfunctions can produce interesting and significant applications. First, they are of fundamental importance to the invisibility cloaking which is a frontier technology \cite{GKLU5,LiuUhl}. Second, they have been used to establish novel unique identifiability results for the inverse scattering problems by a single far-field measurement \cite{Bsource,EBL,BL2016,BL2017,BL2018,CX,CDL,DCL,LT,LTY}, which constitutes a longstanding problem in the inverse scattering theory \cite{CK,LiuZou4}. Furthermore, in \cite{CDHLW}, a super-resolution wave imaging scheme was developed by making use the geometric properties of the transmission eigenfunctions. 

However, most of the existing studies discussed above are concerned with the transmission eigenfunctions associated with the acoustic or electromagnetic scattering. There is very limited study on the elastic transmission eigenfunctions due to their more complicated physical and mathematical nature. In \cite{EBL}, the authors proved that the elastic transmission eigenfunctions locally vanish around a corner under the condition that the eigenfunctions are H\"older continuous in the domain. The result was applied to deriving a novel uniqueness result in determining an unknown elastic source from its far-field pattern. In this paper, we shall provide a much more comprehensive study of this important geometric property for the elastic transmission eigenfunctions. First, we shall consider a more general formulation of the elastic transmission eigenvalue problem which includes \eqref{eq:lame1n1} as a special case. Second, we establish the local vanishing property under both the H\"older continuity and the Fourier extension property of the generalised transmission eigenfunctions. Third, we apply the newly established results to deriving two novel unique identifiability results for the inverse elastic problem in determining the polygonal/polyhedral support of an inhomogeneous medium independent of its content. Finally, we would like to briefly discuss the mathematical strategies that we develop to derive the results. In order to establish the geometric property, one needs to track the singularity of the transmission eigenfunctions (with respect to their smoothness) induced by the geometric singularity of the domain, namely the corner. To achieve that end, we develop a microlocal argument that can be localized around the corner. An integral identity involving the difference of the two transmission eigenfunctions as well as a special type of CGO (Complex Geometric Optics) solutions is a critical ingredient in our study. Compared to most of the existing studies mentioned earlier, there is a boundary integral terms due to the generalized transmission conditions in our study, which together with the more complicated nature of the Lam\'e system, makes the related analysis and estimates highly technical and subtle. In tracking the order of the asymptotic parameter in the phase of the CGO solutions, we manage to derive the desired results.

\subsection{Mathematical setup and summary of the main findings}

Let $\eta\in L^{\infty}(\partial\Omega)$ be a real-valued function. Instead of \eqref{eq:lame1n1}, we consider the following generalized elastic  transmission eigenvalue problem for ${\mathbf v}=(v_\ell)_{\ell=1}^n,{\mathbf w}=(w_\ell)_{\ell=1}^n\in H^{1}(\Omega)^n$,
\begin{align}\label{eq:lame1}
\left\{
\begin{array}{l}
{\lambda \Delta {\mathbf v}+(\lambda+\mu) \nabla \nabla \cdot {\mathbf v}+\omega^{2} {\mathbf v}={\mathbf 0} }\hspace*{2cm}\ \mbox{ in } \Omega, \\[5pt]
{\lambda \Delta {\mathbf w}+(\lambda+\mu) \nabla \nabla \cdot {\mathbf w}+\omega^{2}(1+V) {\mathbf w}={\mathbf 0} }\hspace*{0.5cm}\ \mbox{ in } \Omega, \\[5pt]
{\mathbf w}={\mathbf v},\quad T_{\nu}{\mathbf v}+\eta {\mathbf v}=T_{\nu}{\mathbf w} \hspace*{3.3cm}\mbox{ on } \Gamma. 
  \end{array}
\right.
\end{align}
where $\Gamma\subset\partial\Omega$ is an open subset. We note that if $\eta\equiv 0$ and $\Gamma=\partial\Omega$, the transmission eigenvalue problem \eqref{eq:lame1} is reduced to \eqref{eq:lame1n1}. Hence, we refer to \eqref{eq:lame1} as the generalised transmission eigenvalue problem. It is particular to note that the transmission condition $T_{\nu}{\mathbf v}+\eta {\mathbf v}=T_{\nu}{\mathbf w}$ not only brings mathematical generalisation but also 
is physically meaningful. In fact, it is referred to as the conductive transmission condition in the context of electromagnetic scattering, which arises in effectively describing a thin layer of highly conducting coating \cite{Ang,CDL}. In \eqref{eq:lame1}, the generalised transmission condition can also be used to effectively describe a thin layer of highly lossy elastic coating. However, we shall not explore more about this point since it is not the focus of the current article. If $\eta\equiv 0$, the existence and properties of the transmission eigenvalues to \eqref{eq:lame1} have been studied in \cite{Bellis2013,Bellis2010}. In this paper, we shall study the intrinsic geometric properties of the transmission eigenfunctions assuming their existence in the general case.  

The major geometric finding can be sketched as follows. Let $\Gamma=\Gamma^-\cup\Gamma^+$, where $\Gamma^\pm$ are two non-collinear/non-coplanar line segments or planes in 2D and 3D, respectively. That is, $\Gamma^-$ and $\Gamma^+$ form a (non-degenerate) corner on $\partial\Omega$. Under mild conditions on $V$ and $\eta$ as well as the necessary regularity requirements on $\mathbf{v}$ and $\mathbf{w}$ as discussed above, it is shown that $\mathbf{v}$ and $\mathbf{w}$ are vanishing around the corner. The 2D results are contained in Theorems~\ref{thm:29} and \ref{thm:23} and Corollary~\ref{cor:22} for $\eta\neq 0$ and $\eta\equiv 0$, respectively, whereas the corresponding 3D results are contained in Theorem~\ref{thm:31} and Corollaries~\ref{cor:3.1} and \ref{cor:34 eta}. According to our earlier discussion, the geometric results imply that when non-scattering or nearly non-scattering occurs, the incident and the total wave fields propagate in a peculiar manner that avoids the corner places of the inhomogeneous medium. Moreover, the results can be used to establish novel unique identifiability results the geometrical inverse elastic problem, which are contained in Theorems~\ref{thm:41} and \ref{Th: unique eta}. 

Before proceeding further to prove our main result in $\mathbb R^2$ in Section \ref{sec:2}, we would like to summarize our main methodologies to provide the readers a global picture of our study. Consider the elastic transmission eigenfunctions $(\mathbf v,\mathbf w)$ fulfilling \eqref{eq:lame1}. We use the elastic Herglotz wave function to approximate  $\mathbf v$ with certain accuracy and kernel increasing property, which severs as a certain regularity characterization for  $\mathbf v \in H^1(\Omega )^2$. By virtue of the complex geometric optics (CGO) solution introduced in \cite{EBL}, we establish the integral equality via the Green formula. The asymptotic decay of all integrals in the underlying integral equality  with respect to the asymptotic parameter in the CGO is carefully studied, where we extract the leading order terms in the aforementioned asymptotic  analysis. With the above preparations, we prove the vanishing property of the elastic transmission eigenfunction near a planar corner. For the 3D case, by using the dimensional reduction technique and similar to the 2D result, we can establish the local geometrical characterization of the elastic transmission eigenfunction near an edge corner in $\mathbb R^3$ under generic conditions, which shall be clearer from our subsequent analysis in Section \ref{sec:3}.

The rest of the paper is organized as follows. In Sections \ref{sec:2} and \ref{sec:3}, we present the studies in two and three dimensions respectively. Section \ref{sect:4} is devoted to the study of the inverse elastic problem.

\section{vanishing near corners of generalized elastic transmission eigenfunctions: two-dimensional case}\label{sec:2}

In this section, we consider the geometric property of the generalized elastic transmission eigenfunction to \eqref{eq:lame1} in two dimensions. First, we introduce the geometric setup of our study. For $\mathbf{x}=(x_{1}, x_{2})^\top  \in \mathbb R^2$, the polar coordinate of $\mathbf x$ is given by $\mathbf{x}=(r \cos\theta, r \sin\theta)^\top$. Denote an open sector $ W \Subset \mathbb{R}^{2}$ and its boundary $\Gamma^{\pm}$ as follows:
\begin{equation}\label{eq:lame2}
\begin{split}
W&=\{\mathbf{x}\in \mathbb{R}^{2}|{\mathbf x}\neq\mathbf{ 0},\ \ \theta_{m}<\arg(x_{1}+{\rm i}x_{2})<\theta_{M}\},\\
\Gamma^-&=\{\mathbf{x}\in \mathbb{R}^{2}|{\mathbf x}\neq\mathbf{ 0},\ \ \arg(x_{1}+{\rm i}x_{2})=\theta_{m}\}, \\
\Gamma^+&=\{\mathbf{x}\in \mathbb{R}^{2}|{\mathbf x}\neq\mathbf{ 0},\ \ \arg(x_{1}+{\rm i}x_{2})=\theta_{M}\},
\end{split}
\end{equation}
where $-\pi<\theta_{m}<\theta_{M}<\pi$.
Let $B_{h}$ and $B_{\varepsilon}$  denote open disks centered at $\mathbf 0$ of radii $h \in \mathbb R_+$ and $\varepsilon \in \mathbb R_+ $ with $ \varepsilon < h $, respectively. In the sequel, we set
\begin{equation}\label{eq:SIGN}
	S_h= W\cap B_h,\, \Gamma_h^{\pm }= \Gamma^{\pm } \cap B_h,\, \overline  S_h=\overline{ W} \cap B_h, \, \ \mbox{and}\ 	\Lambda_h=S_h \cap \partial B_h.	
\end{equation}

The elastic Herglotz wave function $\mathbf{v}_{\mathbf g}$ in $\mathbb R^2$ is defined by
\begin{equation}\label{eq:h2d}
 \begin{aligned}
\mathbf{v}_{\mathbf g}=e^{-\frac{{\rm i}\pi}{4}}\int_{\mathbb S^{1}}\Big   \{\sqrt{\frac{k_{p}}{\omega }}e^{ {\rm i} k_{p}\mathbf{d}\cdot \mathbf{x}}g_{p}({\mathbf d})\mathbf{d}+\sqrt{\frac{k_{s}}{\omega }}e^{{\rm i} k_{s}\mathbf{d}\cdot \mathbf{x}}g_{s}({\mathbf d})\mathbf{d}^{\perp}\Big  \}{\rm d}\sigma(\mathbf{d}),
 \end{aligned}
 \end{equation}
where the kernel  ${\mathbf g} =(g_{p},g_{s})^{\top }$ satisfies  $ g_{p},\, g_{s}\in L^{2}({\mathbb S^{1}})$,   ${\mathbf d}, \mathbf{d}^{\perp} \in \mathbb S^1$ and ${\mathbf d} \perp \mathbf{d}^{\perp}$.

\begin{lem}\cite[Theorem 3.4]{Arens}\label{lem:herg}
	Let $D \Subset \mathbb R^2$ be a bounded Lipschitz domain with a connected complement. Then the set of elastic Herglotz wave functions is dense with respect to the $H^1(D )^2$-norm in the set of solutions to the Lam\'e  system \begin{equation}\label{eq:navier}
\mathcal{L}  \mathbf v +\omega^{2} {\mathbf v}={\mathbf 0} , \quad  \mathcal{L} := \lambda \Delta +(\lambda+\mu) \nabla \left(\nabla \cdot\right)\quad\mbox{in}\ \ D. 
\end{equation}
\end{lem}

%
%
%

 By  virtue of Lemma \ref{lem:herg}, for any generalized elastic transmission eigenfunctions $(\mathbf v ,\mathbf w) \in H^1(\Omega )^2 \times H^1(\Omega )^2$ to \eqref{eq:lame1}, there exists a sequence of elastic Herglotz wave functions $\mathbf{v}_{j}$ given by
 \begin{equation}\label{eq:vj}
 \begin{aligned}
\mathbf{v}_{j}(\mathbf x ) =e^{-\frac{{\rm i}\pi}{4}}\int_{\mathbb S^{1}}\Big   \{\sqrt{\frac{k_{p}}{\omega }}e^{ {\rm i} k_{p}\mathbf{d}\cdot \mathbf{x}}g_{jp}({\mathbf d})\mathbf{d}+\sqrt{\frac{k_{s}}{\omega }}e^{{\rm i} k_{s}\mathbf{d}\cdot \mathbf{x}}g_{js}({\mathbf d})\mathbf{d}^{\perp}\Big  \}{\rm d}\sigma(\mathbf{d}),
 \end{aligned}
 \end{equation}
 which can approximate $\mathbf v$ to an arbitrary accuracy in $H^1(\Omega )^2$. It is clear that $\mathbf{v}_j$ can be regarded as the Fourier extension of $\mathbf{v}$. 

In what follows, we shall split the real and imaginary parts of the elastic transmission eigenfunctions $({\mathbf v}, {\mathbf w})$ to \eqref{eq:lame1}  as
\begin{equation}\label{eq:decouple}
{\mathbf v}={\mathbf v}_{\sf R}+{\mathrm i} {\mathbf v}_{\sfi},\quad {\mathbf w}={\mathbf w}_{\sfr}+{\mathrm i} {\mathbf w}_{\sfi}.	
\end{equation}
It is straightforward to verify that both $(\mathbf{v}_{\sf R}, \mathbf{w}_{\sf R})$ and $(\mathbf{v}_{\sf R}, \mathbf{w}_{\sf R})$ satisfy \eqref{eq:lame1}. In a similar manner, we let the real and imaginary part of the kernel functions $g_{j\beta }({\mathbf d})$ ($\beta=p,s$) of \eqref{eq:vj} be given by
\begin{equation}\label{eq:gjpgjs}
	g_{j\beta }({\mathbf d})=g_{j\beta }^{\sfr }({\mathbf d})+\mathrm i g_{j\beta }^{\sfi }({\mathbf d}).
\end{equation}
We can derive the following auxiliary propositions. 

\begin{prop}\label{Pro:21}
Let the elastic Herglotz wave function $\mathbf v_{j}$ be defined by \eqref{eq:vj}. Let
\begin{equation}\label{eq:vrj 28}
\mathbf v_{j} (\mathbf x)=
\mathbf v_{j} ^\sfr(\mathbf x)+{\rm i}\mathbf v_{j}^\sfi (\mathbf x),
\end{equation}
where $\mathbf v_{j} ^\sfr(\mathbf x)$ and $\mathbf v_{j}^\sfi (\mathbf x)$ are real valued functions. Then
\begin{align}\label{vjrd}
\mathbf v_{j} ^\sfr(\mathbf x)=& \frac{1}{\sqrt 2}\int_{\mathbb S^1}\sqrt{\frac{k_{p}}{\omega }}\bigg ( \cos (k_{p} \mathbf{d}\cdot \mathbf{x})g_{jp}^\sfr ({\mathbf d})+ \cos (k_{p} \mathbf{d}\cdot \mathbf{x})g_{jp}^\sfi ({\mathbf d})  +\sin (k_{p} \mathbf{d}\cdot \mathbf{x})g_{jp}^\sfr ({\mathbf d}) \notag
\\&-\sin (k_{p} \mathbf{d}\cdot \mathbf{x})g_{jp}^\sfi ({\mathbf d}) \bigg )\mathbf{d}+
\sqrt{\frac{k_{s}}{\omega }}\bigg ( \cos (k_{s} \mathbf{d}\cdot \mathbf{x})g_{js}^\sfr ({\mathbf d})+ \cos (k_{s} \mathbf{d}\cdot \mathbf{x})g_{js}^\sfi ({\mathbf d}) \notag \\
& +\sin (k_{s} \mathbf{d}\cdot \mathbf{x})g_{js}^\sfr ({\mathbf d})-\sin (k_{s} \mathbf{d}\cdot \mathbf{x})g_{js}^\sfi ({\mathbf d}) \bigg )\mathbf{d}^{\perp}
{\rm d}\sigma(\mathbf{d})
\end{align}
and
\begin{equation}\notag
\begin{aligned}
\mathbf v_{j} ^\sfi(\mathbf x)=& \frac{1}{\sqrt 2}\int_{\mathbb S^1}\sqrt{\frac{k_{p}}{\omega }}\bigg ( -\cos (k_{p} \mathbf{d}\cdot \mathbf{x})g_{jp}^\sfr ({\mathbf d})+ \cos (k_{p} \mathbf{d}\cdot \mathbf{x})g_{jp}^\sfi ({\mathbf d})  +\sin (k_{p} \mathbf{d}\cdot \mathbf{x})g_{jp}^\sfr ({\mathbf d})\\&+\sin (k_{p} \mathbf{d}\cdot \mathbf{x})g_{jp}^\sfi ({\mathbf d}) \bigg )\mathbf{d}+
\sqrt{\frac{k_{s}}{\omega }}\bigg ( -\cos (k_{s} \mathbf{d}\cdot \mathbf{x})g_{js}^\sfr ({\mathbf d})+ \cos (k_{s} \mathbf{d}\cdot \mathbf{x})g_{js}^\sfi ({\mathbf d}) \\& +\sin (k_{s} \mathbf{d}\cdot \mathbf{x})g_{js}^\sfr ({\mathbf d})+\sin (k_{s} \mathbf{d}\cdot \mathbf{x})g_{js}^\sfi ({\mathbf d}) \bigg )\mathbf{d}^{\perp}
{\rm d}\sigma(\mathbf{d}).
\end{aligned}
\end{equation}
\end{prop}

\begin{proof}
The proof follows from using Euler's formula as well as straightforward (though a bit tedious) calculations. We skip the details. 
\end{proof}

\begin{prop}\label{Pro:2.2}
	Let the elastic Herglotz wave function $\mathbf v_{j}$ be defined by \eqref{eq:vj}. Denote
	\begin{equation}
\begin{aligned}
 & \mathbf v_{jp} ^\sfr(\mathbf 0)= \frac{1}{\sqrt 2}\int_{\mathbb S^1}\sqrt{\frac{k_{p}}{\omega }}g_{jp}^\sfr ({\mathbf d})\mathbf{d}{\rm d}\sigma(\mathbf{d}),\quad \mathbf v_{js} ^\sfr(\mathbf 0)=\frac{1}{\sqrt 2}\int_{\mathbb S^1}\sqrt{\frac{k_{s}}{\omega }}g_{js}^\sfr ({\mathbf d})\mathbf{d}^{\perp}{\rm d}\sigma(\mathbf{d}),\\&  \mathbf v_{jp} ^\sfi(\mathbf 0)= \frac{1}{\sqrt 2}\int_{\mathbb S^1}\sqrt{\frac{k_{p}}{\omega }}g_{jp}^\sfi ({\mathbf d})\mathbf{d}{\rm d}\sigma(\mathbf{d}),\quad   \mathbf v_{js} ^\sfi(\mathbf 0)=\frac{1}{\sqrt 2}\int_{\mathbb S^1}\sqrt{\frac{k_{s}}{\omega }}g_{js}^\sfi ({\mathbf d})\mathbf{d}^{\perp}{\rm d}\sigma(\mathbf{d}).
\end{aligned}
\end{equation}
Then
\begin{equation}\label{eq:vj0}
 \begin{aligned}
	\mathbf v_{j} (\mathbf 0)=
\mathbf v_{j} ^\sfr(\mathbf 0)+{\rm i}\mathbf v_{j}^\sfi (\mathbf 0)&: = (\mathbf v_{jp} ^\sfr(\mathbf 0)+\mathbf v_{js} ^\sfr(\mathbf 0)+ \mathbf v_{jp}^\sfi (\mathbf 0)+\mathbf v_{js}^\sfi (\mathbf 0)) \\
&\quad +{\rm i}(\mathbf v_{jp}^\sfi (\mathbf 0)+\mathbf v_{js}^\sfi (\mathbf 0) -\mathbf v_{jp} ^\sfr(\mathbf 0)-\mathbf v_{js} ^\sfr(\mathbf 0)),
  \end{aligned}
\end{equation}	
where
 $g_{j\beta }^{\sfr }({\mathbf d})$ and $ g_{j\beta }^{\sfi }({\mathbf d})$ ($\beta=p,s$) are defined in \eqref{eq:gjpgjs}. Let $J_\ell(t)$ be the $\ell$-th Bessel function of the first kind for $\ell \in \mathbb N \cup \{0\}$.
 Denote $J_{\ell,\beta}= J_{\ell}(k_\beta  |\mathbf{x}|) $, $\beta=p,\ s$. Furthermore, we have
\begin{equation}\label{eq:vjr}
 \begin{aligned}
  \mathbf v_{j} ^\sfr(\mathbf x)=&
 \mathbf v_{jp} ^\sfr(\mathbf 0)J_{0,p}
 +\mathbf v_{jp} ^\sfi(\mathbf 0)J_{0,p}
 +\mathbf v_{js} ^\sfr(\mathbf 0)J_{0,s}
 +\mathbf v_{js} ^\sfi(\mathbf 0)J_{0,s}
\\& +\sqrt{2}\sum^{+\infty}_{\ell =1}  (-1)^\ell \Big(   J_{2\ell,p}\mathbf{A}_{jp,1}^{(\ell )}+ J_{2\ell,s}\mathbf{A}_{js,1}^{(\ell)}
  +   J_{2\ell-1,p}\mathbf{A}_{jp,2}^{(\ell)}
+ J_{2\ell-1,s}\mathbf{A}_{js,2}^{(\ell)}\Big),
\end{aligned}
\end{equation}
and
\begin{equation}\label{eq:vji}
 \begin{aligned}
 \mathbf v_{j} ^\sfi(\mathbf x)=&  \mathbf v_{jp} ^\sfi(\mathbf 0)J_{0,p}+\mathbf v_{jp} ^\sfi(\mathbf 0)J_{0,p}
 -\mathbf v_{js} ^\sfr(\mathbf 0)J_{0,s}
 -\mathbf v_{js} ^\sfr(\mathbf 0)J_{0,s}
\\&  +\sqrt{2}\sum^{+\infty}_{\ell =1}  [(-1)^\ell \Big(  J_{2\ell,p}\mathbf{A}_{jp,3}^{(\ell)}+
J_{2\ell,s}\mathbf{A}_{js,3}^{(\ell)}
  + J_{2\ell-1,p}\mathbf{A}_{jp,2}^{(\ell )}
+ J_{2\ell-1,s}\mathbf{A}_{js,2}^{(\ell )}\Big),
\end{aligned}
\end{equation}
where $\mathbf{p}_p=\mathbf d$, $\mathbf{p}_s=\mathbf d^\perp$, and 
\begin{equation} \notag
 \begin{aligned}
\mathbf{A}_{j\beta ,1}^{(\ell) }=&\int_{\mathbb S^1}\sqrt{\frac{k_{\beta }}{\omega }}g_{j\beta }^\sfr({\mathbf d})\cos ( 2\ell\theta)\mathbf{p}_\beta {\rm d}\sigma(\mathbf{d})
+\int_{\mathbb S^1}\sqrt{\frac{k_{\beta }}{\omega }}g_{j\beta }^\sfi({\mathbf d})\cos ( 2\ell\theta))\mathbf{p}_\beta {\rm d}\sigma(\mathbf{d}),\\
 \mathbf{A}_{j\beta,2}^{(\ell)}=&
\int_{\mathbb S^1}\sqrt{\frac{k_{\beta }}{\omega }}g_{j\beta }^\sfi({\mathbf d})\cos ((2\ell-1)\theta)\mathbf{p}_\beta {\rm d}\sigma(\mathbf{d})
- \int_{\mathbb S^1}\sqrt{\frac{k_{\beta }}{\omega }}g_{j\beta }^\sfr({\mathbf d})\cos ((2\ell-1)\theta)\mathbf{p}_\beta {\rm d}\sigma(\mathbf{d}),\\
\mathbf{A}_{j\beta ,3}^{(\ell)}=&\int_{\mathbb S^1}\sqrt{\frac{k_{\beta }}{\omega }}g_{j\beta }^\sfr({\mathbf d})\cos ( 2\ell\theta)\mathbf{p}_\beta {\rm d}\sigma(\mathbf{d})
-\int_{\mathbb S^1}\sqrt{\frac{k_{\beta }}{\omega }}g_{j\beta }^\sfi({\mathbf d})\cos ( 2\ell\theta))\mathbf{p}_\beta {\rm d}\sigma(\mathbf{d}), \beta=p  \mbox{ or } s, 
\end{aligned}
\end{equation}
and $\theta $ is the angle between $\mathbf x$ and $ \mathbf d$ in \eqref{eq:vj}. It holds that
\begin{equation}\label{eq:gjps norm}
	\begin{split}
		|\mathbf{A}_{jp,i}^{(\ell)}| \leq 2\sqrt{\frac{k_{p} \pi }{\omega }} \|g_{jp}\|_{L^2(\mathbb S^1 )},\quad |\mathbf{A}_{js,i}^{(\ell)}| \leq 2\sqrt{\frac{k_{s} \pi }{\omega }} \|g_{js}\|_{L^2(\mathbb S^1 )},\quad i=1,\ldots, 3.
	\end{split}
\end{equation}
\end{prop}

\begin{proof}
	In view of \eqref{eq:vj},  we can directly derive \eqref{eq:vj0}. Using the Jacobi-Anger expansion (cf. \cite{CK}), we have
	\begin{equation}\label{eq:jae}
		\begin{split}
			e^{{\rm i}k_\beta \mathbf{d}\cdot \mathbf{x}}&=J_{0}(k_\beta |\mathbf{x}|)+2\sum^{+\infty}_{n=1}{\rm i}^{n}J_{n}(k_\beta |\mathbf{x}|)\cos n\theta, \\
			&=J_{0}(k_\beta |\mathbf{x}|)+2\sum^{+\infty}_{\ell =1} \Big [(-1)^\ell  J_{2\ell}(k_\beta |\mathbf{x}|)\cos ( 2\ell\theta)  \\
			&\quad -{\rm i}  (-1)^\ell  J_{2\ell-1}(k_\beta |\mathbf{x}|)\cos ((2\ell-1)\theta)\Big],
		\end{split}
	\end{equation}
	where $\, \theta =\angle (\mathbf x, \mathbf d),\, \beta=p,s$.
Substituting \eqref{eq:vj0} and \eqref{eq:jae} into \eqref{eq:vj}, we can obtain \eqref{eq:vjr} and \eqref{eq:vji} by direct calculations. \eqref{eq:gjps norm} can be obtained by using the Cauchy-Schwarz inequality.
\end{proof}


We shall make use of the complex geometrical optics solution (CGO) $\mathbf u (s \mathbf x)$ introduced in \cite{EBL}, where $s\in \mathbb R_+$ is an asymptotic parameter. We next review some quantitative properties of $\mathbf u (s \mathbf x)$, which shall be used  in our subsequent analysis.

\begin{lem} \cite[Proposition 3.1]{EBL} \label{lem:22}
Let $\Omega\subset\mathbb{R}^2$ such that $\Omega\cap(\mathbb{R}_{-}\cup\{\mathbf 0
\})=\emptyset$. Denote
\begin{equation}\label{eq:lame3}
\mathbf{u}( \mathbf{x} )=
\begin{pmatrix}
    \exp(-s\sqrt{z})\\
    {\rm i}\exp(-s\sqrt{z})
    \end{pmatrix}:=\begin{pmatrix}
    u_1( \mathbf{x})\\
    u_2( \mathbf{x})
\end{pmatrix} , \quad \mathbf{x}=(x_1,x_2)^\top,
\end{equation}
where $z=x_{1}+{\rm i}x_{2}$ and $s\in \mathbb{R}_{+}$. The complex square root is defined as
\begin{equation}\notag 
\sqrt{z}=\sqrt{|z|}\left(\cos\frac{\theta}{2}+{\rm i}\sin\frac{\theta}{2}\right ),
\end{equation}
where $-\pi<\theta\leq\pi$ is the argument of $z$. Then ${\mathbf u}$ satisfies
$$
\mathcal{L}\mathbf{u}={\mathbf 0}\ \ \mbox{ in } \ \ \Omega.
$$
Let the open sector $W$ be defined in  \eqref{eq:lame2}.  Then
\begin{equation}\label{eq:lame5}
\int_{W} u_1(\mathbf{x}){\rm d}\mathbf{x}=6{\rm i}(e^{-2\theta_{M}{\rm i}}-e^{-2\theta_{m} \mathrm i})s^{-4}.
\end{equation}
In addition for $\alpha,h>0$ and $j\in\{1,2\}$ we have the upper bounds
\begin{equation}\label{eq:lame6}
\int_{W}|u_{j}(\mathbf{x})||\mathbf{x}|^{\alpha}{\rm d}\mathbf{x}\leq\frac{2(\theta_{M}-\theta_{m})\Gamma(2\alpha+4)}{\delta_{W}^{2\alpha+4}}s^{-2\alpha-4},
\end{equation}
and
\begin{equation}\label{eq:lame7}
\int_{W\backslash S_{h}}|u_{j}(\mathbf{x})|{\rm d}\mathbf{x}\leq\frac{6(\theta_{M}-\theta_{m})}{\delta_{W}^{4}}s^{-4}e^{-\delta_{W}s\sqrt{h}/2},
\end{equation}
where $\delta_{W}=\min_{\theta_{m}<\theta<\theta_{M}} \cos \left(\frac{\theta}{2}\right)$ is a positive constant.
\end{lem}

 The following lemma states the the regularity of the CGO solution $\mathbf u(\mathbf x)$ defined in \eqref{eq:lame3}. 
\begin{lem}\label{lem:23}
	Let $S_h$ be defined in \eqref{eq:SIGN} and $\mathbf u(\mathbf x)$ be given in \eqref{eq:lame3}. Then $\mathbf u(\mathbf x) \in H^1(S_h)^2$ and $\mathcal L  \mathbf u (\mathbf x)=\mathbf 0$ in $S_h$.  Furthermore, it holds that
	\begin{equation} \label{eq:lame21}
	\|\mathbf u(\mathbf x)\|_{L^2(S_h)^2 }^2\leq 	 (\theta_M-\theta_m) e^{- 2s\sqrt{\Theta } \delta_W } h^2,
	\end{equation}
	and
	\begin{equation}\label{eq:lame34}
		\left  \||\mathbf x|^\alpha \mathbf  u( \mathbf x) \right \|_{L^{2}(S_h )^2  }^2 \leq s^{-4(\alpha+1 )} \frac{4(\theta_M-\theta_m)  }{(4\delta_W^2)^{2\alpha+2  } } \Gamma(4\alpha+4),
	\end{equation}
where $ \Theta  \in [0,h ]$ and $\delta_W$ is defined in \eqref{eq:lame7}.
\end{lem}

\begin{proof}
The proof follows from a similar argument to that of \cite[Lemma 2.3]{DCL} and we skip the details. 
\end{proof}

\begin{lem}\label{lem:24 u tu}
	Suppose that $\Lambda_h$ and $\mathbf u(\mathbf x)$  are defined  by  \eqref{eq:SIGN} and  \eqref{eq:lame3} respectively. Recall that $\delta_W >0$ is given in \eqref{eq:lame7}. We have
	\begin{subequations}
		\begin{align}
			\|\mathbf u( \mathbf x)\|_{H^1(\Lambda_h )^{2}}& \leq \sqrt{ h+ \frac{s^2}{2}}\sqrt {\theta_M-\theta_m }e^{-s \sqrt h \delta_W} , \label{eq:u H1}\\
			\| T_{\nu}(\mathbf{u}) \|_{L^2(\Lambda_h )^{2} } &\leq \frac{s\mu}{\sqrt 2}  \sqrt{\theta_M-\theta_m } e^{-s \sqrt h \delta_W}, \label{eq:Tu L2}
		\end{align}
	\end{subequations}
	both of which decay exponentially as $s\rightarrow +\infty. $
\end{lem}

\begin{proof}
By \eqref{eq:lame3}, one has
\begin{equation}\label{eq:21 u1norm}
	\| u_1( \mathbf x)\|_{L^2(\Lambda_h )^{2}} \leq \sqrt{h}e^{-s \sqrt h \delta_W} \sqrt{\theta_M-\theta_m}.
\end{equation}
It is directly verified that
	\begin{equation}\notag
	\begin{split}
		\frac{\partial u_1( \mathbf x)}{\partial r}&=-\frac{ s }{2r^{1/2}} e^{-s\sqrt{r} (\cos( \theta/2)+\mathrm i \sin (\theta/2 ) )+\mathrm i \theta/2 },\\ \frac{\partial u_1(\mathbf  x)}{\partial \theta}&=- \frac{\mathrm i s \sqrt{r}}{2} e^{-s\sqrt{r} (\cos( \theta/2)+\mathrm i \sin (\theta/2 ) )+\mathrm i \theta/2 },
	\end{split}
	\end{equation}
	which can be used to obtain that
		\begin{equation}\label{eq:23 partial}
	\begin{split}
		\frac{\partial u_1( \mathbf x)}{\partial x_1}&=-\frac{ s }{2r^{1/2}} e^{-s\sqrt{r} (\cos( \theta/2)+\mathrm i \sin (\theta/2 ) )-\mathrm i \theta/2 },\\ \frac{\partial u_1(\mathbf x)}{\partial x_2}&=- \frac{i s}{2r^{1/2}} e^{-s\sqrt{r} (\cos( \theta/2)+\mathrm i \sin (\theta/2 ) )-\mathrm i \theta/2 }.
	\end{split}
	\end{equation}
Therefore
\begin{equation}\label{eq:22 norm}
\left	\| \nabla u_1( \mathbf x ) \right \|_{L^2(\Lambda_h )^{2}} \leq \frac{s}{2} e^{-s \sqrt h \delta_W }\sqrt{\theta_M-\theta_m}.
\end{equation}
Combining \eqref{eq:21 u1norm} and \eqref{eq:22 norm}, as well as noting $u_2(\mathbf x )=\mathrm i u_2(\mathbf x )$, we can prove \eqref{eq:u H1}.

Using \eqref{eq:Tu def}, it is directly calculated that
$
T_\nu \mathbf u=\mu \nabla \mathbf u \ \nu. 
$
 Therefore by virtue of \eqref{eq:23 partial}, \eqref{eq:22 norm}  and  Cauchy-Schwarz inequality, we have
\begin{equation}\notag
	\begin{split}
		\| T_{\nu}(\mathbf{u}) \|_{L^2(\Lambda_h ) ^{2}}^2& =2\mu^2 \left	\| \nabla u_1( \mathbf x ) \right \|_{L^2(\Lambda_h )^{2}}^2 \leq \frac{s\mu^2}{2} e^{-2s \sqrt h \delta_W} (\theta_M-\theta_m).
	\end{split}
\end{equation}

The proof is complete. 
\end{proof}

We proceed to derive several key lemmas in order to establish the main geometric result of this section. It is first recalled the following Green formula for the Lam\'e operator; see \cite[Lemm 3.4]{costabel88} and \cite[Theorem 4.4]{McLean}.

\begin{lem}\label{lem:25}
Suppose that $\Omega\Subset \mathbb R^n (n=2,3) $ is a bounded Lipschitz domain. Let $\mathbf{u_1} \in H^1(\Omega )^2 $ and $\mathbf v_1 \in H^1(\Omega )^2$ satisfying $\mathcal L \mathbf{u_1}  \in L^2(\Omega )^n$ and $\mathcal L \mathbf{v_1}  \in L^2(\Omega )^n$. The the following Green identity holds
\begin{equation}\label{eq:green1}
	\int_{\Omega}(  \mathcal{L}\mathbf{u}_{1} \cdot \mathbf{v}_1 -  \mathcal{L}\mathbf{v}_{1} \cdot\mathbf{u}_1  ){\rm d}\mathbf{x}=\int_{\partial\Omega}( T_{\mathbf{\nu}}\mathbf{u}_1 \cdot \mathbf{v}_{1}-T_{\mathbf{\nu}}\mathbf{v}_{1} \cdot \mathbf{u}_1) {\rm d}\sigma.
\end{equation}
\end{lem}

Recall the splitting \eqref{eq:decouple}. In what follows, we shall mainly focus on establishing the relevant results for $(\mathbf{v}_{\sf R}, \mathbf{w}_{\sf R})$. Due to the symmetric role of $(\mathbf{v}_{\sf R}, \mathbf{w}_{\sf R})$ and $(\mathbf{v}_{\sf I}, \mathbf{w}_{\sf I})$, those results hold equally for $(\mathbf{v}_{\sf I}, \mathbf{w}_{\sf I})$, and hence $(\mathbf{v},\mathbf{w})$.

\begin{lem}\label{lem:26}
 Let $\mathbf{v}_\sfr\in H^{1}(\Omega)^2$ and $\mathbf{w}_\sfr\in H^{1}(\Omega)^2$ be a pair of generalized elastic transmission eigenfunctions to \eqref{eq:lame1}. Let the CGO solution $\mathbf{u}$ and the elastic Herglotz wave function $\mathbf{v}_{j}^\sfr$ be defined in \eqref{eq:lame3}  and  \eqref{eq:vjr} respectively. Assume that the Lipschitz domain $\Omega\subset\mathbb{R}^2$ contains a corner $ S_h \Subset \Omega \cap W$, where $S_h$ is defined in \eqref{eq:SIGN} and $W$ is a sector defined in \eqref{eq:lame2}. Denote $q=1+V$, where $V$ is defined in \eqref{eq:lame1}.  Then the following integral equality holds
\begin{equation}\label{eq:int1}
	I_1+I_2=I_{\Lambda_h }- I_{\pm} -I_{\pm}^\Delta,
\end{equation}
where
\begin{align*}
&I_1=\omega^{2}\int_{S_{h}}(q\mathbf{w}_\sfr -\mathbf{v}_j^\sfr)\cdot\mathbf{u}( \mathbf{x}){\rm d}\mathbf{x},\quad I_2= -\omega^{2}\int_{S_{h}}(\mathbf{v}_\sfr-\mathbf{v}_{j}^\sfr)\cdot\mathbf{u}( \mathbf{x}){\rm d}\mathbf{x},  \\
&I_{\Lambda_h }=\int_{\Lambda_{h}}(T_{\nu}(\mathbf{v}_\sfr-\mathbf{w}_\sfr))\cdot\mathbf{u}-(T_{\nu}(\mathbf{u}))\cdot(\mathbf{v}_\sfr-\mathbf{w}_\sfr){\rm d}\sigma,\\
&	I_{\pm } = \int_{\Gamma_{h}^{\pm}}\eta \mathbf{u}\cdot\mathbf{v}_{j}^\sfr{\rm d}\sigma, \quad I_{\pm }^\Delta =\int_{\Gamma_{h}^{\pm}}\eta \mathbf{u}\cdot(\mathbf{v}_\sfr-\mathbf{v}_{j}^\sfr){\rm d}\sigma.
\end{align*}
 Here $\Lambda_h$ and $\Gamma_h^\pm$ are defined in \eqref{eq:SIGN}.
\end{lem}

\begin{proof}

Recall that the differential operator $\mathcal{L}$ is defined in \eqref{eq:navier}. In view of the first and second equation in \eqref{eq:lame1}, we have
\begin{equation}\label{eq:lame11}
\mathcal{L}\mathbf{v}_\sfr=-\omega^{2}\mathbf{v}_\sfr,\quad
\mathcal{L}\mathbf{u}_\sfr=-\omega^{2}q\mathbf{w}_\sfr \quad \mbox{ in } \ \ S_h.
\end{equation}
Using the boundary condition in \eqref{eq:lame1}, it yields that
\begin{equation}\label{eq:lame13}
\mathbf{v}_\sfr-\mathbf{w}_\sfr=0,\quad
T_{\nu}(\mathbf{v}_\sfr-\mathbf{w}_\sfr)=-\eta\mathbf{v}_\sfr \quad \mbox{ on }  \ \ \Gamma_{h}^{\pm}.
\end{equation}
Using Green's formula \eqref{eq:green1} on the domain $S_{h}$ 
together with ${\mathcal L} \mathbf{u}=\mathbf{0}$ in $S_h$, we have
\begin{equation}\label{eq:lame25}
\int_{S_{h}}(\mathcal{L}(\mathbf{v}_\sfr-\mathbf{w}_\sfr))\cdot\mathbf{u}( \mathbf{x}){\rm d}\mathbf{x}=\int_{\Gamma_{h}^{\pm}\cup\Lambda_{h}}(T_{\nu}(\mathbf{v}_\sfr-\mathbf{w}_\sfr))\cdot\mathbf{u}-(T_{\nu}(\mathbf{u}))\cdot(\mathbf{v}_\sfr-\mathbf{w}_\sfr){\rm d}\sigma. 
\end{equation}
By virtue of  \eqref{eq:lame13}, we have
\begin{equation}\label{eq:lame27}
 \begin{aligned}
&\int_{\Gamma_{h}^{\pm}}(T_{\nu}(\mathbf{v}_\sfr-\mathbf{w}_\sfr))\cdot\mathbf{u}-(T_{\nu}(\mathbf{u}))\cdot(\mathbf{v}_\sfr-\mathbf{w}_\sfr){\rm d}\sigma\\=&\int_{\Gamma_{h
}^{\pm}}-\eta \mathbf{u}\cdot\mathbf{v}_\sfr{\rm d}\sigma=  \int_{\Gamma_{h}^{\pm}}-\eta \mathbf{u}\cdot\mathbf{v}_{j}^\sfr{\rm d}\sigma+\int_{\Gamma_{h}^{\pm}}-\eta \mathbf{u}\cdot(\mathbf{v}_\sfr-\mathbf{v}_{j}^\sfr){\rm d}\sigma.
 \end{aligned}
\end{equation}
From \eqref{eq:lame11}, we have
\begin{equation}\label{eq:lame14}
 \begin{aligned}
&\int_{S_{h}}(\mathcal{L}(\mathbf{v}_\sfr-\mathbf{w}_\sfr))\cdot\mathbf{u}(\mathbf{x}){\rm d}\mathbf{x} =
\int_{S_{h}}(-\omega^{2}\mathbf{v}_\sfr+\omega^{2}q\mathbf{w}_\sfr)\cdot\mathbf{u}(\mathbf{x}){\rm d}\mathbf{x}
 \\=&\int_{S_{h}}(-\omega^{2}\mathbf{v}_{j}^\sfr+\omega^{2}q\mathbf{w}_\sfr)\cdot\mathbf{u}(\mathbf{x}){\rm d}\mathbf{x}
 +\int_{S_{h}}-\omega^{2}(\mathbf{v}_\sfr-\mathbf{v}_{j}^\sfr)\cdot\mathbf{u}(\mathbf{x}){\rm d}\mathbf{x}.
 \end{aligned}
\end{equation}
By \eqref{eq:lame25}, \eqref{eq:lame27} and \eqref{eq:lame14}, we can derive \eqref{eq:int1}.

The proof is complete. 
\end{proof}

\begin{lem}\label{lem:2.6}
Let $ I_{\Lambda_{h}}$ be defined in \eqref{eq:int1}. Under the same setup as that in Lemma \ref{lem:26}, we have the following estimate
\begin{equation}\label{e}
| I_{\Lambda_{h}}|\leq C\frac{\sqrt{ 2h+ s^2}+\mu s}{\sqrt 2}\sqrt {\theta_M-\theta_m }e^{-s \sqrt h \delta_W}\| \mathbf{v}_\sfr-\mathbf{w}_\sfr \|_{H^1(S_{h } )^{2}  },
\end{equation}
where $C$ is a positive constant coming from the trace  theorem, $S_{h }$ and $\delta_W>0$ are defined in \eqref{eq:SIGN} and  \eqref{eq:lame7}, respectively.
\end{lem}
\begin{proof}
By using the H{\"o}lder inequality, Lemma \ref{lem:24 u tu}, and the trace theorem, one has
\begin{align}
|I_{\Lambda_h }| &\leq  \|(T_{\nu}(\mathbf{v}_\sfr-\mathbf{w}_\sfr))\|_{H^{-1/2} ( \Lambda_{h} ) ^{2}} \|\mathbf{u} \|_{ H^{1/2}(\Lambda_{h}) ^{2}}+ \| T_{\nu}(\mathbf{u}) \|_{L^2(\Lambda_h ) ^{2}} \| \mathbf{v}_\sfr-\mathbf{w}_\sfr \|_{L^2(\Lambda_h ) ^{2}} \notag \\
&\leq \left(  \|\mathbf{u} \|_{ H^{1}(\Lambda_{h}) ^{2}}  +  \| T_{\nu}(\mathbf{u}) \|_{L^2(\Lambda_h ) ^{2}}\right) \| \mathbf{v}_\sfr-\mathbf{w}_\sfr \|_{H^1(S_{h } ) ^{2} } .\label{eq:233 bound}
\end{align}
Substituting \eqref{eq:u H1} and \eqref{eq:Tu L2} into \eqref{eq:233 bound}, one can obtain \eqref{e}.
\end{proof}

\begin{lem}\label{lem:27}
Under the same setup as that in Lemma~\ref{lem:26}, we further suppose that the boundary  parameter $\eta$ of \eqref{eq:lame1} satisfies  $\eta\in C^{\alpha}(\overline{\Gamma_{h}^{\pm}})$ for $0<\alpha<1$. For any given constants $\beta_1$, $\beta_2$ and $\gamma$ satisfying $\gamma>\max\{\beta_{1}, \beta_{2}\}>0$, assume that there exits a sequence of  $ \{\mathbf{v}_{j}^{\sfr}\}_{j=1}^{+\infty}$ defined by \eqref{eq:vrj 28} with kernels $g_{jp}$ and $g_{js}$ can approximate $\mathbf{v}_{\sfr}$ in $ H^{1}(S_{h})$ fulfilling
\begin{equation}\label{eq:cond thm3}
\| \mathbf{v}_{\sfr}-\mathbf{v}_{j}^\sfr\|_{H^{1}(S_{h})^{2}}\leq j^{-\gamma }, \ \| g_{jp}\|_{L^{2}(\mathbb S^{1})}\leq j^{\beta_{1}}, \ \| g_{js}\|_{L^{2}(\mathbb S^{1})}\leq j^{\beta_{2}}. 
\end{equation}
 Recall that $I_{2}$ and $I_{\pm}^\Delta$ are defined in \eqref{eq:int1}. Then the following integral estimates hold:
\begin{equation}\label{eq:lame22}
 \begin{aligned}
\left| I_{2}\right|&\leq\omega^{2}h\sqrt{\theta_{M}-\theta_{m}}e^{-s\sqrt{\Theta}\delta_{W}}j^{-\gamma },
 \end{aligned}
\end{equation}
and
\begin{equation}\label{eq:lame35}
\begin{aligned}
\left| I_{\pm}^\Delta\right|\leq &\Big (|\eta(\mathbf{0})| h\sqrt{\theta_{M}-\theta_{m}}e^{-s\sqrt{\Theta}\delta_{W}}\\&+\|\eta\|_{C^{\alpha}}\frac{2\sqrt{\theta_{M}-\theta_{m}\Gamma(4\alpha+4)}}{(2\delta_{W})^{2\alpha+2}}s^{-2(\alpha+1) }\Big)j^{-\gamma },
\end{aligned}
\end{equation}
where $\Theta\in[0,h]$, $\delta_{W}$ is defined in \eqref{eq:lame7}, $\theta_{m}$ and  $\theta_{M}$  are defined in \eqref{eq:lame1}.
\end{lem}
\begin{proof}
By using the Cauchy-Schwarz inequality, we have
\begin{equation}\label{eq:lame16}
\left| I_{2}\right|\leq\omega^{2}\| \mathbf{v}_\sfr-\mathbf{v}_{j}^\sfr\|_{L^{2}(S_{h})^{2}}\| \mathbf{u}(\mathbf{x})\|_{L^{2}(S_{h})^{2}}\leq\omega^{2} \|\mathbf{u}(\mathbf{x})\|_{L^{2}(S_{h})^{2}}j^{-\gamma }
.
\end{equation}
In view of \eqref{eq:lame21}, we can immediately obtain  \eqref{eq:lame22}.

Since $\eta\in C^{\alpha}(\overline{\Gamma}_{h}^{\pm})$, we have the following expansion of $\eta(\mathbf{x})$ at the origin as
\begin{equation}\label{eq:lame29}
\eta(\mathbf{x})=\eta(\mathbf{0})+\delta\eta(\mathbf{x}), \mid \delta\eta(\mathbf{x})\mid\leq \|\eta\|_{C^{\alpha}} | \mathbf{x}|^{\alpha}.
\end{equation}
By using the Cauchy-Schwarz inequality and the trace theorem, we have
 \begin{align}
  \left| I_{\pm}^\Delta\right|\leq&|\eta(\mathbf{0})|\int_{\Gamma_{h}^{\pm}}| \mathbf{u}|| (\mathbf{v}_\sfr-\mathbf{v}_{j}^\sfr)| {\rm d}\sigma+\| \eta\|_{C^{\alpha}}\int_{\Gamma_{h}^{\pm}}| \mathbf{x}|^{\alpha}|\mathbf{u}|| (\mathbf{v}_\sfr-\mathbf{v}_{j}^\sfr)| {\rm d}\sigma  \notag
  \\
  \leq&|\eta(\mathbf{0})|\| \mathbf{v}_\sfr-\mathbf{v}_{j}^\sfr\|_{H^{\frac{1}{2}}(\Gamma_{h}^{\pm})^{2}}\| \mathbf{u}\|_{H^{-\frac{1}{2}}(\Gamma_{h}^{\pm})^{2}}+\| \eta\|_{C^{\alpha}}\| \mathbf{v}_\sfr-\mathbf{v}_{j}^\sfr\|_{H^{\frac{1}{2}}(\Gamma_{h}^{\pm})^{2}}\| | \mathbf{x}|^{\alpha} \mathbf{u}\|_{H^{-\frac{1}{2}}(\Gamma_{h}^{\pm})^{2}} \notag
 \\
 \leq&|\eta(\mathbf{0})|\| \mathbf{v}-\mathbf{v}_{j}^\sfr\|_{H^{1}(S_{h})^{2}}\| \mathbf{u}\|_{L^{2}(S_{h})^{2}}+\| \eta\|_{C^{\alpha}}\| \mathbf{v}_\sfr-\mathbf{v}_{j}^\sfr\|_{H^{1}(S_{h})^{2}}\|| \mathbf{x}|^{\alpha} \mathbf{u}\|_{L^{2}(S_{h})^{2}} \notag
  \\
  \leq&\bigg(|\eta(\mathbf{0})|\| \mathbf{u}\|_{L^{2}(S_{h})^{2}}+\| \eta\|_{C^{\alpha}}\|| \mathbf{x}|^{\alpha} \mathbf{u}\|_{L^{2}(S_{h})^{2}}\bigg)j^{-\gamma }
 . \label{eq:lame30}
 \end{align}
By \eqref{eq:lame21}, \eqref{eq:lame34} and \eqref{eq:lame30},  it readily yields \eqref{eq:lame35}.
\end{proof}

\begin{lem}\label{lem:28}
Under the same setup as that in Lemma \ref{lem:27}, we further suppose that $q\mathbf{w}_\sfr  \in C^\alpha ( S_h )^2$ ($0<\alpha<1$) and hence
 \begin{equation}\label{eq:fr}
 	\mathbf {f}_\sfr (\mathbf{x}):=q\mathbf{w}_\sfr(\mathbf{x})=\mathbf f_\sfr (\mathbf{0}) +\delta \mathbf f_\sfr(\mathbf{x}),\quad |\delta \mathbf f_\sfr|\leq \| \mathbf{f}_\sfr (\mathbf{x})\|_{C^{\alpha}(S_h)^{2}}| \mathbf{x}|^{\alpha}. 
 \end{equation}
Then the following integral estimate holds
\begin{equation}\label{eq:lame42} 
 \begin{aligned}
\left|  I_{1}\right| \leq& \omega^{2}\Big ( 4
\bigg(
 \sqrt{\frac{\pi k_{p}}{\omega}}(1+k_{p})\|g_{jp}\|_{L^2(S_h)}+
\sqrt{\frac{\pi k_{s}}{\omega}}(1+k_{s})\|g_{js}\|_{L^2(S_h)}
 \bigg) \\&\times\frac{(\theta_{M}-\theta_{m})\Gamma(2\alpha+4)}{\delta_{W}^{2\alpha+4}}s^{-2\alpha-4}
\\&+\|\mathbf{f}_\sfr (\mathbf{x})\|_{C^{\alpha}(\Omega)^{2}} \frac{2\sqrt{2}(\theta_{M}-\theta_{m})\Gamma(4\alpha+4)}{\delta_{W}^{2\alpha+4}}s^{-2\alpha-4}
+6\sqrt{2}( |\mathbf{f}_\sfr(\mathbf{0})|+|\mathbf{v}_{j}^\sfr(\mathbf{0}|)
\\
&\times  | e^{-2\theta_{M}\mathrm i}-e^{-2\theta_{m}\mathrm i}|s^{-4} \Big ),
 \end{aligned}
\end{equation}
where $\delta_{W}$ is defined in \eqref{eq:lame7}, $\theta_{m}$ and  $\theta_{M}$  are defined in \eqref{eq:lame1}.
\end{lem}

\begin{proof}
Since $\mathbf{v}_{j}^\sfr  \in C^\alpha(S_h)^2$, $\alpha\in (0, 1)$, we  have the following splitting
\begin{equation}\label{eq:lame37}
 \begin{aligned}
&\mathbf{v}_{j}^\sfr(\mathbf{x})=\mathbf{v}_{j}^\sfr(\mathbf{0})+\delta\mathbf{v}_{j}^\sfr(\mathbf{x}),\quad |\delta\mathbf{v}_{j}^\sfr|\leq \|\mathbf{v}_{j}^\sfr\|_{C^{\alpha}(\Omega)^{2}}| \mathbf{x}|^{\alpha}.
 \end{aligned}
\end{equation}
Therefore, it holds that
\begin{equation}\label{eq:lame38}
 \begin{aligned}
\int_{S_{h}}(q\mathbf{w}_\sfr - \mathbf{v}_{j}^\sfr )\cdot\mathbf{u}(\mathbf{x}){\rm d}\mathbf{x}
&=
 \int_{S_{h}}( \mathbf{f}_\sfr(\mathbf{0})-\mathbf{v}_{j}^\sfr(\mathbf{0}))  \cdot \mathbf{u}(\mathbf{x}){\rm d}\mathbf{x}
\\
&-\int_{S_{h}}\delta \mathbf{v}_{j}^\sfr(\mathbf{x})\cdot\mathbf{u}(\mathbf{x}){\rm d}\mathbf{x}
+\int_{S_{h}}\delta \mathbf{f}_\sfr (\mathbf{x})\cdot\mathbf{u}(\mathbf{x}){\rm d}\mathbf{x}.
 \end{aligned}
\end{equation}
Using the compact embedding of H{\"o}lder spaces, one can obtain that
$$
\|\mathbf{v}_{j}^\sfr \|_{C^\alpha (S_h)} \leq {\rm diam}(S_h)^{1-\alpha} \|\mathbf{v}_{j}^\sfr\|_{C^1(S_h)},
$$
where $ {\rm diam}(S_h) $ is the diameter of $S_h$. By direct computations, we have
\begin{equation}\label{eq:vC1}
 \begin{aligned}
\|\mathbf{v}_{j}^\sfr\|_{C^1(S_h)}\leq& \sqrt{\frac{\pi k_{p}}{\omega}}(1+k_{p})(\|g_{jp}^{\sfr}\|_{L^2(S_h)}+\|g_{jp}^{\sfi}\|_{L^2(S_h)})
\\&+\sqrt{\frac{\pi k_{s}}{\omega}}(1+k_{s})(\|g_{js}^{\sfr}\|_{L^2(S_h)}+\|g_{js}^{\sfi}\|_{L^2(S_h)})
\\\leq&\sqrt{2}\sqrt{\frac{\pi k_{p}}{\omega}}(1+k_{p})\|g_{jp}\|_{L^2(S_h)}+\sqrt{2}
\sqrt{\frac{\pi k_{s}}{\omega}}(1+k_{s})\|g_{js}\|_{L^2(S_h)}.
 \end{aligned}
\end{equation}
Due to \eqref{eq:lame6}, \eqref{eq:fr} and \eqref{eq:vC1}, one can verify  that
\begin{equation}\label{eq:lame39}
 \begin{aligned}
\left|\int_{S_{h}}\delta \mathbf{v}_{j}^\sfr(\mathbf{x})\cdot\mathbf{u}(\mathbf{x}){\rm d}\mathbf{x}\right|&\leq
\|\mathbf{v}_{j}^\sfr\|_{C^{\alpha}(\Omega)^{2}}\int_{S_h}|\mathbf{u}(\mathbf{x})|| \mathbf{x}|^{\alpha}{\rm d}\mathbf{x}
\\ \quad \leq & 4\sqrt{ \frac{\pi }{\omega} } \bigg(
 k_{p}^{1/2}(1+k_{p})\|g_{jp}\|_{L^2(S_h)}+
 k_{s}^{1/2}(1+k_{s})\|g_{js}\|_{L^2(S_h)}
 \bigg) \\
 &\times \frac{(\theta_{M}-\theta_{m})\Gamma(2\alpha+4)}{\delta_{W}^{2\alpha+4}}s^{-2\alpha-4},
 \end{aligned}
\end{equation}
and
\begin{equation}\label{eq:lame40}
 \begin{aligned}
\left|\int_{S_{h}}\delta \mathbf{f}_\sfr (\mathbf{x})\cdot\mathbf{u}(\mathbf{x}){\rm d}\mathbf{x}\right|\leq&
\| \mathbf{f}_\sfr (\mathbf{x})\|_{C^{\alpha}(\Omega)^{2}}\int_{W}|\mathbf{u}(\mathbf{x})||\mathbf{x}|^{\alpha}{\rm d}\mathbf{x}
\\\leq & \|\mathbf{f}_\sfr (\mathbf{x})\|_{C^{\alpha}(\Omega)^{2}}\frac{2\sqrt{2}(\theta_{M}-\theta_{m})\Gamma(2\alpha+4)}{\delta_{W}^{2\alpha+4}}s^{-2\alpha-4}.
 \end{aligned}
\end{equation}
Finally, by \eqref{eq:lame5}, \eqref{eq:lame39} and \eqref{eq:lame40}, one can arrive at \eqref{eq:lame42}.

The proof is complete. 
\end{proof}

\begin{lem}\cite[Lemma 2.4]{DCL}\label{lem:24}
	For any $\zeta>0$, if $\omega(\theta ) >0 $, then
\begin{equation}\label{eq:zeta}
	 \int_{0}^h r^\zeta  e^{-s\sqrt{r} \omega(\theta)} {\rm d} r=O ( s^{-2\zeta-2} )\quad\mbox{as  $s\rightarrow +\infty$.}
\end{equation}
\end{lem}

\begin{lem}\cite[Lemma 2.8]{DCL}\label{lem:u0 int}
	Recall that $\Gamma_h^\pm$ and $u_1(\mathbf x)$ are defined in \eqref{eq:SIGN} and \eqref{eq:lame3}, respectively. We have	\begin{align}\label{eq:I311}
	\begin{split}
\int_{\Gamma_{h} ^+ }  u_1 (\mathbf x)  {\rm d} \sigma &=2 s^{-2}\left( \mu(\theta_M )^{-2}-   \mu(\theta_M )^{-2} e^{ -s\sqrt{h} \mu(\theta_M ) }\right.  \\&\left. \hspace{3.5cm} -  \mu(\theta_M )^{-1} s\sqrt{h}   e^{ -s\sqrt{h} \mu(\theta_M ) }  \right  ),  \\
\int_{\Gamma_{h} ^- }  u_1 (\mathbf x)  {\rm d} \sigma &=2 s^{-2} \left( \mu(\theta_m )^{-2}-   \mu(\theta_m )^{-2} e^{ -s\sqrt{h}\mu(\theta_m )} \right.  \\&\left. \hspace{3.5cm}  -  \mu(\theta_m )^{-1} s\sqrt{h}   e^{ -s\sqrt{h}\mu(\theta_m ) }  \right  ),
	\end{split}
\end{align}
	where $ \mu(\theta )=\cos(\theta/2) +\mathrm i \sin( \theta/2 )$.

\end{lem}
\begin{lem}\label{lem:210}
Consider the same setup as that in Lemma \ref{lem:26} and suppose that $\eta$ has the expansion \eqref{eq:lame29}. Recall that $I_{\pm}$ is defined \eqref{eq:int1} and denote
\begin{equation}\label{int2}
I_{\pm}=\mathcal{I}^{\pm}_{1}+\eta(\mathbf{0})\mathcal{I}^{\pm}_{2},
\end{equation}
where
$\mathcal{I}^{\pm}_{1}=\int_{\Gamma_{h}^{\pm}}\delta\eta\mathbf{u}\cdot\mathbf{v}_{j}^\sfr{\rm d}\sigma$ and
$\mathcal{I}^{\pm}_{2}=\int_{\Gamma_{h}^{\pm}}\mathbf{u}\cdot\mathbf{v}_{j}^\sfr{\rm d}\sigma$. Then the following estimate holds
\begin{equation}\label{eq:deuvj1}
 \begin{aligned}
\left|\mathcal{I}^{\pm}_{1}  \right|
\leq O(s^{-2\alpha-2})+ \left(\|g_{jp}\|_{L^2(\mathbb S^1 )}+ \|g_{js}\|_{L^2(\mathbb S^1 )} \right )\times
O(s^{-2\alpha-4})\ \ \mbox{as $s\rightarrow +\infty$.}
 \end{aligned}
\end{equation}
\end{lem}
\begin{proof}
Using \eqref{eq:vjr} and the triangle inequality, one can show that
\begin{equation}
 \begin{aligned}
 |\mathcal{I}^{-}_{1}| \leq |\mathcal{I}^{-(1)}_{11}|+|\mathcal{I}^{-(2)}_{11}|+|\mathcal{I}^{-(3)}_{11}|
 +|\mathcal{I}^{-(4)}_{11}|+|\mathcal{I}^{-(1)}_{12}|+|\mathcal{I}^{-(2)}_{12}|+|\mathcal{I}^{-(3)}_{12}|
 +|\mathcal{I}^{-(4)}_{12}|,
 \end{aligned}
\end{equation}
where
\begin{eqnarray*}
  &&\mathcal{I}^{-(1)}_{11}=\int_{\Gamma_{h}^{\pm}}\delta\eta\mathbf{u}\cdot\mathbf v_{jp} ^\sfr(\mathbf 0)J_{0}(k_p |\mathbf{x}|){\rm d}\sigma, \ \
  \mathcal{I}^{-(2)}_{11}=\int_{\Gamma_{h}^{\pm}}\delta\eta\mathbf{u}\cdot \mathbf v_{js} ^\sfr(\mathbf 0)J_{0}(k_s |\mathbf{x}|) {\rm d}\sigma,
 \\ && \mathcal{I}^{-(3)}_{11}=\int_{\Gamma_{h}^{\pm}}\delta\eta\mathbf{u}\cdot\mathbf  v_{jp} ^\sfi(\mathbf 0)J_{0}(k_p |\mathbf{x}|){\rm d}\sigma, \ \
   \mathcal{I}^{-(4)}_{11}=\int_{\Gamma_{h}^{\pm}}\delta\eta\mathbf{u}\cdot \mathbf v_{js} ^\sfi(\mathbf 0)J_{0}(k_s |\mathbf{x}|) {\rm d}\sigma,
  \\ && \mathcal{I}^{-(1)}_{12}=\sqrt{2}\sum^{+\infty}_{\ell =1}  (-1)^\ell \int_{\Gamma_{h}^{\pm}}\delta\eta\mathbf{u}\cdot\mathbf{A}_{jp,1}^{(\ell )}  J_{2\ell}(k_p |\mathbf{x}|){\rm d}\sigma,
  \\ && \mathcal{I}^{-(2)}_{12}=\sqrt{2}\sum^{+\infty}_{\ell =1}  (-1)^\ell  \int_{\Gamma_{h}^{\pm}}\delta\eta\mathbf{u}\cdot\mathbf{A}_{js,1}^{(\ell)} J_{2\ell}(k_s |\mathbf{x}|){\rm d}\sigma,
  \\ && \mathcal{I}^{-(3)}_{12}= \sqrt{2}\sum^{+\infty}_{\ell =1}  (-1)^\ell   \int_{\Gamma_{h}^{\pm}}\delta\eta\mathbf{u}\cdot\mathbf{A}_{jp,2}^{(\ell )} J_{2\ell-1}(k_p |\mathbf{x}|){\rm d}\sigma,
  \\ && \mathcal{I}^{-(4)}_{12}=\sqrt{2}\sum^{+\infty}_{\ell =1}  (-1)^\ell \int_{\Gamma_{h}^{\pm}}\delta\eta\mathbf{u}\cdot\mathbf{A}_{js,2}^{(\ell)}  J_{2\ell-1}(k_s |\mathbf{x}|) {\rm d}\sigma.
  \end{eqnarray*}
The following series expression for the Bessel function $J_p(t)$ can be found in \cite{Abr} as
	 \begin{equation}\label{eq:Jp}
	 	J_p(t)= \frac{t^p}{2^p p!}+\frac{t^p}{2^p } \sum_{\ell=1}^\infty  \frac{(-1)^\ell t^{2\ell }}{4^\ell (\ell !)^2  }, \quad  \mbox{ for } p =1,\,2,\ldots,
	 \end{equation}
which is uniformly and absolutely convergent with respect to  $t \in [0,+\infty)$. 	Therefore, by \eqref{eq:Jp}, one has
\begin{equation}\notag
 \begin{aligned}
\mathcal{I}^{-(1)}_{11}
 = \int_{\Gamma_{h}^{-}}\delta\eta\mathbf{u}\cdot\mathbf v_{jp} ^\sfr(\mathbf 0){\rm d}\sigma
 +\int_{\Gamma_{h}^{-}}\delta\eta\mathbf{u}\cdot\mathbf v_{jp} ^\sfr(\mathbf 0)\sum^{\infty}_{n=1}\frac{(-1)^{n}k_p^{2n}|\mathbf{x}|^{2n}}{4^{n}(n!)^{2}}{\rm d}\sigma,
 \end{aligned}
\end{equation}
which can be used to derive that
\begin{equation}
 \begin{aligned}
\left| \mathcal{I}^{-(1)}_{11}\right|
\leq&
\left|\int_{\Gamma_{h}^{-}}\delta\eta\mathbf{u}\cdot\mathbf v_{jp} ^\sfr(\mathbf 0){\rm d}\sigma\right|
 +\left|\int_{\Gamma_{h}^{-}}\delta\eta\mathbf{u}\cdot\mathbf v_{jp} ^\sfr(\mathbf 0)\sum^{\infty}_{n=1}\frac{(-1)^{n}k_p^{2n}|\mathbf{x}|^{2n}}{4^{n}(n!)^{2}}{\rm d}\sigma\right|
\\\leq&
\sqrt{2}\|\eta\|_{C^{\alpha}}
|\mathbf v_{jp} ^\sfr(\mathbf 0)|
\int_{0}^{h} r^{\alpha} e^{-s\sqrt{r}\cos\frac{\theta_m}{2}}{\rm d}r
\\&+
\sqrt{2}\|\eta\|_{C^{\alpha}}
|\mathbf v_{jp} ^\sfr(\mathbf 0)|
\sum^{\infty}_{n=1}\frac{k_p^{2n}}{4^{n}(n!)^{2}}
\int_{0}^{h} r^{\alpha+2n} e^{-s\sqrt{r}\cos\frac{\theta_m}{2}}{\rm d}r
\end{aligned}
\end{equation}
From Lemma \ref{lem:24}, we know that
\begin{equation}\notag
\begin{aligned}
\sqrt{2}\|\eta\|_{C^{\alpha}}
|\mathbf v_{jp} ^\sfr(\mathbf 0)|
\int_{0}^{h} r^{\alpha} e^{-s\sqrt{r}\cos\frac{\theta_m}{2}}{\rm d}r &=O(s^{-2\alpha-2}),\\
\sqrt{2}\|\eta\|_{C^{\alpha}}
|\mathbf v_{jp} ^\sfr(\mathbf 0)|
\sum^{\infty}_{n=1}\frac{(-1)^{n}k^{2n}}{4^{n}(n!)^{2}}
\int_{0}^{h} r^{\alpha+2n} e^{-s\sqrt{r}\cos\frac{\theta}{2}}{\rm d}r
&\leq
\sqrt{2}\|\eta\|_{C^{\alpha}}
|\mathbf v_{jp} ^\sfr(\mathbf 0)|
\sum^{\infty}_{n=1}\frac{h^{2n-2}k^{2n}}{4^{n}(n!)^{2}}\\
&\hspace*{-4cm} \times
\int_{0}^{h} r^{\alpha+2} e^{-s\sqrt{r}\cos\frac{\theta_m}{2}}{\rm d}r
=O(s^{-2\alpha-4})\quad \mbox{as $s\rightarrow + \infty $.}
\end{aligned}
\end{equation}
Hence one can conclude that
\begin{equation}\label{eq:lame54}
\begin{aligned}
&\left|  \mathcal{I}^{-(1)}_{11} \right|
\leq O(s^{-2\alpha-2})\quad \mbox{as $s\rightarrow + \infty $.}
\end{aligned}
\end{equation}
Similarly, we can derive that
\begin{equation}\label{eq:lame541}
\begin{aligned}
\left| \mathcal{I}^{-(2)}_{11}\right|
\leq O(s^{-2\alpha-2}),\quad
\left| \mathcal{I}^{-(3)}_{11}\right|
\leq O(s^{-2\alpha-2}),\quad
\left| \mathcal{I}^{-(4)}_{11}\right|
\leq O(s^{-2\alpha-2}),
\end{aligned}
\end{equation}
as $s\rightarrow + \infty $.
By virtue of \eqref{eq:gjps norm} and \eqref{eq:Jp}, together with the Cauchy-Schwarz inequality, we can derive that
\begin{equation}\label{eq:lame55}
\begin{aligned}
\left|  \mathcal{I}^{-(1)}_{12} \right|
\leq&2\|\eta\|_{C^{\alpha}}
\int_{0}^{h} r^\alpha \left| u_{1}\sum^{\infty}_{\ell=1}(-1)^\ell\left(\frac{k_p^{2\ell}r^{2\ell}}{2^{2\ell}(2\ell)!}+\frac{k_p^{2\ell}r^{2\ell}}{2^{2\ell}}
\sum^{\infty}_{n=1}\frac{(-1)^{n}k_p^{2n}r^{2n}}{4^{n}(n!)^{2}}\right)\right| |\mathbf{A}_{jp,1}^{(\ell)}|{\rm d}r
\\ \leq&
4 \sqrt{\frac{k_p\pi}{\omega}} \|\eta\|_{C^{\alpha}}
\|g_{jp}\|_{L^2(\mathbb S^1 )}\Bigg(
\sum^{\infty}_{\ell=1}\frac{k_p^{2\ell}h^{2\ell-2}}{2^{2\ell}(2\ell)!}\int_{0}^{h}r^{\alpha+2}e^{-s\sqrt{r}\cos\frac{\theta_m}{2}}
{\rm d}r
\\&+
\sum^{\infty}_{\ell=1}\frac{k_p^{2\ell}h^{2\ell}}{2^{2\ell}}\sum^{\infty}_{n=1}\frac{k_p^{2n}h^{2(n-1)}}{4^{n}(n!)^{2}}\int_{0}^{h}r^{\alpha+2}e^{-s\sqrt{r}\cos\frac{\theta_m}{2}}
{\rm d}r\Bigg)
\\ \leq& \|g_{jp}\|_{L^2(\mathbb S^1 )}\times
O(s^{-2\alpha-4})\quad\mbox{as $s\rightarrow +\infty$.}
\end{aligned}
\end{equation}
Similarly, one can show that
\begin{equation}\label{eq:ajp int}
	\begin{split}
 &\left|  \mathcal{I}^{-(2)}_{12}\right| \leq \|g_{jp}\|_{L^2(\mathbb S^1 )}\times
O(s^{-2\alpha-4}),\ \
\left|  \mathcal{I}^{-(3)}_{12}\right| \leq \|g_{js}\|_{L^2(\mathbb S^1 )}\times
O(s^{-2\alpha-4}),\ \
\\& \left|  \mathcal{I}^{-(4)}_{12}\right| \leq \|g_{js}\|_{L^2(\mathbb S^1 )}\times
O(s^{-2\alpha-4})\quad\mbox{as $s\rightarrow +\infty$.}
	\end{split}
\end{equation}
Combining \eqref{eq:lame54}, \eqref{eq:lame541}, \eqref{eq:lame55} with \eqref{eq:ajp int}, we have
\begin{equation}\label{eq:deuvj2}
 \begin{aligned}
\left|   \mathcal{I}^{-}_{1}\right|
\leq O(s^{-2\alpha-2})+ \left(\|g_{jp}\|_{L^2(\mathbb S^1 )}+ \|g_{js}\|_{L^2(\mathbb S^1 )} \right )\times
O(s^{-2\alpha-4})\quad\mbox{as $s\rightarrow +\infty$.}
 \end{aligned}
\end{equation}
Using a similar argument for \eqref{eq:deuvj1},  we can show that
\begin{equation}\label{eq:deuvj3}
 \begin{aligned}
\left|  \mathcal{I}^{+}_{1}\right|
\leq O(s^{-2\alpha-2})+ \left(\|g_{jp}\|_{L^2(\mathbb S^1 )}+ \|g_{js}\|_{L^2(\mathbb S^1 )} \right )\times
O(s^{-2\alpha-4})\quad\mbox{as $s\rightarrow +\infty. $}
 \end{aligned}
\end{equation}
Finally, from \eqref{eq:deuvj2} and \eqref{eq:deuvj3}, it yields that \eqref{eq:deuvj1}.
\end{proof}

\begin{lem}\label{lem:211}
Consider the same setup as that in Lemma \ref{lem:26} and recall that $\mathcal{I}^{\pm}_{2}$ is defined \eqref{int2}. When $s\rightarrow +\infty$, the following results hold
\begin{eqnarray}
& &\mathcal{I}^{\pm(p)}_{211}+ \mathcal{I}^{\pm(s)}_{211}+ \mathcal{I}^{\pm(p)}_{212}+ \mathcal{I}^{\pm(s)}_{212}=
 \eta(\mathbf{0})
\bigg(\mathbf{v}_{jp}^\sfr(\mathbf{0})+\mathbf{v}_{js}^\sfr(\mathbf{0})
+\mathbf{v}_{jp}^\sfi(\mathbf{0})+\mathbf{v}_{js}^\sfi(\mathbf{0})\bigg)\cdot
\begin{pmatrix}
    1\\
   {\rm i}
\end{pmatrix} \notag
\\
&&\times2 s^{-2} \bigg(\mu(\theta_M )^{-2}-   \mu(\theta_M )^{-2} e^{ -s\sqrt{h} \mu(\theta_M ) } -  \mu(\theta_M )^{-1} s\sqrt{h}   e^{ -s\sqrt{h} \mu(\theta_M ) } \notag
\\
&&+\mu(\theta_m )^{-2}-   \mu(\theta_m )^{-2} e^{ -s\sqrt{h} \mu(\theta_m ) } -  \mu(\theta_m )^{-1} s\sqrt{h}   e^{ -s\sqrt{h} \mu(\theta_m ) }\bigg), \notag
 \\
 &&\left|\mathcal{I}^{\pm(p)}_{212}\right| \leq O(s^{-6}), \ \
 \left|\mathcal{I}^{\pm(s)}_{212}\right| \leq O(s^{-6}),  \label{uvjp1}
\end{eqnarray}
where
\begin{eqnarray*}
&&\mathcal{I}^{\pm(\beta)}_{211}=\int_{\Gamma_{h}^{\pm}}\mathbf{u}\cdot\mathbf v_{j\beta} ^\sfr(\mathbf 0){\rm d}\sigma, \ \
\mathcal{I}^{\pm(\beta )}_{212}= \int_{\Gamma_{h}^{\pm}}\mathbf{u}\cdot\mathbf v_{j\beta } ^\sfr(\mathbf 0)S_\beta {\rm d}\sigma,\ \
S_\beta=\sum^{\infty}_{n=1}\frac{(-1)^{n}k_{\beta}^{2n}|\mathbf{x}|^{2n}}{4^{n}(n!)^{2}}. 
\end{eqnarray*}
Denote
\begin{eqnarray*}
 &&\mathcal{I}^{\pm(\beta)}_{213}=\int_{\Gamma_{h}^{\pm}}\mathbf{u}\cdot\mathbf v_{j\beta } ^\sfi(\mathbf 0){\rm d}\sigma, \ \
\mathcal{I}^{\pm(\beta )}_{214}= \int_{\Gamma_{h}^{\pm}}\mathbf{u}\cdot\mathbf v_{j\beta} ^\sfi(\mathbf 0)S_\beta {\rm d}\sigma,\\ 
&&\mathcal{I}^{\pm(\beta)}_{22}= \sqrt{2}\sum^{+\infty}_{\ell =1} (-1)^\ell\int_{\Gamma_{h}^{\pm}}\mathbf{u}\cdot\mathbf{A}_{j\beta,1}^{(\ell)}    J_{2\ell}(k_\beta  |\mathbf{x}|){\rm d}\sigma,\\
\\ &&\mathcal{I}^{\pm(\beta )}_{23}= \sqrt{2}\sum^{+\infty}_{\ell =1}  (-1)^\ell \int_{\Gamma_{h}^{\pm}}\mathbf{u}\cdot\mathbf{A}_{j\beta,1}^{(\ell )}   J_{2\ell-1}(k_\beta  |\mathbf{x}|){\rm d}\sigma. 
\end{eqnarray*}
It holds that
\begin{eqnarray*}
&& \left|\mathcal{I}^{\pm(p)}_{214}\right| \leq O(s^{-6}), \ \ \left|\mathcal{I}^{\pm(s)}_{214}\right| \leq O(s^{-6}),\notag
 \\
 &&\left|\mathcal{I}^{\pm(p)}_{22}\right|\leq\|g_{jp}\|_{L^{2}(\mathbb{S}^{1})}\times O(s^{-4}), \ \
\left|\mathcal{I}^{\pm(s)}_{22}\right|\leq\|g_{js}\|_{L^{2}(\mathbb{S}^{1})}\times O(s^{-4}), \notag
 \\
 &&\left|\mathcal{I}^{\pm(p)}_{23}\right|\leq\|g_{jp}\|_{L^{2}(\mathbb{S}^{1})}\times O(s^{-4}), \ \
\left|\mathcal{I}^{\pm(s)}_{23}\right|\leq\|g_{js}\|_{L^{2}(\mathbb{S}^{1})}\times O(s^{-4}) \label{uvjp2}
\end{eqnarray*}
as $s \rightarrow +\infty $.
\end{lem}

\begin{proof}
Using \eqref{eq:vjr}, we can deduce that
\begin{equation}\label{uvje}
 \begin{aligned}
 \mathcal{I}^{-}_{2}=\mathcal{I}^{-(1)}_{21}+\mathcal{I}^{-(2)}_{21}+\mathcal{I}^{-(3)}_{21}
 +\mathcal{I}^{-(4)}_{21}+\mathcal{I}^{-(p)}_{22}+\mathcal{I}^{-(s)}_{22}+\mathcal{I}^{-(p)}_{23}
 +\mathcal{I}^{-(s)}_{23},
 \end{aligned}
 \end{equation}
where
\begin{eqnarray*}
&&\mathcal{I}^{-(1)}_{21}=\int_{\Gamma_{h}^{-}}\mathbf{u}\cdot\mathbf v_{jp} ^\sfr(\mathbf 0)J_{0}(k_p |\mathbf{x}|){\rm d}\sigma, \ \ \mathcal{I}^{-(2)}_{21}=\int_{\Gamma_{h}^{-}}\mathbf{u}\cdot\mathbf v_{js} ^\sfr(\mathbf 0)J_{0}(k_s |\mathbf{x}|){\rm d}\sigma,
\\&&\mathcal{I}^{-(3)}_{21}=\int_{\Gamma_{h}^{-}}\mathbf{u}\cdot\mathbf v_{jp} ^\sfi(\mathbf 0)J_{0}(k_p |\mathbf{x}|){\rm d}\sigma, \ \ \mathcal{I}^{-(4)}_{21}=\int_{\Gamma_{h}^{-}}\mathbf{u}\cdot\mathbf v_{js} ^\sfi(\mathbf 0)J_{0}(k_s |\mathbf{x}|){\rm d}\sigma,
\\&&\mathcal{I}^{-(p)}_{22}= \sqrt{2}\sum^{+\infty}_{\ell =1}  (-1)^\ell \int_{\Gamma_{h}^{-}}\mathbf{u}\cdot\mathbf{A}_{jp,1}^{(\ell)}  J_{2\ell}(k_p |\mathbf{x}|){\rm d}\sigma,
\\&&\mathcal{I}^{-(s)}_{22}=\sqrt{2}\sum^{+\infty}_{\ell =1}  (-1)^\ell \int_{\Gamma_{h}^{-}}\mathbf{u}\cdot\mathbf{A}_{js,1}^{(\ell )}   J_{2\ell}(k_s |\mathbf{x}|){\rm d}\sigma,
\\&&\mathcal{I}^{-(p)}_{23}= \sqrt{2}\sum^{+\infty}_{\ell =1}  (-1)^\ell  \int_{\Gamma_{h}^{-}}\mathbf{u}\cdot\mathbf{A}_{jp,1}^{(\ell )} J_{2\ell-1}(k_p |\mathbf{x}|){\rm d}\sigma,
\\ &&\mathcal{I}^{-(s)}_{23}= \sqrt{2}\sum^{+\infty}_{\ell =1}  (-1)^\ell \int_{\Gamma_{h}^{-}}\mathbf{u}\cdot\mathbf{A}_{js,1}^{(\ell)}   J_{2\ell-1}(k_s |\mathbf{x}|){\rm d}\sigma.
\end{eqnarray*}
Using \eqref{eq:Jp}, one can obtain that
\begin{equation}\label{uvjpe}
 \begin{aligned}
 \mathcal{I}^{-(1)}_{21}=\mathcal{I}^{-(p)}_{211}+\mathcal{I}^{-(p)}_{212}.
 \end{aligned}
\end{equation}
For $ \mathcal{I}^{-(p)}_{211}$, we have
\begin{equation}\label{2.61}
 \begin{aligned}
 \mathcal{I}^{-(p)}_{211}=
 \mathbf{v}_{jp} ^\sfr(\mathbf{0})\cdot
\begin{pmatrix}
    1\\
   {\rm i}
\end{pmatrix}
\int_{0}^{h}e^ {-s\sqrt{r}\mu(\theta_{m})}{\rm d}r,
  \end{aligned}
\end{equation}
where $ \mu(\theta)=\cos\frac{\theta}{2}+{\rm i}\sin\frac{\theta}{2}$ .
By virtue of \eqref {2.61} and Lemma \ref{lem:u0 int},  we have
\begin{equation}
 \begin{aligned}
 \mathcal{I}^{-(p)}_{211}=
 \mathbf{v}_{jp} ^\sfr(\mathbf{0})\cdot
\begin{pmatrix}
    1\\
   {\rm i}
\end{pmatrix}
\frac{2}{s^{2}}\bigg(\frac{1}{\mu(\theta_{m})^2}-\frac{e^{-s\sqrt{h}\mu(\theta_{m})}}{\mu(\theta_{m})^2}-s\sqrt{h}\frac{e^{-s\sqrt{h}\mu(\theta_{m})}}{\mu(\theta_{m})}\bigg).
  \end{aligned}
\end{equation}
Using Lemma \ref{lem:24}, one has
\begin{equation}\label{uvjp}
 \begin{aligned}
 \left|\mathcal{I}^{-(p)}_{212}\right| \leq&
\sqrt{2}|\mathbf{v}_{jp}^\sfr(\mathbf{0})|
\sum^{\infty}_{n=1}\frac{k_{p}^{2n}}{4^{n}(n!)^{2}}\int_{0}^{h}r^{2n}e^{-s\sqrt{r}\cos\frac{\theta_{m}}{2}}{\rm d}r
\\\leq&
\sqrt{2}|\mathbf{v}_{jp}^\sfr(\mathbf{0})|
\sum^{\infty}_{n=1}\frac{k_{p}^{2n}h^{2n-2}}{4^{n}(n!)^{2}}\int_{0}^{h}r^{2}e^{-s\sqrt{r}\cos\frac{\theta_{m}}{2}}{\rm d}r=
O(s^{-6}),
  \end{aligned}
\end{equation}
as $s\rightarrow +\infty. $
Similarly, We can derive that
\begin{equation}\label{uvjp1}
 \begin{aligned}
 &\mathcal{I}^{\pm(p)}_{211}+ \mathcal{I}^{\pm(s)}_{211}+ \mathcal{I}^{\pm(p)}_{213}+ \mathcal{I}^{\pm(s)}_{213}=
 \eta(\mathbf{0})
\bigg(\mathbf{v}_{jp}^\sfr(\mathbf{0})+\mathbf{v}_{js}^\sfr(\mathbf{0})
+\mathbf{v}_{jp}^\sfi(\mathbf{0})+\mathbf{v}_{js}^\sfi(\mathbf{0})\bigg)\cdot
\begin{pmatrix}
    1\\
   {\rm i}
\end{pmatrix}
\\&\times2 s^{-2} \bigg(\mu(\theta_M )^{-2}-   \mu(\theta_M )^{-2} e^{ -s\sqrt{h} \mu(\theta_M ) } -  \mu(\theta_M )^{-1} s\sqrt{h}   e^{ -s\sqrt{h} \mu(\theta_M ) }
\\&+\mu(\theta_m )^{-2}-   \mu(\theta_m )^{-2} e^{ -s\sqrt{h} \mu(\theta_m ) } -  \mu(\theta_m )^{-1} s\sqrt{h}   e^{ -s\sqrt{h} \mu(\theta_m ) }\bigg)
 \\&\left|\mathcal{I}^{\pm (p)}_{212}\right| \leq O(s^{-6}), \ \
 \left|\mathcal{I}^{\pm(s)}_{212}\right| \leq O(s^{-6}), \ \ \left|\mathcal{I}^{\pm(p)}_{213}\right| \leq O(s^{-6}), \ \ \left|\mathcal{I}^{\pm(s)}_{213}\right| \leq O(s^{-6}),\ \
  \end{aligned}
\end{equation}
as $s\rightarrow +\infty. $

By virtue of \eqref{eq:Jp} and using the Cauchy-Schwarz inequality, we can derive that
\begin{equation}\label{uajp1}
 \begin{aligned}
 \mathcal{I}^{-(1)}_{22}= &\sqrt{2}\sum^{+\infty}_{\ell =1}  (-1)^\ell\int_{\Gamma_{h}^{-}}\mathbf{u}\cdot\mathbf{A}_{jp,1}^{(\ell)}  \frac{k^{2\ell}r^{2\ell}}{2^{2\ell}(2\ell!)}{\rm d}\sigma
 \\&+
\sqrt{2}\sum^{+\infty}_{\ell =1}  (-1)^\ell  \int_{\Gamma_{h}^{-}}\mathbf{u}\cdot\mathbf{A}_{jp,1}^{(\ell)}  \frac{k^{2\ell}r^{2\ell}}{2^{2\ell}}
\sum^{\infty}_{n=1}\frac{(-1)^{n}k^{2n}r^{2n}}{4^{n}(n!)^{2}} {\rm d}\sigma
.
 \end{aligned}
\end{equation}
From Lemma \ref{lem:24} and \eqref{eq:gjps norm} , it is easy to see that
\begin{equation}\label{(2.63)}
 \begin{aligned}
 \left| \sqrt{2}\sum^{+\infty}_{\ell =1}  (-1)^\ell \int_{\Gamma_{h}^{-}}\mathbf{u}\cdot\mathbf{A}_{jp,1}^{(\ell)}  \frac{k^{2\ell}r^{2\ell}}{2^{2\ell}(2\ell!)}{\rm d}\sigma\right|
 \leq&
 2
\sum^{\infty}_{\ell=1}\frac{k^{2\ell}h^{2\ell-1}}{2^{2\ell}(2\ell!)}\int_{0}^{h}re^{-s\sqrt{r}\cos\frac{\theta_{m}}{2}} |\mathbf{A}_{jp,1}^{(\ell)}|
{\rm d}r
 \\\leq&
 \|g_{jp}\|_{L^{2}(\mathbb{S}^{1})}\times O(s^{-4})
 \end{aligned}
\end{equation}
and
\begin{eqnarray}
& &\left|\sqrt{2}\sum^{+\infty}_{\ell =1}  (-1)^\ell\int_{\Gamma_{h}^{-}}\mathbf{u}\cdot\mathbf{A}_{jp,1}^{(\ell)} \frac{k^{2\ell}r^{2\ell}}{2^{2\ell}}
\sum^{\infty}_{n=1}\frac{(-1)^{n}k^{2n}r^{2n}}{4^{n}(n!)^{2}} {\rm d}\sigma  \right| \notag
 \\
& &
\leq
2
\sum^{\infty}_{\ell=1}\frac{k^{2\ell}r^{2\ell}}{2^{2\ell}}
\sum^{\infty}_{n=1}\frac{k^{2n}h^{2(n-1)}}{4^{n}(n!)^{2}} \int_{0}^{h}r^{2}e^{-s\sqrt{r}\cos\frac{\theta_{m}}{2}} |\mathbf{A}_{jp,1}^{(\ell)}|
{\rm d}r \notag
 \\ & &\leq
\|g_{jp}\|_{L^{2}(\mathbb{S}^{1})}\times O(s^{-6}) \label{(2.64)}
\end{eqnarray}
as $s\rightarrow +\infty. $

Combining  \eqref{uajp1}, \eqref{(2.63)} with \eqref{(2.64)}, we have the estimate
\begin{equation}
 \begin{aligned}\label{(2.65)}
\left|\mathcal{I}^{-(p)}_{22}\right|\leq\|g_{jp}\|_{L^{2}(\mathbb{S}^{1})}\times O(s^{-4})\quad\mbox{as $s\rightarrow +\infty. $}
 \end{aligned}
\end{equation}
Using a similar argument as that for \eqref{(2.65)},  we can show that
\begin{equation}\label{(2.66)}
 \begin{aligned}
 &\left|\mathcal{I}^{+(p)}_{22}\right|\leq\|g_{js}\|_{L^{2}(\mathbb{S}^{1})}\times O(s^{-4}), \ \
\left|\mathcal{I}^{\pm(s)}_{22}\right|\leq\|g_{js}\|_{L^{2}(\mathbb{S}^{1})}\times O(s^{-4}),
 \\&\left|\mathcal{I}^{\pm(p)}_{23}\right|\leq\|g_{jp}\|_{L^{2}(\mathbb{S}^{1})}\times O(s^{-4}), \ \
\left|\mathcal{I}^{\pm(s)}_{23}\right|\leq\|g_{js}\|_{L^{2}(\mathbb{S}^{1})}\times O(s^{-4}),
 \end{aligned}
\end{equation}
as $s\rightarrow +\infty. $

The proof is complete. 
\end{proof}

\begin{thm}\label{thm:29}
 Let $(\mathbf{v}, \mathbf{w})\in H^{1}(\Omega)^{2}\times H^{1}(\Omega)^{2}$ be a pair of eigenfunctions to \eqref{eq:lame1} associated with $\omega\in\mathbb{R}_{+}$. Assume that the domain $\Omega\subset\mathbb{R}^{2}$ contains a corner $\Omega \cap B_h= \Omega\cap W$ with $h\ll 1$. By rigid motions if necessary, we can assume that the vertex of the corner is $\mathbf 0 \in \partial \Omega $. Let $W$ be the sector defined in \eqref{eq:lame2} and $S_{h}=\Omega \cap B_h= \Omega\cap W $  in $\Omega$. Suppose that $q\mathbf{w} \in C^{\alpha}(\overline{S_{h}})^2$ with $q:=1+V$ and $\eta\in C^{\alpha}(\overline{\Gamma_{h}^{\pm}})$ for $0<\alpha<1$. If the following conditions are fulfilled:
\begin{itemize}
\item[{\rm (a)}]~For any given constants $\gamma>\max\{\beta_{1}, \beta_{2}\}>0$, assume that there exits a sequence Herglotz functions $\{\mathbf{v}_{j}\}_{j=1}^\infty  $, where $\mathbf{v}_{j} $ is  defined in \eqref{eq:vj},  can  approximate $\mathbf{v}$ in $H^{1}(S_{h})^{2}$ fulfilling 
\begin{equation}\label{eq:cond thm3 new}
\| \mathbf{v}-\mathbf{v}_{j}\|_{H^{1}(S_{h})^{2}}\leq j^{-\gamma }, \ \| g_{jp}\|_{L^{2}(\mathbb S^{1})}\leq j^{\beta_{1}}, \ \| g_{js}\|_{L^{2}(\mathbb S^{1})}\leq j^{\beta_{2}},
\end{equation}
\item[{\rm (b)}]~the function $\eta(\mathbf{x})$ doest not vanish at the corner point, i.e.,
\begin{equation}\label{etao}
\eta(\mathbf{0} )\neq0,
\end{equation}
\item[{\rm (c)}]~the corner is non-degenerate, namely the angles $\theta_{m}$ and $\theta_{M}$ of the sector $W$ satisfy
\begin{equation}\label{eq:cond theta}
-\pi<\theta_{m}<\theta_{M}<\pi \mbox{ and } \theta_{M}-\theta_{m}\neq\pi;
\end{equation}
\end{itemize}
then we have
\begin{equation}\label{eq:277}
	\lim_{ \rho\rightarrow+0 }\frac{1}{m(B(\mathbf{0},\rho)\cap \Omega)}\int_{B(\mathbf{0},\rho)\cap\Omega}|\mathbf{v}(\mathbf{x})|\mathrm d \mathbf{x}=0,
\end{equation}
where $m(B(\mathbf{0},\rho)\cap\Omega)$ is the measure of $B(\mathbf{0},\rho)\cap\Omega$.
\end{thm}

%
\begin{proof}

As remarked earlier, we shall make use of the splitting \eqref{eq:decouple} and it is sufficient for us to show that $\mathbf{v}_\sfr$ satisfies the geometric property \eqref{eq:277}. First, it is easy to see that $q\mathbf{w}_\sfr \in C^{\alpha }(\overline{S}_h)^{2}$ and $\mathbf v_\sfr \in H^1(S_h)^2$ can be approximated by $\{\mathbf v_j^\sfr \}_{j=1}^{+\infty}$ defined in \eqref{eq:vrj 28} satisfying \eqref{eq:cond thm3}.  Therefore the assumptions in Lemmas \ref{lem:26}--\ref{lem:211} are fulfilled.

Substituting  \eqref{int2},  \eqref{uvje}, \eqref{uvjpe} and  \eqref{uvjp1} into  \eqref{eq:int1} and rearranging the terms, we have
\begin{eqnarray}
& &\eta(\mathbf{0})
\bigg(\mathbf{v}_{jp}^\sfr(\mathbf{0})+\mathbf{v}_{js}^\sfr(\mathbf{0})
+\mathbf{v}_{jp}^\sfi(\mathbf{0})+\mathbf{v}_{js}^\sfi(\mathbf{0})\bigg)\cdot
\begin{pmatrix}
    1\\
   {\rm i}
\end{pmatrix} \label{eq:lame66}
\\
&& \times  2 s^{-2} \bigg(\mu(\theta_M )^{-2}-   \mu(\theta_M )^{-2} e^{ -s\sqrt{h} \mu(\theta_M ) } -  \mu(\theta_M )^{-1} s\sqrt{h}   e^{ -s\sqrt{h} \mu(\theta_M ) } \notag
\\
&& +\mu(\theta_m )^{-2}-   \mu(\theta_m )^{-2} e^{ -s\sqrt{h} \mu(\theta_m ) } -  \mu(\theta_m )^{-1} s\sqrt{h}   e^{ -s\sqrt{h} \mu(\theta_m ) }\bigg) \notag
\\
& &=-\eta(\mathbf{0})(\mathcal{I}^{\pm(p)}_{22}+\mathcal{I}^{\pm(s)}_{22}+\mathcal{I}^{\pm(p)}_{23}
 +\mathcal{I}^{\pm(s)}_{23}
+\mathcal{I}^{\pm(p)}_{211}+\mathcal{I}^{\pm(s)}_{212}+\mathcal{I}^{\pm(p)}_{214}
  ) \notag
\\
& &- \mathcal{I}^{\pm}_{1}+I_{\Lambda_h } - I_1-I_2-I_{\pm}^\Delta. \notag
\end{eqnarray}
Multiplying $s^{2}$ on the both side of \eqref{eq:lame66}, by virtue of \eqref{e}, \eqref{eq:lame22},   \eqref{eq:lame35}, \eqref{eq:lame42},  \eqref{int2}, \eqref{eq:deuvj1},  \eqref{uvjp1} and \eqref{eq:cond thm3}, and letting $s=j^{\varrho/2}$ ($\max\{\beta_{1}, \beta_{2}\}<\varrho<\gamma$) with $j\rightarrow +\infty$, we have
\begin{equation}\label{eq:279}
 \begin{aligned}
\lim_{j\rightarrow \infty}\eta(\mathbf{0})
\bigg(\mathbf{v}_{jp}^\sfr(\mathbf{0})+\mathbf{v}_{js}^\sfr(\mathbf{0})
+\mathbf{v}_{jp}^\sfi(\mathbf{0})+\mathbf{v}_{js}^\sfi(\mathbf{0})\bigg)\cdot
\begin{pmatrix}
    1\\
   {\rm i}
\end{pmatrix}
\bigg(\mu^{-2}(\theta_{m})+\mu^{-2}(\theta_{M})\bigg)=0,
 \end{aligned}
\end{equation}
which further implies that
\begin{equation}\label{eq:273 eq}
\lim_{j\rightarrow \infty}\eta(\mathbf{0})
\mathbf{v}_{j}^\sfr(\mathbf{0}) \bigg(\mu^{-2}(\theta_{m})+\mu^{-2}(\theta_{M})\bigg)=0.
\end{equation}
Here we use the fact that
\begin{equation}\notag
 \begin{aligned}
 \mathbf{v}_{j}^\sfr(\mathbf{0})=\mathbf{v}_{jp}^\sfr(\mathbf{0})+\mathbf{v}_{js}^\sfr(\mathbf{0})
+\mathbf{v}_{jp}^\sfi(\mathbf{0})+\mathbf{v}_{js}^\sfi(\mathbf{0}),
 \end{aligned}
\end{equation}
according to Proposition \ref{Pro:2.2}.

Under the condition \eqref{eq:cond theta}, from \cite[Lemma 2.10]{DCL}, we know that
\begin{equation}\label{eq:mum}
\mu^{-2}(\theta_{m})+\mu^{-2}(\theta_{M})\neq 0.
\end{equation}
Since $\eta$ is a real valued function, by virtue of \eqref{etao} and \eqref{eq:mum}, from \eqref{eq:273 eq},  one has
\begin{equation}
 \begin{aligned}
 \lim_{j\rightarrow \infty} \mathbf{v}_{j}^\sfr(\mathbf{0}) \cdot
\begin{pmatrix}
    1\\
   {\rm i}
\end{pmatrix}=0.
 \end{aligned}
\end{equation}
which implies that
 \begin{equation}\label{eq:281}
 \begin{aligned}
 \lim_{j\rightarrow \infty}\mathbf{v}_{j}^\sfr(\mathbf{0})=\mathbf{0}.
  \end{aligned}
\end{equation}
Finally, it can be directly deduced that
\begin{equation}\label{finally}
 \begin{aligned}
&\lim_{\rho\rightarrow0}\frac{1}{m(B(\mathbf{0},\rho)\cap\Omega)}\int_{B(\mathbf{0},\rho)\cap\Omega}|\mathbf{v}_\sfr(\mathbf{x})|{\rm d}\mathbf{x}\\
&\leq
\lim_{j\rightarrow\infty}\lim_{\rho\rightarrow0}\frac{1}{m(B(\mathbf{0},\rho)\cap\Omega}\int_{B(\mathbf{0},\rho)\cap\Omega}\big|\mathbf{v}_\sfr(\mathbf{x})-\mathbf{v}_{j}^\sfr(\mathbf{x})\big|{\rm d}\mathbf{x} \\
&+\lim_{\rho\rightarrow0}\frac{1}{m(B(\mathbf{0},\rho)\cap\Omega)}\int_{B(\mathbf{0},\rho)\cap\Omega}\big|\mathbf{v}^\sfr_{j}(\mathbf{x})\big|{\rm d}\mathbf{x}.
 \end{aligned}
\end{equation}
Combining \eqref{eq:cond thm3} and  \eqref{eq:281} with \eqref{finally}, one can prove \eqref{eq:277}.

The proof is complete. 
\end{proof}

\begin{rem}\label{rem:23}
We would like to point out that the Fourier extension property \eqref{eq:cond thm3 new} can be generalized as follows
	\begin{equation}\label{eq:cond thm3 gene}
\| \mathbf{v}-\mathbf{v}_{j}\|_{H^{1}(S_{h})^{2}}\leq \phi_1( j), \ \| g_{jp}\|_{L^{2}(\mathbb S^{1})}\leq \frac{1}{\phi_2( j)}, \ \| g_{js}\|_{L^{2}(\mathbb S^{1})}\leq\frac{1}{ \phi_3( j)},
\end{equation}
where $\phi_{\ell}(j) \in \mathbb R_+$ are strict decreasing functions with respect to $j$ and $\lim_{ j\rightarrow +\infty }\phi_\ell(j)=0$, $\ell=1,2,3$, satisfying
$$
(\phi_1(j))^\gamma \leq  \phi_4(j)= \min\{\phi_2(j), \phi_3(j) \},  \quad 0< \gamma<1,\quad \forall j\in \mathbb N.
$$
In fact, by letting
\begin{equation}\label{eq:s 288}
	s=\frac{1}{(\phi_1(j) )^{ \rho/2 }},\quad \gamma< \rho <1,
\end{equation}
one can show that
\begin{equation}\label{eq:289}
 s^2 \| \mathbf{v}-\mathbf{v}_{j}\|_{H^{1}(S_{h})^{2}}\leq	
 (\phi_1(j))^{1-\rho},\, s^{-2}\| g_{jp}\|_{L^{2}(\mathbb S^{1})}\leq (\phi_1( j))^{ \rho -\gamma },\, s^{-2}\| g_{js}\|_{L^{2}(\mathbb S^{1})}\leq (\phi_1( j))^{ \rho -\gamma }
\end{equation}
By virtue of \eqref{eq:cond thm3 gene}, under the same setup of Theorem  \ref{thm:29}, and using a similar argument as that in Lemma \ref{lem:27}, one can prove that
\begin{equation}\label{eq:290}
	\begin{split}
		\left| I_{2}\right|&\leq\omega^{2}h\sqrt{\theta_{M}-\theta_{m}}e^{-s\sqrt{\Theta}\delta_{W}}\phi_1( j),\\
		\left| I_{\pm}^\Delta\right|\leq &\Big (|\eta(\mathbf{0})| h\sqrt{\theta_{M}-\theta_{m}}e^{-s\sqrt{\Theta}\delta_{W}}\\&+\|\eta\|_{C^{\alpha}}\frac{2\sqrt{\theta_{M}-\theta_{m}\Gamma(4\alpha+4)}}{(2\delta_{W})^{2\alpha+2}}s^{-2(\alpha+1) }\Big)\phi_1( j).
	\end{split}
\end{equation}
In the integral identity \eqref{eq:lame66}, we can multiply $s^2$ on both sides of \eqref{eq:lame66} and choose $s$ by \eqref{eq:s 288}. Then by virtue of \eqref{eq:289},  \eqref{eq:290}, \eqref{e}, \eqref{eq:lame42},  \eqref{int2} and \eqref{eq:deuvj1}, we can obtain
$$
\lim_{j\rightarrow \infty}\mathbf{v}_{j}(\mathbf{0})=\mathbf{0}
$$
 by letting $j\rightarrow +\infty$. Therefore, we can prove \eqref{eq:277} under the same setup of Theorem \ref{thm:29}, where the condition \eqref{eq:cond thm3} is replaced by \eqref{eq:cond thm3 gene} and $q\mathbf w\in C^\alpha(\overline{S_h} )^2 $.

\end{rem}

We next consider the degenerate case of \eqref{eq:lame1} with $\eta \equiv 0$ in \eqref{eq:lame1}. We have


\begin{cor}\label{cor:22}
Under the same setup as that in Theorem~\ref{thm:29} but with $\eta\equiv 0$, we have
\begin{equation}\label{eq:lln1}
\lim_{\rho\rightarrow 0}\frac{1}{m(B(\mathbf{0},\rho)\cap\Omega)}\int_{B(\mathbf{0},\rho)\cap\Omega}|V(\mathbf x)\mathbf{w}(\mathbf{x})|{\rm d}\mathbf{x}=0.
\end{equation}
\end{cor}

\begin{rem}
The difference between Theorem~\ref{thm:29} and Corollary~\ref{cor:22} lies in their conclusions \eqref{eq:277} and \eqref{eq:lln1}. If one further assumes in Corollary \ref{cor:22} that $V(\mathbf x)\in C(\overline{S_h})$ and $ V(\mathbf 0) \neq 0$, it is easy to show that	
\[
\lim_{ \rho \rightarrow +0 }\frac{1}{m(B(\mathbf 0, \rho  )\cap \Omega)} \int_{B(\mathbf 0, \rho )\cap \Omega}|\mathbf w (\mathbf x) | {\rm d}\mathbf x=0,
\]
which together with the transmission condition $\mathbf{v}=\mathbf{w}$ on $\Gamma$ further implies that \eqref{eq:277} holds as well. 

\end{rem}

%
\begin{proof}[Proof of Corollary~\ref{cor:22}]
The proof follows from the one of Theorem \ref{thm:29} with some necessary modifications.
Similar to \eqref{eq:int1},  we have the integral identity as follows
\begin{equation}\label{eq:I1 294}
	I_1+I_2=I_{\Lambda_h },
\end{equation}
where  $I_1$, $I_2$ and $I_{\Lambda_h }$ are defined in \eqref{eq:int1}. Substituting \eqref{eq:fr} and \eqref{eq:lame37} into \eqref{eq:I1 294}, one can see that
\begin{equation}\label{eq:lame77}
 \begin{aligned}
(-\omega^{2}\mathbf{v}_{j}^\sfr(\mathbf{0})+\omega^{2}\mathbf{f}_\sfr(\mathbf{0}))\int_{S_{h}}\mathbf{u}(\mathbf{x}){\rm d}\mathbf{x}
+I_2
 = &I_{\Lambda_h }+\int_{S_{h}}\omega^{2}\delta\mathbf{f}_\sfr(\mathbf{x})\cdot\mathbf{u}(\mathbf{x}){\rm d}\mathbf{x}
 \\&-\int_{S_{h}}\omega^{2}\delta\mathbf{v}^\sfr_{j}(\mathbf{x})\cdot\mathbf{u}(\mathbf{x}){\rm d}\mathbf{x},
 \end{aligned}
\end{equation}
where $\mathbf{v}_{j}^\sfr(\mathbf{0}), \mathbf{u}(\mathbf{x}),  \mathbf{f}_\sfr(\mathbf{0})$ are defined in \eqref{eq:vjr}, \eqref{eq:lame3} and \eqref{eq:fr}.

From Lemma \ref{lem:22}, it is easy to obtain that
\begin{equation}\label{eq:lame78}
 \begin{aligned}
&(-\omega^{2}\mathbf{v}_{j}^\sfr(\mathbf{0})+\omega^{2}\mathbf{f}_\sfr(\mathbf{0}))\int_{S_{h}}\mathbf{u}(\mathbf{x}){\rm d}\mathbf{x}
\\=&
\omega^{2}(\mathbf{f}_\sfr(\mathbf{0})-\mathbf{v}_{j}^\sfr(\mathbf{0}) )\cdot \begin{bmatrix}
	1 \\ \rm i
\end{bmatrix}
 \left(\int_{W}{u}_1(\mathbf{x}){\rm d}\mathbf{x}-\int_{W\backslash S_{h}}{u}_1(\mathbf{x}){\rm d}\mathbf{x}\right)
\\=&
\omega^{2}(\mathbf{f}_\sfr(\mathbf{0})-\mathbf{v}_{j}^\sfr(\mathbf{0}) )\cdot \begin{bmatrix}
	1 \\ \rm i
\end{bmatrix}
 \left( 6{\rm i}(e^{-2\theta_{M}{\rm i}}-e^{-2\theta_{m} \mathrm i})s^{-4}-\int_{W\backslash S_{h}}{u}_1(\mathbf{x}){\rm d}\mathbf{x} \right).
 \end{aligned}
\end{equation}

Substituting \eqref{eq:lame78}  into \eqref{eq:lame77} and rearranging the terms, we have
\begin{equation}\label{eq:lame79}
 \begin{aligned}
&\omega^{2} (\mathbf{f}_\sfr(\mathbf{0})-\mathbf{v}_{j}^\sfr(\mathbf{0})) \cdot  \begin{bmatrix}
	1 \\ \rm i
\end{bmatrix} 6{\rm i}(e^{-2\theta_{M}{\rm i}}-e^{-2\theta_{m} \mathrm i})s^{-4}
 \\=&-I_2+I_{\Lambda_h}
+\int_{S_{h}}\omega^{2}\delta\mathbf{f}_\sfr(\mathbf{x})\cdot\mathbf{u}(\mathbf{x}){\rm d}\mathbf{x}
 -\int_{S_{h}}\omega^{2}\delta\mathbf{v}_{j}^\sfr(\mathbf{x})\cdot\mathbf{u}(\mathbf{x}){\rm d}\mathbf{x}
\\& + \omega^{2} (\mathbf{f}_\sfr(\mathbf{0})-\mathbf{v}_{j}^\sfr(\mathbf{0})) \cdot  \begin{bmatrix}
	1 \\ \rm i
\end{bmatrix}\int_{W\backslash S_{h}}{u}_1(\mathbf{x}){\rm d}\mathbf{x}.
 \end{aligned}
\end{equation}
Note that \eqref{eq:cond thm3 new} holds. Multiplying $s^{4}$ on the both sides of \eqref{eq:lame79},  using \eqref{eq:lame7},  \eqref{eq:lame22}, \eqref{eq:lame39} and \eqref{eq:lame40}, and letting $s=j^{\varrho/4}$ ($\max\{\beta_{1}, \beta_{2}\}<\varrho<\gamma$) with $j\rightarrow +\infty$, we have
\begin{equation}\label{eq:294 vrj}
 \begin{aligned}
\lim_{j\rightarrow\infty}\mathbf{v}_{j}^\sfr(\mathbf{0})=\mathbf{f}_\sfr(\mathbf{0}).
 \end{aligned}
\end{equation}
Since
\begin{equation}\notag
 \begin{aligned}
\lim_{j \rightarrow \infty}  \mathbf{v}_j^\sfr(\mathbf{0})&=\lim_{j \rightarrow \infty}  \lim_{ \rho \rightarrow +0 }\frac{1}{m(B(\mathbf{0}, \rho  )\cap\Omega)} \int_{B(\mathbf{0}, \rho )\cap\Omega} \mathbf{v}_j^\sfr(\mathbf{x})  {\rm d} \mathbf{x}\\&= \lim_{ \rho \rightarrow +0 }\frac{1}{m(B(\mathbf{0}, \rho  )\cap\Omega)} \int_{B(\mathbf{0}, \rho )\cap\Omega} \mathbf{v}_\sfr(\mathbf{x})  {\rm d} \mathbf{x},
\end{aligned}
\end{equation}
and
$$
\mathbf{f}_\sfr(\mathbf{0})= \lim_{ \rho \rightarrow +0 }\frac{1}{m(B(\mathbf{0}, \rho  )\cap\Omega)} \int_{B(\mathbf{0}, \rho )\cap\Omega}   q\mathbf{w}_\sfr(\mathbf{x})  {\rm d} \mathbf{x},
$$
and
$$
 \lim_{ \rho \rightarrow +0 }\frac{1}{m(B(\mathbf{0}, \rho  )\cap\Omega)} \int_{B(\mathbf{0}, \rho )\cap\Omega} \mathbf{v}_\sfr(\mathbf{x})  {\rm d} \mathbf{x} = \lim_{ \rho \rightarrow +0 }\frac{1}{m(B(\mathbf{0}, \rho  )\cap\Omega)} \int_{B(\mathbf{0}, \rho )\cap\Omega} \mathbf{w}_\sfr(\mathbf{x})  {\rm d} \mathbf{x},
$$
we can finish the proof of this corollary by using \eqref{eq:294 vrj}.
	\end{proof}

Finally, we establish the vanishing property under the H\"older regularity of the transmission eigenfunctions, which shall be useful for our study of the inverse elastic problem in Section~\ref{sect:4}. 

\begin{thm}\label{thm:23}
Let $(\mathbf{v}, \mathbf{w})\in H^{1}(\Omega)^{2}\times H^{1}(\Omega)^{2}$ be a pair of eigenfunctions to \eqref{eq:lame1} associated with $\omega\in\mathbb{R}_{+}$. Let $W, S_h$ and $q$ be those described in Theorem~\ref{thm:29}. Suppose that $\mathbf{v}\in C^{\alpha}(\overline{S_{h}})^2$, $q\mathbf{w}\in C^{\alpha}(\overline{S_{h}})^2$, $\eta\in C^{\alpha}(\overline{\Gamma_{h}^{\pm}})$ with $\eta(\mathbf{0})\neq 0$, for $\alpha\in (0, 1)$ and the corner $W$ is non-degenerate. Then we have
\begin{equation}\label{eq:ccll1}
\mathbf{v}(\mathbf{0})=\mathbf{0}. 
\end{equation}

\end{thm}

\begin{proof}


It is sufficient for us to show that \eqref{eq:ccll1} holds for $\mathbf{v}_\sfr$. Similar to \eqref{eq:int1}, one can establish the following integral identity
\begin{equation}\label{cor3int}
 \begin{aligned}
&-\omega^{2}\int_{S_{h}}(\mathbf{v}_\sfr(\mathbf{x})-\mathbf {f}_\sfr (\mathbf{x}))\cdot\mathbf{u}(\mathbf{x}){\rm d}\mathbf{x}
=I_{\Lambda_h}
-\int_{\Gamma_{h}^{\pm}}\eta \mathbf{u}\cdot \mathbf{v}_\sfr{\rm d}\sigma,
 \end{aligned}
\end{equation}
where $I_{\Lambda_h }$ is defined in \eqref{eq:int1} and $\mathbf f_\sfr(\mathbf x)$ is given in \eqref{eq:fr}.

Since  $ \mathbf{f}_\sfr(\mathbf{0})\in C^\alpha(\overline {S}_h )$ and $\eta \in C^\alpha\left(\overline{\Gamma}_h^\pm \right)$,    we know that  $\eta$ and $\mathbf{f}_\sfr(\mathbf{x})$ have the expansions \eqref{eq:lame29} and \eqref{eq:fr} around the origin, respectively. Furthermore, due to the fact that  $\mathbf{v}\in C^\alpha(\overline { S}_h )^2$, we have the following expansions
\begin{equation}\label{vre}
 \begin{aligned}
&\mathbf{v}_\sfr(\mathbf{x})=\mathbf{v}_\sfr(\mathbf{0})+\delta\mathbf{v}_\sfr(\mathbf{x}),\quad |\delta\mathbf{v}_\sfr|\leq \|\mathbf{v}_\sfr(\mathbf{x})\|_{C^{\alpha}(\Omega)^{2}}| \mathbf{x}|^{\alpha}
.
 \end{aligned}
\end{equation}
Substituting \eqref{eq:lame29}, \eqref{eq:fr} and \eqref{vre} into \eqref{cor3int}, we can derive that
\begin{equation}\label{2.85}
 \begin{aligned}
&-\omega^{2}(\mathbf{v}_\sfr(\mathbf{0})-\mathbf {f}_\sfr (\mathbf{0}))\int_{S_{h}}\mathbf{u}(\mathbf{x}){\rm d}\mathbf{x}
-\omega^{2}\int_{S_{h}}(\delta\mathbf{v}_\sfr(\mathbf{x})-\delta\mathbf {f}_\sfr (\mathbf{x}))\cdot\mathbf{u}(\mathbf{x}){\rm d}\mathbf{x}
\\&\quad =I_{\Lambda_h }
-\eta(\mathbf{0})\mathbf{v}_\sfr(\mathbf{0})\int_{\Gamma_{h}^{\pm}}\mathbf{u}(\mathbf{x}){\rm d}\sigma-\eta(\mathbf{0})\int_{\Gamma_{h}^{\pm}}\mathbf{u}(\mathbf{x})\cdot\delta\mathbf{v}_\sfr(\mathbf{x}){\rm d}\sigma
\\
&\quad -\mathbf{v}_\sfr(\mathbf{0})\int_{\Gamma_{h}^{\pm}}\delta\eta(\mathbf{x})\mathbf{u}(\mathbf{x}){\rm d}\sigma -\int_{\Gamma_{h}^{\pm}}\delta\eta(\mathbf{x})\mathbf{u}(\mathbf{x})\cdot\delta\mathbf{v}_\sfr(\mathbf{x}){\rm d}\sigma.
\end{aligned}
\end{equation}
From \eqref{2.85}, after rearranging terms, we have
\begin{equation}\label{2.93}
 \begin{aligned}
&\eta(\mathbf{0})\mathbf{v}_\sfr(\mathbf{0})\int_{\Gamma_{h}^{\pm}}\mathbf{u}{\rm d}\sigma
\\ =&-\omega^{2}(\mathbf{v}_\sfr(\mathbf{0})-\mathbf{f}_\sfr(\mathbf{0}))\int_{S_{h}}\mathbf{u}(\mathbf{x}){\rm d}\mathbf{x}
-\int_{S_{h}}\omega^{2}(\delta\mathbf{v}_\sfr(\mathbf{x})-\delta\mathbf{f}_\sfr(\mathbf{x}))\cdot\mathbf{u}(\mathbf{x}){\rm d}\mathbf{x}
\\&+I_{\Lambda_h}
-\mathbf{v}_\sfr(\mathbf{0})\int_{\Gamma_{h}^{\pm}}\delta\eta(\mathbf{x})\mathbf{u}(\mathbf{x}){\rm d}\sigma-\int_{\Gamma_{h}^{\pm}}\delta\eta(\mathbf{x})\mathbf{u}(\mathbf{x})\cdot\delta\mathbf{v}_\sfr(\mathbf{x}){\rm d}\sigma
\\&-\eta(\mathbf{0})\int_{\Gamma_{h}^{\pm}}\mathbf{u}(\mathbf{x})\cdot\delta\mathbf{v}_\sfr(\mathbf{x}){\rm d}\sigma.
\end{aligned}
\end{equation}
Using \eqref{eq:I311} in Lemma \ref{lem:u0 int}, we have
 \begin{align}\label{2.86}
\int_{\Gamma_{h}^{-}}\mathbf{u}{\rm d}\sigma&=
\begin{pmatrix}
    1\\
   {\rm i}
\end{pmatrix}
\int_{0}^{h}e^{-s\sqrt{r}\mu(\theta_{m})}{\rm d}r
\\&=
\begin{pmatrix}
    1\\
   {\rm i}
\end{pmatrix}
2 s^{-2} \bigg(\mu(\theta_m )^{-2}-   \mu(\theta_m )^{-2} e^{ -s\sqrt{h} \mu(\theta_m ) } -  \mu(\theta_m )^{-1} s\sqrt{h}   e^{ -s\sqrt{h} \mu(\theta_m ) }\bigg), \notag
\end{align}
and
 \begin{align}\label{2.87}
\int_{\Gamma_{h}^{+}}\mathbf{u}{\rm d}\sigma&=
\begin{pmatrix}
    1\\
   {\rm i}
\end{pmatrix}
\int_{0}^{h}e^{-s\sqrt{r}\mu(\theta_{M})}{\rm d}r
\\&=
\begin{pmatrix}
    1\\
   {\rm i}
\end{pmatrix}
2 s^{-2} \bigg(\mu(\theta_M )^{-2}-   \mu(\theta_M )^{-2} e^{ -s\sqrt{h} \mu(\theta_M ) } -  \mu(\theta_M )^{-1} s\sqrt{h}   e^{ -s\sqrt{h} \mu(\theta_M ) }\bigg), \notag
\end{align}
where $\mu(\theta)$ is defined in \eqref{eq:I311}.
By virtue of \eqref{eq:zeta}, using \eqref{eq:lame29} and \eqref{vre} we  have the estimates
 \begin{align}
 \left|\int_{\Gamma_{h}^{-}}\mathbf{u}(\mathbf{x})\cdot\delta\mathbf{v}_\sfr(\mathbf{x}){\rm d}\sigma\right|
 & \leq\sqrt{2}\|\mathbf{v}^\sfr\|_{C^{\alpha}}\int_{0}^{h}r^{\alpha}e^{-s\sqrt{r}\cos\frac{\theta_{m}}{2}}{\rm d}r=O(s^{-2\alpha-2}),\notag \\
\left|\int_{\Gamma_{h}^{-}}\delta\eta(\mathbf{x})\mathbf{u}(\mathbf{x}){\rm d}\sigma\right|
&\leq\sqrt{2}\|\eta\|_{C^{\alpha}}\int_{0}^{h}r^{\alpha}e^{-s\sqrt{r}\cos\frac{\theta_{m}}{2}}{\rm d}r=O(s^{-2\alpha-2}), \label{2.88} \\
\left|\int_{\Gamma_{h}^{-}}\delta\eta(\mathbf{x})\mathbf{u}(\mathbf{x})\cdot\delta\mathbf{v}_\sfr(\mathbf{x}){\rm d}\sigma\right|
 &\leq\sqrt{2}\|\mathbf{v}_\sfr\|_{C^{\alpha}}\|\eta\|_{C^{\alpha}}\int_{0}^{h}r^{2\alpha}e^{-s\sqrt{r}\cos\frac{\theta_{m}}{2}}{\rm d}r=O(s^{-4\alpha-2}), \notag
 \end{align}
and
 \begin{align}\label{2.89}
\left|\int_{\Gamma_{h}^{+}}\mathbf{u}(\mathbf{x})\cdot\delta\mathbf{v}_\sfr(\mathbf{x}){\rm d}\sigma\right| \notag
&\leq\sqrt{2}\|\mathbf{v}^\sfr\|_{C^{\alpha}}\int_{0}^{h}r^{\alpha}e^{-s\sqrt{r}\cos\frac{\theta_{M}}{2}}{\rm d}r=O(s^{-2\alpha-2}),\\
\left|\int_{\Gamma_{h}^{+}}\delta\eta(\mathbf{x})\mathbf{u}(\mathbf{x}){\rm d}\sigma\right|
&\leq\sqrt{2}\|\eta\|_{C^{\alpha}}\int_{0}^{h}r^{\alpha}e^{-s\sqrt{r}\cos\frac{\theta_{M}}{2}}{\rm d}r=O(s^{-2\alpha-2}), \\
\left|\int_{\Gamma_{h}^{+}}\delta\eta(\mathbf{x})\mathbf{u}(\mathbf{x})\cdot\delta\mathbf{v}_\sfr(\mathbf{x}){\rm d}\sigma\right|
 &\leq\sqrt{2}\|\mathbf{v}_\sfr\|_{C^{\alpha}}\|\eta\|_{C^{\alpha}}\int_{0}^{h}r^{2\alpha}e^{-s\sqrt{r}\cos\frac{\theta_{M}}{2}}{\rm d}r=O(s^{-4\alpha-2}), \notag
  \end{align}
as $s\rightarrow +\infty$.

Similarly, using \eqref{eq:lame6}, \eqref{eq:fr} and \eqref{vre}, one has
\begin{equation}\label{2.91}
 \begin{aligned}
\left|\int_{S_{h}}\delta\mathbf{v}_\sfr(\mathbf{x})\cdot\mathbf{u}(\mathbf{x}){\rm d}\mathbf{x}\right|
\leq&\sqrt{2}\| \mathbf{v}_\sfr\|_{C^{\alpha}}
\int_{W} |u_{1}( \mathbf{x })||\mathbf{x}|^{\alpha}{\rm d}\mathbf{x}
\\\leq&2\sqrt{2}\| \mathbf{v}_\sfr\|_{C^{\alpha}}\frac{(\theta_{M}-\theta_{m})\Gamma(2\alpha+4)}{\delta_{W}^{2\alpha+4}}s^{-2\alpha-4},\\
\left|\int_{S_{h}}\delta\mathbf{f}_\sfr(\mathbf{x})\cdot\mathbf{u}(\mathbf{x}){\rm d}\mathbf{x}\right|
\leq&\sqrt{2}\| \mathbf{f}_\sfr\|_{C^{\alpha}}
\int_{W} |u_{1}( \mathbf{x })||\mathbf{x}|^{\alpha}{\rm d}\mathbf{x}
\\\leq&2\sqrt{2}\| \mathbf{f}_\sfr\|_{C^{\alpha}}\frac{(\theta_{M}-\theta_{m})\Gamma(2\alpha+4)}{\delta_{W}^{2\alpha+4}}s^{-2\alpha-4},
\end{aligned}
\end{equation}
as $s\rightarrow +\infty$.

Multiplying $s^{2}$ on the both side of \eqref{2.93}, by virtue of \eqref{e}, \eqref{2.86}, \eqref{2.87}, \eqref{2.88}, \eqref{2.89}, and \eqref{2.91}, and letting $s\rightarrow +\infty$, we have
\begin{equation}\label{eq:2110}
 \begin{aligned}
\eta(\mathbf{0})
\mathbf{v}_\sfr(\mathbf{0})\cdot
\begin{pmatrix}
    1\\
   {\rm i}
\end{pmatrix}
\bigg(\mu(\theta_M )^{-2}+   \mu(\theta_m )^{-2}\bigg)=0.
 \end{aligned}
\end{equation}
Since the corner $W$ is non-degenerate, namely $\theta_M-\theta_m\neq \pi$, from \cite[Lemma 2.10]{DCL}, we know that
\begin{equation}\notag
\mu^{-2}(\theta_{M})+\mu^{-2}(\theta_{m})\neq 0.
\end{equation}
Finally, by noting $\eta(\mathbf{0})\neq \mathbf{0}$, it is easy to infer from \eqref{eq:2110} that $\mathbf{v}_\sfr(\mathbf{0})=\mathbf{0}$.

The proof is complete. 
\end{proof}

\begin{rem}
{ Similar to Corollary  \ref{cor:22}, under the same setup of Theorem \ref{thm:23}, for the degenerate case of \eqref{eq:lame1} with $\eta \equiv 0$ in \eqref{eq:lame1}, if $V$ is H\"older-regular near the corner and $V(\mathbf 0) \neq 0$, one can prove that  $\mathbf{v}(\mathbf{0})=\mathbf{0}$, where $\mathbf{v}\in H^{1}(\Omega)^{2} \cap C^\alpha(\overline{S_h} )$ and $\mathbf{w}\in H^{1}(\Omega)^{2}$ are  a pair of the generalized elastic transmssion eigenfunction to \eqref{eq:lame1} associated with $\omega\in\mathbb{R}_{+}$ such that $q\mathbf w \in C^\alpha(\overline{S_h})$. We choose not to discuss the details in this paper. In \cite[Theorem 1.5]{EBL}, the corresponding vanishing property was established when $\mathbf v-\mathbf w \in H^2(\Omega )^2$,  $V$ is  H\"older continuous around the corner such that$V(\mathbf 0) \neq 0$, and  either $\mathbf v$ or $\mathbf w$ is H\"older continuous around the corner. Compared with \cite[Theorem 1.5]{EBL}, the assumption $\mathbf v-\mathbf w \in H^2(\Omega )^2$ can be removed in our setting.   }
\end{rem}

\section{vanishing near corners of  generalized elastic transmission eigenfunctions: three-dimensional case}\label{sec:3}

	In this section, we establish the vanishing property of the generalized elastic transmission eigenfunctions for the 3D case.  In principle, we could also consider a generic corner in the usual sense as the one for the 2D case. However, we consider  a 3D corner described by $S_h \times (-M,M)$, where $S_h$ is defined in \eqref{eq:SIGN} and $M\in\mathbb{R}_+$. The 3D corner $S_h \times (-M,M)$ is a more general corner geometry in 3D. It is readily seen that  $S_h \times (-M,M)$  actually describes an edge singularity. In what follows, we suppose that the Lipschitz domain $\Omega\subset\mathbb{R}^3$ with $\mathbf{0}\in \partial \Omega$ possesses a 3D corner.  Let $\mathbf{0} \in \mathbb{R}^{2}$ be the vertex of $S_h$ and $ x_3 \in (-M,M)$. Then  $(\mathbf{0},x_{3} )$ is defined as an edge point of $S_h \times (-M,M)$.


In order to make use of the CGO solution $\mathbf u(\mathbf x)$ introduced in Lemma \ref{lem:22} to study the vanishing property of $(\mathbf v, \mathbf w)$ to \eqref{eq:lame1} at a 3D corner, we define the following dimension reduction operator.
\begin{defn}\label{Def}
Let ${S_h}\subset \mathbb{R}^{2}$ be defined in \eqref{eq:SIGN}, $M>0$. For a given function $\mathbf{g}$ in the domain $S_h \times (-M,M )$.  Pick up any point $x_{3} \in (-M, M)$. Suppose that $\phi \in C_0^{\infty}( (x_{3}-L, x_{3}+  L) )$ is  a nonnegative function and $\phi\not\equiv0$, where $L$  is   sufficiently small  such that $ (x_{3}-L, x_{3}+  L) \subset   (-M,M)$, and write $\mathbf{x}=(\mathbf{x}', x_3) \in \mathbb{R}^{3}, \,\mathbf{ x}'\in \mathbb{R}^{2}$.  The dimension reduction operator $\mathcal{R}$ is defined by
	\begin{equation}\label{rg.}
	\mathcal{R}(\mathbf{g})(\mathbf{x}')=\int_{{x}_{3} -L}^{{x}_{3}+ L} \phi(x_3)\mathbf{g}(\mathbf{x}',x_3){\rm d}x_3,
	\end{equation}
	where $\mathbf{x}'\in {S}_h$.
\end{defn}

Before presenting the main results of this section, we first analyze the regularity of the functions after applying the dimension reduction operator.  Using a similar argument of \cite[Lemma 3.4]{Bsource}, we can prove the following lemma, whose detailed proof is omitted.

\begin{lem}\label{lem:31}
	Let  $\mathbf{g}\in H^m({S}_h\times(-M ,M))^3$, $m=1,2$. Then
	\begin{equation}\notag
	{\mathcal R}(\mathbf{g})(\mathbf{x}') \in H^m({S}_h)^3.
    \end{equation}
	Similarly, if $\mathbf{g}\in   C^\alpha (\overline {S_h}\times [-M,M])^3$, $0<\alpha<1$, then
	\begin{equation}\notag
	{\mathcal R}(\mathbf{g})(\mathbf{x}') \in C^\alpha (\overline{S_h})^3.
    \end{equation}
\end{lem}

The elastic Herglotz wave function $\mathbf{v}_{\mathbf g}$ in $\mathbb R^3$ is defined by
\begin{equation} \label{eq:h3d}
 \begin{aligned}
\mathbf{v}_{\mathbf g}(\mathbf x)=\int_{\mathbb S^{2}} \{e^{\mathrm ik_{p}\mathbf{x}\cdot \mathbf{d}}\mathbf{g}_{p}(\mathbf{d})+e^{\mathrm i k_{s}\mathbf{x}\cdot \mathbf{d}}\mathbf{g}_{s}(\mathbf{d})\}{\rm d}\sigma (\mathbf{d}),
 \end{aligned}
 \end{equation}
where the kernel $\mathbf{g}=\mathbf{g}_{p}+\mathbf{g}_{s}$ with $\mathbf{g}_{p}\in L^{2}(\mathbb{S}^{2})^{3}$ and $\mathbf{g}_{s}\in L^{2}(\mathbb{S}^{2})^{3}$ and ${\mathbf d} \in \mathbb S^2$. Lemma~\ref{lem:herg} holds equally in the three dimensions. In view of this lemma, for any pair of the generalized elastic transmission eigenfunction $(\mathbf v, \mathbf w)$ to \eqref{eq:lame1}, there exits a sequence Herglotz wave function $\{\mathbf v_j\}_{j=1}^{+\infty }$ defined by
\begin{equation}\label{eq:vj3}
 \begin{aligned}
\mathbf{v}_j(\mathbf{x})=\int_{S^{2}} \{e^{ik_{p}\mathbf{x}\cdot \mathbf{d}}\mathbf{g}_{jp}(\mathbf{d})+e^{ik_{s}\mathbf{x}\cdot \mathbf{d}}\mathbf{g}_{js}(\mathbf{d})\}{\rm d}\sigma (\mathbf{d}),
 \end{aligned}
\end{equation}
 where $\mathbf{g}_j=\mathbf{g}_{jp}+\mathbf{g}_{js}$ with $\mathbf{g}_{jp}\in L^{2}(\mathbb{S}^{2})^{3}$ and $\mathbf{g}_{js}\in L^{2}(\mathbb{S}^{2})^{3}$, can approximate $\mathbf v$ to an arbitrary accuracy  in $H ^{1}(\Omega)^{3}$. Henceforth, we let the real and imaginary parts of the kernel functions $\mathbf{g}_{j\beta }({\mathbf d})$ ($\beta=p,s$) in \eqref{eq:vj3} be defined by
\begin{equation}\label{eq:gjpgjs3}
	\mathbf{g}_{j\beta }({\mathbf d})=\mathbf{g}_{\sfr, j\beta }({\mathbf d})+\mathrm i \mathbf{g}_{\sfi, j\beta }({\mathbf d}),
\end{equation}
where
\begin{equation}\label{g3}
 \begin{aligned}
&\mathbf{g}_{\sfr, jp }({\mathbf d})=
\begin{pmatrix}
   \mathbf{g}_{\sfr, jp }^{(1,2)} ({\mathbf d}) \\
    g_{\sfr,jp }^{(3) }({\mathbf d})
\end{pmatrix} \in \mathbb R^3, \
\mathbf{g}_{\sfr, js }({\mathbf d})=
\begin{pmatrix}
   \mathbf{g}_{\sfr, js }^{(1,2)} ({\mathbf d}) \\
    g_{\sfr, js }^{(3) }({\mathbf d})
\end{pmatrix} \in \mathbb R^3,
\\
&\mathbf{g}_{\sfi, jp }({\mathbf d})=
\begin{pmatrix}
   \mathbf{g}_{\sfi, jp }^{(1,2)} ({\mathbf d}) \\
    g_{\sfi, jp }^{(3)} ({\mathbf d})
\end{pmatrix}\in \mathbb R^3, \
\mathbf{g}_{\sfi, js }({\mathbf d})=
\begin{pmatrix}
   \mathbf{g}_{\sfi, js }^{(1,2)}({\mathbf d}) \\
    g_{\sfi, js }^{(3)} ({\mathbf d})
\end{pmatrix}\in \mathbb R^3
 \end{aligned}
\end{equation}
with $\mathbf{g}_{\sfr, j\beta }^{(1,2)} ({\mathbf d}) \in \mathbb{R}^2$ and $\mathbf{g}_{\sfi, j\beta }^{(1,2)} ({\mathbf d}) \in \mathbb{R}^2$ .

Similar to Proposition \ref{Pro:2.2}, using Jacobi-Anger expansion we have the expansion of the real and imaginary part of $\mathbf{v}_j(\mathbf{x})$ defined in \eqref{eq:vj3} as follows. 

\begin{prop}\label{Pro:31}
	Let the elastic Herglotz wave function $\mathbf v_{j}$ be defined by \eqref{eq:vj3}. Denote
	\begin{equation}\notag
\begin{aligned}
 & \mathbf v_{jp} ^\sfr(\mathbf 0)= \int_{\mathbb S^2} g_{jp}^\sfr (\mathbf{d}){\rm d}\sigma(\mathbf{d}),\
 \mathbf v_{js} ^\sfr(\mathbf 0)=\int_{\mathbb S^2} g_{js}^\sfr (\mathbf{d}){\rm d}\sigma(\mathbf{d}),\\&
  \mathbf v_{jp} ^\sfi(\mathbf 0)= \int_{\mathbb S^2} g_{jp}^\sfi (\mathbf{d}){\rm d}\sigma(\mathbf{d}),  \
   \mathbf v_{js} ^\sfi(\mathbf 0)=\int_{\mathbb S^2} g_{js}^\sfi (\mathbf{d}){\rm d}\sigma(\mathbf{d}).
\end{aligned}
\end{equation}
Then
\begin{equation}\label{0}
 \begin{aligned}
	\mathbf v_{j} (\mathbf 0)&=
\mathbf v_{j} ^\sfr(\mathbf 0)+{\rm i}\mathbf v_{j}^\sfi (\mathbf 0)
=(\mathbf v_{jp} ^\sfr(\mathbf 0)+\mathbf v_{js} ^\sfr(\mathbf 0))+{\rm i}(\mathbf v_{jp}^\sfi (\mathbf 0)+\mathbf v_{js}^\sfi (\mathbf 0)).
  \end{aligned}
\end{equation}	
where
 $g_{j\beta }^{\sfr }({\mathbf d})$ and $ g_{j\beta }^{\sfi }({\mathbf d})$ ($\beta=p,s$) are defined in \eqref{eq:gjpgjs3}. Let $j_\ell(t)$ be the $\ell$-th Bessel function for $\ell \in \mathbb N \cup \{0\}$.
 Furthermore, we have
\begin{equation}\label{eq:vr j 3}
 \begin{aligned}
  \mathbf v_{j} (\mathbf x)= \mathbf v_{j} ^\sfr(\mathbf x)+\rm i  \mathbf v_{j} ^\sfi(\mathbf x),
  \end{aligned}
\end{equation}
 where
 \begin{eqnarray}
 \mathbf v_{j} ^\sfr(\mathbf x)&=&j_{0}(k_p |\mathbf{x}|)\int_{\mathbb S^2} g_{jp}^\sfr (\mathbf{d}){\rm d}\sigma(\mathbf{d})
 +j_{0}(k_s |\mathbf{x}|)\int_{\mathbb S^2} g_{js}^\sfr (\mathbf{d}){\rm d}\sigma(\mathbf{d}) \label{vjrs}
 \\ &&+\sum^{+\infty}_{\ell =1} (-1)^\ell (4\ell+1) j_{2\ell}(k_p |\mathbf{x}|)\int_{\mathbb S^2} g_{jp}^\sfr (\mathbf{d})P_{2\ell}(\cos  \varphi){\rm d}\sigma(\mathbf{d}) \notag
 \\
 &&+\sum^{+\infty}_{\ell =1} (-1)^\ell (4\ell+1) j_{2\ell}(k_s |\mathbf{x}|)\int_{\mathbb S^2} g_{js}^\sfr (\mathbf{d}) P_{2\ell}(\cos  \varphi) {\rm d}\sigma(\mathbf{d}) \notag
 \\
 &&+\sum^{+\infty}_{\ell =1}  (-1)^\ell(4\ell-1) j_{2\ell-1}(k_p |\mathbf{x}|)\int_{\mathbb S^2} g_{jp}^\sfi (\mathbf{d})P_{2\ell-1}(\cos \varphi){\rm d}\sigma(\mathbf{d}) \notag
  \\
  &&+\sum^{+\infty}_{\ell =1}  (-1)^\ell (4\ell-1) j_{2\ell-1}(k_s |\mathbf{x}|)\int_{\mathbb S^2} g_{js}^\sfi (\mathbf{d})P_{2\ell-1}(\cos \varphi){\rm d}\sigma(\mathbf{d}), \notag \\
 \mathbf v_{j} ^\sfi(\mathbf x) &=&j_{0}(k_p |\mathbf{x}|)\int_{\mathbb S^2} g_{jp}^\sfi (\mathbf{d}){\rm d}\sigma(\mathbf{d})
 +j_{0}(k_s |\mathbf{x}|)\int_{\mathbb S^2} g_{js}^\sfi (\mathbf{d}){\rm d}\sigma(\mathbf{d}) \notag
  \\
  &&-\sum^{+\infty}_{\ell =1} (-1)^\ell (4\ell+1) j_{2\ell}(k_p |\mathbf{x}|)\int_{\mathbb S^2} g_{jp}^\sfr (\mathbf{d})P_{2\ell}(\cos  \varphi){\rm d}\sigma(\mathbf{d}) \notag
  \\
  &&-\sum^{+\infty}_{\ell =1} (-1)^\ell (4\ell+1) j_{2\ell}(k_s |\mathbf{x}|)\int_{\mathbb S^2} g_{js}^\sfr (\mathbf{d})P_{2\ell}(\cos \varphi){\rm d}\sigma(\mathbf{d}) \notag
 \\
 &&+\sum^{+\infty}_{\ell =1}  (-1)^\ell(4\ell-1) j_{2\ell-1}(k_p |\mathbf{x}|)\int_{\mathbb S^2} g_{jp}^\sfi (\mathbf{d})P_{2\ell-1}(\cos  \varphi){\rm d}\sigma(\mathbf{d}) \notag
  \\
  &&+\sum^{+\infty}_{\ell =1}  (-1)^\ell (4\ell-1) j_{2\ell-1}(k_s |\mathbf{x}|)\int_{\mathbb S^2} g_{js}^\sfi (\mathbf{d})P_{2\ell-1}(\cos  \varphi){\rm d}\sigma(\mathbf{d}), \notag
 \end{eqnarray}
and $\varphi $ is the angle between $\mathbf x$ and $ \mathbf d$ in \eqref{eq:vj3}.
\end{prop}

In view of the explicit expression $\mathbf{v}_j(\mathbf{x})$ defined in \eqref{eq:vj3}, using Euler formula, one can handily compute that
\begin{prop}
Let the elastic Herglotz wave function $\mathbf v_{j}$ be defined in \eqref{eq:vj3}, where $\mathbf v_{j} ^\sfr(\mathbf x) $ and $\mathbf v_{j}^\sfi (\mathbf x)
$ are the real and imaginary parts of  $\mathbf v_{j}$ respectively. Then it holds that
\begin{equation}\label{eq:vrj 313}
\begin{aligned}
\mathbf v_{j} ^\sfr(\mathbf x)=& \int_{\mathbb S^2}\bigg ( \cos (k_{p} \mathbf{d}\cdot \mathbf{x})\mathbf{g}_{jp}^\sfr ({\mathbf d})-\sin (k_{p} \mathbf{d}\cdot \mathbf{x})\mathbf{g}_{jp}^\sfi ({\mathbf d})+\cos (k_{s} \mathbf{d}\cdot \mathbf{x})\mathbf{g}_{js}^\sfr ({\mathbf d})\\&-\sin (k_{s} \mathbf{d}\cdot \mathbf{x})\mathbf{g}_{js}^\sfi ({\mathbf d}) \bigg ){\rm d}\sigma(\mathbf{d}),
\end{aligned}
\end{equation}
and
\begin{equation}\notag
\begin{aligned}
\mathbf v_{j} ^\sfi(\mathbf x)=& \int_{\mathbb S^2}\bigg (  \cos (k_{p} \mathbf{d}\cdot \mathbf{x})\mathbf{g}_{jp}^\sfi ({\mathbf d})  +\sin (k_{p} \mathbf{d}\cdot \mathbf{x})\mathbf{g}_{jp}^\sfr ({\mathbf d})+\cos (k_{s} \mathbf{d}\cdot \mathbf{x})\mathbf{g}_{js}^\sfi ({\mathbf d}) \\& +\sin (k_{s} \mathbf{d}\cdot \mathbf{x})\mathbf{g}_{js}^\sfr ({\mathbf d}) \bigg )
{\rm d}\sigma(\mathbf{d}).
\end{aligned}
\end{equation}
\end{prop}


Let ${S_h}\subset \mathbb{R}^{2}$ be defined in \eqref{eq:SIGN} and  $M>0$. For any fixed  $x_3 \in (-M,M)$   and $L>0$ defined in Definition \ref{Def}, we suppose that  $L $ is sufficiently small such that $(x_3-L,x_3+L) \subset (-M,M) $.  Write  $\mathbf{x}=(\mathbf{x}', x_3) \in \mathbb{R}^{3}, \, \mathbf{x}'\in \mathbb{R}^{2}$. In what follows, we consider the transmission eigenvalue problem for $\mathbf{v},\mathbf{w}\in H^1({S}_h\times(-M,M))^{3}$:
\begin{align}\label{eq:3DT}
\left\{
\begin{array}{l}
{\lambda \Delta {\mathbf v}+(\lambda+\mu) \nabla \nabla \cdot {\mathbf v}+\omega^{2} {\mathbf v}={\mathbf 0} }\hspace*{2cm}\ \mathbf{x}'\in S_{h},\ -M<x_{3}<M, \\[5pt]
{\lambda \Delta {\mathbf w}+(\lambda+\mu) \nabla \nabla \cdot {\mathbf w}+\omega^{2}(1+V) {\mathbf w}={\mathbf 0} }\hspace*{0.5cm}\ \mathbf{x}'\in S_{h},\ -M<x_{3}<M, \\[5pt]
{\mathbf w}={\mathbf v},\quad T_{\nu}{\mathbf v}+\eta {\mathbf v}=T_{\nu}{\mathbf w} \hspace*{3.2cm} \mathbf{x}'\in \Gamma_{h}^{\pm},\ -M<x_{3}<M,
  \end{array}
\right.
\end{align}
where $\Gamma_h^\pm $ are defined in \eqref{eq:SIGN}, $T_\nu$ is boundary traction operator to $\Gamma_h^\pm \times (-M,M)$,   $q\in L^\infty(S_h \times (-M,M)) $ defined in \eqref{eq:Tu def} and $\eta \in L^\infty(\Gamma_h^\pm \times (-M,M) )$ is independent of $x_3$. Similar to the 2D case, we let $(\mathbf{v}_\sfr, \mathbf{w}_\sfr)$ and $(\mathbf{v}_{\sf I}, \mathbf{w}_{\sf I})$ respectively signify the real and imaginary parts of $(\mathbf{v}, \mathbf{w})$, and both of them satisfy the Lam\'e system \eqref{eq:3DT}. We shall mainly focus on dealing with $(\mathbf{v}_\sfr, \mathbf{w}_\sfr)$ and all the results hold equally for $(\mathbf{v}_{\sf I}, \mathbf{w}_{\sf I})$, and hence $(\mathbf{v}, \mathbf{w})$.

Noting that $\eta$ in \eqref{eq:3DT} is independent of the $x_3$ variable, and by applying the reduction operator $\mathcal R$ defined in Definition \ref{Def}, one can show by direct verifications the following lemma. 

\begin{lem}\label{lemma3.2}
Denote
  \begin{equation}\notag
 \begin{aligned}
 \mathbf{G}_{1}(\mathbf{x}')=&- \omega^{2}\mathcal{R}(\mathbf{v}_\sfr)(\mathbf{x}')-\int^{L}_{-L}\phi''(x_{3})
\begin{pmatrix}
    \lambda v_{1} \\
    \lambda v_{2}\\
    (2\lambda+\mu)v_{3}
\end{pmatrix}
(\mathbf{x}',x_{3}){\rm d}x_{3}
\\&+(\lambda+\mu)\int^{L}_{-L}\phi'(x_{3})
\begin{pmatrix}
    \partial_{1} v_{3}          \\
     \partial_{2} v_{3}                              \\
    \partial_{1} v_{1}+\partial_{2}v_{2}
\end{pmatrix}
(\mathbf{x}',x_{3}){\rm d}x_{3},
\\ \mathbf{G}_{2}(\mathbf{x}')=&-\omega^{2}\mathcal{R}(\mathbf{f}_\sfr)(\mathbf{x}')-\int^{L}_{-L}\phi''(x_{3})
\begin{pmatrix}
    \lambda w_{1} \\
    \lambda w_{2}\\
    (2\lambda+\mu)w_{3}
\end{pmatrix}
(\mathbf{x}',x_{3}){\rm d}x_{3}
 \\&+(\lambda+\mu)\int^{L}_{-L}\phi'(x_{3})
\begin{pmatrix}
    \partial_{1} w_{3}          \\
     \partial_{2} w_{3}                              \\
    \partial_{1} w_{1}+\partial_{2}w_{2}
\end{pmatrix}
(\mathbf{x}',x_{3}){\rm d}x_{3},
 \end{aligned}
  \end{equation}
where $\mathbf{f}_\sfr: = (1+V)\mathbf{w}_\sfr= q\mathbf{w}_\sfr$ and $\mathcal{R}$ is the dimension reduction operator associated with $\phi$ defined in Definition \ref{Def}. Denote
\begin{equation}\label{eq:defvw}
 \begin{aligned}
\mathbf{v}_\sfr=
\begin{pmatrix}
   v_{1}\\
    v_{2}\\
    v_{3}
\end{pmatrix}
=
\begin{pmatrix}
   \mathbf{ v}^{(1,2)} \\
    v_{3}
\end{pmatrix}  \in \mathbb R^3 , \quad
\mathbf{w}_\sfr=
\begin{pmatrix}
   w_{1}\\
    w_{2}\\
    w_{3}
\end{pmatrix}=
\begin{pmatrix}
    \mathbf{w}^{(1,2)} \\
    w_{3}
\end{pmatrix}
\in \mathbb R^3 ,
 \end{aligned}
\end{equation}
where $  \mathbf{ v}^{(1,2)} \in \mathbb{R}^2$ and  $  \mathbf{ w}^{(1,2)}\in \mathbb{R}^2$. Assume that $\eta  \in L^\infty(\Gamma_h^\pm \times (-M,M) )$ in \eqref{eq:3DT} is independent of $x_3$. Then it holds that
  \begin{align}\label{1}
\left\{
\begin{array}{l}
{\widetilde{ \mathcal{L}}\mathcal{R}(\mathbf{v}_\sfr)(\mathbf{x}') =  \mathbf{G}_{1}(\mathbf{x}')}\hspace*{7.5cm}\ \mbox{ in } S_{h}, \\[5pt]
{\widetilde{ \mathcal{L}}\mathcal{R}(\mathbf{w}_\sfr)(\mathbf{x}') =  \mathbf{G}_{2}(\mathbf{x}') }\hspace*{7.5cm}\ \mbox{ in } S_{h}, \\[5pt]
\mathcal{R}({\mathbf w}_\sfr)(\mathbf{x}') =\mathcal{R}({\mathbf v}_\sfr)(\mathbf{x}'), \,   \hspace*{7.3cm} \mbox{on } \Gamma_{h}^{\pm},\\[5pt]
\begin{bmatrix}
	T_{\nu}\mathcal{R}({\mathbf v}^{(1,2)})+\lambda  \mathcal{R} (\partial_3 v_3)\nu \\
	\mu \partial_\nu \mathcal{R}(v_3)+\mu \begin{bmatrix}\mathcal{R}(\partial_3 v_1 )\\ \mathcal{R}(\partial_3 v_2 )\end{bmatrix}\nu
\end{bmatrix}  +\eta\mathcal{R} ({\mathbf v}_\sfr) =\begin{bmatrix}
	T_{\nu}\mathcal{R}({\mathbf w}^{(1,2)})+\lambda  \mathcal{R} (\partial_3 w_3)\nu \\
	\mu \partial_\nu \mathcal{R}( w_3)+\mu \begin{bmatrix} \mathcal{R}(\partial_3 w_1) \\ \mathcal{R}(\partial_3 w_2)\end{bmatrix} \nu
\end{bmatrix} 
\mbox{on } \Gamma_{h}^{\pm}
  \end{array}
\right.
\end{align}
  in the distributional sense, where  $\nu$ signifies the exterior unit normal vector to $\Gamma_h^\pm $, $T_\nu$ is the two dimensional boundary traction operator defined in \eqref{eq:Tu def} and
 \begin{equation}\label{eq:l 321}
 \begin{aligned}
 \widetilde{ \mathcal{L}}=
\begin{pmatrix}
   \lambda\Delta'+(\lambda+\mu)\partial_{1}^{2} &(\lambda+\mu)\partial_{1}\partial_{2}&0\\
   (\lambda+\mu)\partial_{1}\partial_{2}& \lambda\Delta'+(\lambda+\mu)\partial_{2}^{2} &0\\
   0&0&\lambda\Delta'
\end{pmatrix} :=
\begin{pmatrix}
 \mathcal{L}&0\\
  0&\lambda\Delta'
\end{pmatrix}
   \end{aligned}
  \end{equation}
with $ \Delta':=\partial_{1}^{2} +\partial_{2}^{2} $  being the Laplace operator with respect to the $\mathbf{x}'$-variables. Here $\mathcal L$ is the two dimensional Lam\'e operator with respect to the $\mathbf{x}'$-variable. 
\end{lem}

%


\begin{lem}\label{lem:pde decouple}
Under the same setup in Lemma \ref{lemma3.2}, the PDE system \eqref{1} is equivalent to
 \begin{align}\label{eq:2.}
\left\{
\begin{array}{l}
{ \mathcal{L}\mathcal{R}(\mathbf{v}^{(1,2)})(\mathbf{x}') =  \mathbf{G}^{(1,2)}_{1}(\mathbf{x}')}\hspace*{4.6cm}\ \mbox{ in } S_{h}, \\[5pt]
{ \mathcal{L}\mathcal{R}(\mathbf{w}^{(1,2)})(\mathbf{x}') =  \mathbf{G}^{(1,2)}_{2}(\mathbf{x}') }\hspace*{4.5cm}\ \mbox{ in } S_{h}, \\[5pt]
{ \mathcal{R}({\mathbf w}^{(1,2)})(\mathbf{x}') =\mathcal{R}({\mathbf v}^{(1,2)})(\mathbf{x}'),   \hspace*{4.2cm}\mbox{ on } \Gamma_{h}^{\pm}},\\[5pt]
{ \ T_{\nu}\mathcal{R}({\mathbf v}^{(1,2)})(\mathbf{x}') +\eta \mathcal{R} ({\mathbf v}^{(1,2)})(\mathbf{x}') =T_{\nu} \mathcal{R}({\mathbf w}^{(1,2)})(\mathbf{x}')\hspace*{0.6cm}\mbox{ on } \Gamma_{h}^{\pm}},
  \end{array}
\right.
\end{align}
and
\begin{align}\label{3}
\left\{
\begin{array}{l}
{\lambda\Delta'\mathcal{R}(v_3)(\mathbf{x}') =  \mathbf{G}_{1}^{(3)}(\mathbf{x}')}\hspace*{6.8cm}\ \mbox{ in } S_{h}, \\[5pt]
{\lambda\Delta'\mathcal{R}(w_3)(\mathbf{x}') =  \mathbf{G}_{2}^{(3)}(\mathbf{x}') }\hspace*{6.7cm}\ \mbox{ in } S_{h}, \\[5pt]
\mathcal{R}(w_3)(\mathbf{x}') =\mathcal{R}(v_3)(\mathbf{x}'),  \partial_{\nu}\mathcal{R}(v_3)(\mathbf{x}') +\mu^{-1}\eta\mathcal{R} (v_3)(\mathbf{x}') =\partial_{\nu}\mathcal{R}(w_3)(\mathbf{x}')  \hspace*{0.03cm}
\mbox{on } \Gamma_{h}^{\pm},
  \end{array}
\right.
\end{align}
where
 \begin{eqnarray*}
 \mathbf{G}_{1}^{(1,2)}(\mathbf{x}')&=&- \omega^{2}\mathcal{R}(\mathbf{v}^{(1,2)})(\mathbf{x}')-\int^{L}_{-L}\phi''(x_{3})
\begin{pmatrix}
    \lambda v_{1} \\
    \lambda v_{2}
\end{pmatrix}
(\mathbf{x}',x_{3}){\rm d}x_{3}
\\&&+(\lambda+\mu)\int^{L}_{-L}\phi'(x_{3})
\begin{pmatrix}
    \partial_{1} v_{3}          \\
     \partial_{2} v_{3}
\end{pmatrix}
(\mathbf{x}',x_{3}){\rm d}x_{3},
\\ \mathbf{G}_{2}^{(1,2)}(\mathbf{x}') &=&-\omega^{2}\mathcal{R}(q\mathbf{w}^{(1,2)})(\mathbf{x}')-\int^{L}_{-L}\phi''(x_{3})
\begin{pmatrix}
    \lambda w_{1} \\
    \lambda w_{2}
\end{pmatrix}
(\mathbf{x}',x_{3}){\rm d}x_{3}
 \\ &&+(\lambda+\mu)\int^{L}_{-L}\phi'(x_{3})
\begin{pmatrix}
    \partial_{1} w_{3}          \\
     \partial_{2} w_{3}
\end{pmatrix}
(\mathbf{x}',x_{3}){\rm d}x_{3},\\
 \mathbf{G}_{1}^{(3)}(\mathbf{x}')&=&- \omega^{2}\mathcal{R}({v}_3)(\mathbf{x}')-(2\lambda+\mu)\int^{L}_{-L}\phi''(x_{3})
    v_{3}
(\mathbf{x}',x_{3}){\rm d}x_{3}
\\&&+(\lambda+\mu)\int^{L}_{-L}\phi'(x_{3})
    (\partial_{1} v_{1}+\partial_{2}v_{2})
(\mathbf{x}',x_{3}){\rm d}x_{3},
\\ \mathbf{G}_{2}^{(3)}(\mathbf{x}') &=&-\omega^{2}\mathcal{R}(q{w}_3)(\mathbf{x}')-(2\lambda+\mu)\int^{L}_{-L}\phi''(x_{3})
    w_{3}
(\mathbf{x}',x_{3}){\rm d}x_{3}
 \\&&+(\lambda+\mu)\int^{L}_{-L}\phi'(x_{3})
   ( \partial_{1} w_{1}+\partial_{2}w_{2})
(\mathbf{x}',x_{3}){\rm d}x_{3}.
  \end{eqnarray*}
\end{lem}

\begin{proof}
Since $\mathbf v_\sfr =\mathbf w_\sfr $ on $\Gamma_h^\pm \times [-L,L]$, using \eqref{rg.}, it can be directly deduced that
\begin{equation}\label{eq:324 boundary}
	\mathcal R(\partial_3 ( \mathbf v_\sfr-\mathbf w_\sfr))=-\int_{-L}^L \phi'(x_3) ( \mathbf v_\sfr-\mathbf w_\sfr)=\mathbf 0 \quad \mbox{ on } \Gamma_h^\pm.
\end{equation}
From \eqref{1}, by virtue of \eqref{eq:l 321}, \eqref{eq:defvw} and \eqref{eq:324 boundary}, together with straightforward calculations, one can obtain \eqref{eq:2.} and \eqref{3} respectively.
\end{proof}

Next we mainly study the system \eqref{eq:2.}.
\begin{lem}\label{lem34}
Let $S_h$, $ \Lambda_{h}$ and $\Gamma_h^\pm$ be defined in \eqref{eq:SIGN}. 		Suppose that $ \mathbf{v}^{(1,2)}, \mathbf{w}^{(1,2)}\in H^1({S}_h\times(-M,M))^{2}$ fulfill  \eqref{eq:2.}. Recall that the CGO solution $\mathbf{u}(\mathbf{x})$  is defined in \eqref{eq:lame3} and $P_n(t)$ is the Legendre polynomial. Let $\beta=p \mbox{ or } s$, $j_{\ell,\beta}=j_{\ell}(k_\beta |\mathbf{x}|)$ and 
\begin{eqnarray*}
&&\mathbf{v}_{jp}^{(1,2)}(\mathbf{0})=\int_{\mathbb S^2}   \mathbf{g}_{\sfr, jp}^{(1,2)}({\mathbf d}) {\rm d}\sigma(\mathbf{d}), \ \mathbf{v}_{js}^{(1,2)}(\mathbf{0})=\int_{\mathbb S^2}   \mathbf{g}_{\sfr, js}^{(1,2)}({\mathbf d}) {\rm d}\sigma(\mathbf{d}), \\
& &\mathbf{B}_{j\beta,1}^{(\ell)} =\int_{\mathbb S^2} \mathbf{g}_{\sfr, j\beta }^{(1,2)}({\mathbf d}) P_{2\ell}(\cos  \varphi )  {\rm d}\sigma(\mathbf{d}),
\quad \mathbf{B}_{j\beta,2}^{(\ell)} =\int_{\mathbb S^2}    \mathbf{g}_{\sfi, j\beta}^{(1,2)}({\mathbf d}) P_{2\ell-1}(\cos  \varphi){\rm d}\sigma(\mathbf{d}). 
\end{eqnarray*}
 Denote
 $$\mathbf{v}_j^{\sfr}(\mathbf{x})=
 \begin{pmatrix}
     \mathbf{v}_{j}^{(1,2)}(\mathbf{x}) \\
   v_{j}^{(3)}(\mathbf{x})
\end{pmatrix} \in \mathbb R^3,\quad \mathbf{v}_{j}^{(1,2)}(\mathbf{x}) \in \mathbb R^2,
$$
where
 \begin{align}
 \mathbf{v}^{(1,2)}_j(\mathbf{x})=&\mathbf{v}_{jp}^{(1,2)}(\mathbf{0})j_{0}(k_{p}|\mathbf{x}|)+\mathbf{v}_{js}^{(1,2)}(\mathbf{0})j_{0}(k_{s}|\mathbf{x}|)
+ \sum_{\beta=p,\, s}\sum^{+\infty}_{\ell =1} (-1)^\ell (4\ell+1) j_{2\ell,\beta}\mathbf{B}_{j\beta,1}^{(\ell)} \notag
\\
&+2 \sum_{\beta=p,\, s} \sum^{+\infty}_{\ell =1}  (-1)^\ell  (4\ell-1)j_{2\ell-1,\beta}\mathbf{B}_{j\beta,2}^{(\ell)} , 
\label{eq:vj12 357}
\end{align}
and
\begin{equation}\notag
 \begin{aligned}
  v_{j} ^{(3)}(\mathbf x)=&j_{0}(k_p |\mathbf{x}|)\int_{\mathbb S^2} g_{\sfr ,jp}  ^{(3)}(\mathbf{d}){\rm d}\sigma(\mathbf{d})
 +j_{0}(k_s |\mathbf{x}|)\int_{\mathbb S^2} g_{\sfr ,js}  ^{(3)} (\mathbf{d}){\rm d}\sigma(\mathbf{d})
 \\&+\sum_{\beta=p,\, s } \sum^{+\infty}_{\ell =1} (-1)^\ell (4\ell+1) j_{2\ell}(k_\beta  |\mathbf{x}|)\int_{\mathbb S^2} g_{\sfr ,j\beta }  ^{(3)}(\mathbf{d})P_{2\ell}(\cos \varphi){\rm d}\sigma(\mathbf{d})
 \\&+\sum_{\beta=p,\, s } \sum^{+\infty}_{\ell =1}  (-1)^\ell(4\ell-1) j_{2\ell-1}(k_\beta |\mathbf{x}|)\int_{\mathbb S^2} g_{\sfi ,j\beta }  ^{(3)} (\mathbf{d})P_{2\ell-1}(\cos \varphi){\rm d}\sigma(\mathbf{d})
 \end{aligned}
 \end{equation}
 with $\mathbf{g}_{\sfr ,j\beta}^{(1,2)}  $, $\mathbf{g}_{\sfi ,j\beta}^{(1,2)}$, $g_{\sfr ,j\beta}^{(3)}  $ and $g_{\sfi ,j\beta}^{(3)}$ defined in  \eqref{g3}. Then the following integral equality holds
\begin{equation}\label{4}
	\widetilde{I}_1+\widetilde{I}_2=\widetilde{I}_{\Lambda_h }- \widetilde{I}_{\pm} -\widetilde{I}_{\pm}^\Delta,
\end{equation}
where
\begin{eqnarray}\label{eq:10}
& &\widetilde{I}_1=\int_{S_{h}}(\mathbf{f}_{1j}+\mathbf{f}_{2}+\mathbf{f}_{3}+\mathbf{f}_{4}) \cdot\mathbf{u}(\mathbf{x}'){\rm d}\mathbf{x}', \
 \widetilde{I}_2=\int_{S_{h}}(\mathbf{f}_{1}-\mathbf{f}_{1j}) \cdot\mathbf{u}(\mathbf{x}'){\rm d}\mathbf{x}',
 \\ && \widetilde{I}_{\Lambda_h }=\int_{\Lambda_{h}}(T_{\nu}\mathcal{R}(\mathbf{v}^{(1,2)}-\mathbf{w}^{(1,2)}))\cdot\mathbf{u}-(T_{\nu}(\mathbf{u}))\cdot \mathcal{R}(\mathbf{v}^{(1,2)}-\mathbf{w}^{(1,2)}){\rm d}\sigma, \notag
 \\ && \widetilde{I}_{\pm} =\int_{\Gamma^{\pm}_{h}}\eta(\mathbf{x}')\mathcal{R}(\mathbf{v}_{j}^{(1,2)})\cdot\mathbf{u}(\mathbf{x}'){\rm d}\sigma, \ \widetilde{I}_{\pm}^\Delta=\int_{\Gamma^{\pm}_{h}}\eta(\mathbf{x}')\mathcal{R}(\mathbf{v}^{(1,2)}-\mathbf{v}_{j}^{(1,2)} )\cdot\mathbf{u}(\mathbf{x}'){\rm d}\sigma, \notag
\\ && \mathbf{f}_{1}=-\omega ^{2}\mathcal{R}(\mathbf{v}^{(1,2)}), \
\mathbf{f}_{2}=\omega ^{2}\mathcal{R}(q\mathbf{w}^{(1,2)}),\
 \mathbf{f}_{3}=-\int^{L}_{-L}\phi''(x_{3})
\begin{pmatrix}
    \lambda (v_{1}-w_{1}) \\
    \lambda (v_{2}- w_{2})
\end{pmatrix}
(\mathbf{x}',x_{3}){\rm d}x_{3},\notag
\\
&& \mathbf{f}_{1j}=-\omega ^{2}\mathcal{R}(\mathbf{v}_{j}^{(1,2)}), \
 \mathbf{f}_{4}=(\lambda+\mu)\int^{L}_{-L}\phi'(x_{3})
\begin{pmatrix}
    \partial_{1} (v_{3}-  w_{3})          \\
     \partial_{2}( v_{3}   -  w_{3} )
\end{pmatrix}
(\mathbf{x}',x_{3}){\rm d}x_{3}. \notag
\end{eqnarray}
\end{lem}

\begin{proof}
Recall that $\mathcal{L}$ is defined in \eqref{eq:l 321}. Since $\mu>0$, we see that $3\lambda+2\mu >0$ implies  $3(\lambda+\mu )>0$ and hence also $2\lambda+2\mu>0 $, so $\mathcal{L}$ is strongly elliptic. Using Green's formula \eqref{eq:green1} on the domain $S_{h}$ 
together with ${\mathcal L} \mathbf{u}=\mathbf{0}$ in $S_h$, we have
\begin{equation}\label{eq:green}
 \begin{aligned}
 &\int_{S_{h}}\mathcal{L}\mathcal{R}(\mathbf{v}^{(1,2)}-\mathbf{w}^{(1,2)})\cdot\mathbf{u}(\mathbf{x}'){\rm d}\mathbf{x}'
 \\=&\int_{\Gamma^{\pm}_{h}\cup\Lambda_{h}}(T_{\nu}\mathcal{R}(\mathbf{v}^{(1,2)}-\mathbf{w}^{(1,2)}))\cdot\mathbf{u}-(T_{\nu}(\mathbf{u}))\cdot \mathcal{R}(\mathbf{v}^{(1,2)}-\mathbf{w}^{(1,2)}){\rm d}\sigma.
 \end{aligned}
\end{equation}
Using the boundary condition in \eqref{eq:2.}, it yields that
\begin{equation}\label{5}
 \begin{aligned}
 &\int_{\Gamma^{\pm}_{h}}(T_{\nu}\mathcal{R}(\mathbf{v}^{(1,2)}-\mathbf{w}^{(1,2)}))\cdot\mathbf{u}-(T_{\nu}(\mathbf{u}))\cdot \mathcal{R}(\mathbf{v}^{(1,2)}-\mathbf{w}^{(1,2)}){\rm d}\sigma
\\=& -\int_{\Gamma^{\pm}_{h}}\eta(\mathbf{x}')\mathcal{R}(\mathbf{v}^{(1,2)})\cdot\mathbf{u}(\mathbf{x}'){\rm d}\sigma
\\=&
-\int_{\Gamma^{\pm}_{h}}\eta(\mathbf{x}')\mathcal{R}(\mathbf{v}_{j}^{(1,2)})\cdot\mathbf{u}(\mathbf{x}'){\rm d}\sigma-\int_{\Gamma^{\pm}_{h}}\eta(\mathbf{x}')\mathcal{R}(\mathbf{v}^{(1,2)}-\mathbf{v}_{j}^{(1,2)} )\cdot\mathbf{u}(\mathbf{x}'){\rm d}\sigma.
\end{aligned}
\end{equation}
By virtue of  \eqref{eq:2.}, we have
\begin{equation}\label{6}
 \begin{aligned}
 &\int_{S_{h}}\mathcal{L}\mathcal{R}(\mathbf{v}^{(1,2)}-\mathbf{w}^{(1,2)})\cdot\mathbf{u}(\mathbf{x}'){\rm d}\mathbf{x}'\\=&\int_{S_{h}} (\mathbf{G}_{1}^{(1,2)}-\mathbf{G}_{2}^{(1,2)})\cdot\mathbf{u}(\mathbf{x}'){\rm d}\mathbf{x}'
=\int_{S_{h}}(\mathbf{f}_{1}+\mathbf{f}_{2}+\mathbf{f}_{3}+\mathbf{f}_{4}) \cdot\mathbf{u}(\mathbf{x}'){\rm d}\mathbf{x}'
 \\=&\int_{S_{h}}(\mathbf{f}_{1j}+\mathbf{f}_{2}+\mathbf{f}_{3}+\mathbf{f}_{4}) \cdot\mathbf{u}(\mathbf{x}'){\rm d}\mathbf{x}'
 +
 \int_{S_{h}}(\mathbf{f}_{1}-\mathbf{f}_{1j}) \cdot\mathbf{u}(\mathbf{x}'){\rm d}\mathbf{x}'.
 \end{aligned}
\end{equation}
From \eqref{eq:green}, \eqref{5} and \eqref{6}, we can derive \eqref{4}.
\end{proof}

Similar to Lemma \ref{lem:2.6}, for the integral  $\widetilde{I}_{\Lambda_{h}}$ defined in \eqref{4} one has
\begin{lem}\label{lem:3.5}
Recall that $ \widetilde{I}_{\Lambda_{h}}$ is defined in \eqref{eq:10}. Under the same setup in Lemma \ref{lem34}, the following integral estimate holds
\begin{equation}\label{e3}
| \widetilde{I}_{\Lambda_{h}}|\leq C \frac{\sqrt{ 2h+ s^2}+\mu s}{\sqrt 2}\sqrt {\theta_M-\theta_m }e^{-s \sqrt h \delta_W}\| \mathbf{v}^{(1,2)}-\mathbf{w}^{(1,2)} \|_{H^1(S_{h } )^{2}  },
\end{equation}
where $C$ is a positive constant coming from the trace  theorem, $S_{h }$ and $\delta_W>0$ are defined in \eqref{eq:SIGN} and  \eqref{eq:lame7}, respectively.
\end{lem}


\begin{lem}\label{lem35}
	Under the same setup in Lemma \ref{lem34}, for any given positive constants $\gamma$, $\beta_1$ and $\beta_2$, we assume that there exits a sequence of the Herglotz wave functions $\{\mathbf{v}_j\} _{j=1}^{+\infty} $, where $\mathbf{v}_j $ is defined by \eqref{eq:vj3}, can approximate  $\mathbf{v}$ in $H^1(S_h \times (-M,M))^{3}$  satisfying
\begin{equation}\label{eq:vj aprrox 3}
\| \mathbf{v}-\mathbf{v}_{j}\|_{H^{1}(S_{h} \times (-M,M))^{3}}\leq j^{-\gamma}, \
\| \mathbf{g}_{jp}\|_{L^{2}(\mathbb{S}^{2})^{3}}\leq j^{\beta_{1}}, \
\| \mathbf{g}_{js}\|_{L^{2}(\mathbb{S}^{2})^{3}}\leq j^{\beta_{2}}. 
\end{equation}
Furthermore, suppose that  $\eta \in C^\alpha (\overline{\Gamma_{h}^{\pm}}\times [-M,M] )$, $0<\alpha<1$, and hence it holds that
\begin{equation}\label{etae}
\eta(\mathbf{x}')=\eta(\mathbf{0})+\delta\eta(\mathbf{x}'), \
|\delta\eta(\mathbf{x}')|\leq\|\eta(\mathbf{x}')\|_{C^\alpha}|\mathbf{x}'|^\alpha
\end{equation}for $ 0<\alpha<1$.
Recall that $\widetilde{I}_2$ and $\widetilde{I}_{\pm}^\Delta$ are defined in \eqref{eq:10}. Then we have the following estimates
 \begin{align}
 \left |\widetilde{I}_2\right|
 &\leq  h\sqrt{2L(\theta_{M}-\theta_{m})}e^{-s\sqrt{\Theta}\delta_{W}} \|\phi\|_{L^\infty} j^{-\gamma},\label{eq:20}
 \\
 \left|  \widetilde{I}_{\pm}^\Delta\right|
& \leq  C\|\phi\|_{\infty}(|\eta(0)|h\sqrt{\theta_{M}-\theta_{m}}e^{-s\sqrt{\Theta}\delta_{W}}
   +\| \eta\|_{C^{\alpha}}\frac{2\sqrt{\theta_{M}-\theta_{m}\Gamma(4\alpha+4)}}{(2\delta_{W})^{2\alpha+2}}s^{-2\alpha-2})j^{-\gamma}, \notag
 \end{align}
where $C$ is a positive constants depending on $L$ and $h$.
\end{lem}

\begin{proof}
Clearly, $\mathbf{v}_\sfr$ can be approximated by $\{\mathbf{v}_j^{\sf R}\}_{j=1}^{+\infty} $ in the sense of \eqref{eq:vj aprrox 3}.  Therefore, by the definition of $\mathbf{v}_\sfr^{(1,2)}$  given by \eqref{eq:defvw} and $\mathbf{v}^{(1,2)}_j$ given by \eqref{vj12}, it can be directly seen that
\begin{equation}\label{eq:vj aprrox 3 new}
\| \mathbf{v}_\sfr^{(1,2)}-\mathbf{v}_{j}^{(1,2)}\|_{H^{1}(S_{h}\times (-L,L))^{2} }\leq \| \mathbf{v}_\sfr-\mathbf{v}_{j}^\sfr\|_{H^{1}(S_{h} \times (-M,M))^{3}} \leq j^{-\gamma}.
\end{equation}
By \eqref{eq:vj aprrox 3 new} and using Cauchy-Schwarz inequality,  we have
\begin{equation}\label{eq:127}
 \begin{aligned}
&\|\mathcal{R}(\mathbf{v}^{(1,2)})-\mathcal{R}(\mathbf{v}^{(1,2)}_{j} ) \|_{L^2(S_h)^{2}}^{2}\\
&\quad =\int_{S_h } \left| \int_{-L}^L  \phi(x_3) (\mathbf{v}^{(1,2)}(\mathbf{x}',x_3)-\mathbf{v}^{(1,2)}_j(\mathbf{x}',x_3))  {\rm d} x_3 \right|^2 {\rm d} \mathbf{x}'\\
&\quad \leq  2L \|\phi\|_{L^{\infty} }^2\|\mathbf{v}-\mathbf{v}_j\|_{L^2(S_h \times (-L,L))^{3}}^2.
 \end{aligned}
\end{equation}
By virtue of \eqref{eq:lame21} and \eqref{eq:127}, we have
\begin{equation}\notag 
 \begin{aligned}
 \left |\widetilde{I}_2\right|=&
\left |\int_{S_{h}}(\omega ^{2}\mathcal{R}(\mathbf{v}^{(1,2)})-\omega ^{2}\mathcal{R}(\mathbf{v}^{(1,2)}_{j}))\cdot\mathbf{u}(\mathbf{x}'){\rm d}\mathbf{x}'\right|
 \\\leq& \omega ^2 \| {\mathcal R}( \mathbf{v}^{(1,2)})-{\mathcal R}(\mathbf{v}^{(1,2)}_j) \|_{L^2(S_h )^{2}}  \|\mathbf{u}(\mathbf{x}')\|_{L^2(S_h)^{2}}
 \\\leq &Ch\sqrt{\theta_{M}-\theta_{m}}e^{-s\sqrt{\Theta}\delta_{W}} \|\phi\|_{L^\infty} j^{-\gamma}.
\end{aligned}
\end{equation}
Using Cauchy-Schwarz inequality and the trace theorem, according to  \eqref{etae}, we have
\begin{equation}\notag 
 \begin{aligned}
  \left|   \widetilde{I}_{\pm}^\Delta\right|\leq&|\eta(\mathbf{0})|\int_{\Gamma_{h}^{\pm}}| \mathbf{u}|| \mathcal{R}(\mathbf{v}^{(1,2)}-\mathbf{v}^{(1,2)}_{j})| {\rm d}\sigma+\| \eta\|_{C^{\alpha}}\int_{\Gamma_{h}^{\pm}}| \mathbf{x'}|^{\alpha}|\mathbf{u}|| \mathcal{ R}(\mathbf{v}^{(1,2)}-\mathbf{v}^{(1,2)}_{j})| {\rm d}\sigma
 \\\leq&|\eta(\mathbf{0})|\|\mathcal{R} (\mathbf{v}^{(1,2)}-\mathbf{v}^{(1,2)}_{j})\|_{H^{\frac{1}{2}}(\Gamma_{h}^{\pm})^{2}}\| \mathbf{u}\|_{H^{-\frac{1}{2}}(\Gamma_{h}^{\pm})^{2}}
 \\&+\| \eta\|_{C^{\alpha}}\| \mathcal{R}(\mathbf{v}^{(1,2)}-\mathbf{v}^{(1,2)}_{j})\|_{H^{\frac{1}{2}}(\Gamma_{h}^{\pm})^{2}}\|\| | \mathbf{x'}|^{\alpha} \mathbf{u}\|_{H^{-\frac{1}{2}}(\Gamma_{h}^{\pm})^{2}}
 \\\leq&|\eta(\mathbf{0})|\|\mathcal{R}( \mathbf{v}^{(1,2)}-\mathbf{v}^{(1,2)}_{j})\|_{H^{1}(S_{h})^{2}}\| \mathbf{u}\|_{L^{2}(S_{h})^{2}}
 \\&+\| \eta\|_{C^{\alpha}}\| \mathcal{R} (\mathbf{v}^{(1,2)}-\mathbf{v}_{j}^{(1,2)})\|_{H^{1}(S_{h})^{2}}\|\|| \mathbf{x}'|^{\alpha} \mathbf{u}\|_{L^{2}(S_{h})^{2}}
  \\\leq&C\|\phi\|_{L^{{\infty}}}\|\mathbf{v}^{(1,2)}-\mathbf{v}^{(1,2)}_{j}\|_{H^{1}(S_{h})^{2}}(|\eta(\mathbf 0)|  \| \mathbf{u}\|_{L^{2}(S_{h})^{2}}+\| \eta\|_{C^{\alpha}}\|| \mathbf{x}'|^{\alpha} \mathbf{u}\|_{L^{2}(S_{h})^{2}})
   \\\leq&C\|\phi\|_{L^{\infty}}(|\eta(\mathbf 0)|h\sqrt{\theta_{M}-\theta_{m}}e^{-s\sqrt{\Theta}\delta_{W}}
   +\| \eta\|_{C^{\alpha}}\frac{2\sqrt{\theta_{M}-\theta_{m}\Gamma(4\alpha+4)}}{(2\delta_{W})^{2\alpha+2}}s^{-2\alpha-2})j^{-\gamma},
 \end{aligned}
\end{equation}
where $C$ is a positive constant and the last second inequality comes from Lemma \ref{lem:31}.
\end{proof}

\begin{lem}\label{lem:c alpha}
Let  $\mathbf{f}_{2}(\mathbf{x}')$, $\mathbf{f}_{3}(\mathbf{x}')$ and $\mathbf{f}_{4}(\mathbf{x}')$ be defined in \eqref{eq:10}. If
	  \begin{equation}\label{eq:qw 3}
	  		  q\mathbf w_\sfr \in C^\alpha(\overline{S}_h\times [-M,M] ),
	  \end{equation}
 then $\mathbf{f}_{2}(\mathbf{x}') \in C^\alpha(S_h)$. Furthermore we assume that
\begin{equation}\label{eq:vr-wr}
	 \mathbf{v}_\sfr- \mathbf{w}_\sfr  \in C^{1,\alpha} (\overline{S}_h\times[-M,M])^{3},
\end{equation}
where $\alpha \in (0,1)$.  Then one has $\mathbf{f}_{3}(\mathbf{x}') \in C^\alpha(\overline{S}_h)$ and $\mathbf{f}_{4}(\mathbf{x}')\in C^\alpha(\overline{S}_h)$.
\end{lem}

\begin{proof}
According to \eqref{eq:vr-wr}, one handily computes that
\begin{eqnarray*}
\left| \mathbf{f}_{3}(\mathbf x')- \mathbf{f}_{3}(\mathbf y') \right|&=&|\lambda|\left|\int^{L}_{-L}\phi''(x_{3}) \left(
\begin{pmatrix}
     (v_{1}-w_{1}) \\
    (v_{2}- w_{2})
\end{pmatrix}
(\mathbf{y}',x_{3})-\begin{pmatrix}
     (v_{1}-w_{1}) \\
     (v_{2}- w_{2})
\end{pmatrix}(\mathbf{x}',x_{3})\right) {\rm d}x_{3}\right|\\
&\leq& 2C|\lambda| L\|\phi\|_{C^2}\|\mathbf{v}_\sfr- \mathbf{w}_\sfr \|_{C^{1,\alpha} ( ({S}_h\times(-M,M))^{3})} \|\mathbf x'-\mathbf y'\|^\alpha,   \\
| \mathbf{f}_{4}(\mathbf x')- \mathbf{f}_{4}(\mathbf y')|&=&|\lambda+\mu|\Bigg | \int^{L}_{-L}\phi'(x_{3})\Big (
\begin{pmatrix}
    \partial_{1} (v_{3}-  w_{3})          \\
     \partial_{2}( v_{3}   -  w_{3} )
\end{pmatrix}
(\mathbf{x}',x_{3})\\
&&-\begin{pmatrix}
    \partial_{1} (v_{3}-  w_{3})          \\
     \partial_{2}( v_{3}   -  w_{3} )
\end{pmatrix}
(\mathbf{y}',x_{3}) \Big) {\rm d}x_{3}\Bigg |\\
&\leq&2 |\lambda+\mu| L\|\phi\|_{C^1}\|\mathbf{v}_\sfr- \mathbf{w}_\sfr \|_{C^{1,\alpha} ( ({S}_h\times(-M,M))^{3})} \|\mathbf x'-\mathbf y'\|^\alpha,
\end{eqnarray*}
where $C$ is a positive constant. Therefore we know that $\mathbf{f}_{3}(\mathbf{x}') \in C^\alpha(\overline{S}_h)$ and $\mathbf{f}_{4}(\mathbf{x}')\in C^\alpha(\overline{S}_h)$.

Due to \eqref{eq:qw 3}, $\mathbf{f}_{2}(\mathbf{x}') \in C^\alpha(\overline{S}_h)$ can be obtained directly by virtue of Lemma \ref{lem:31}.
\end{proof}

\begin{lem}\label{lem36}
Recall that  $\mathbf{f}_{2}(\mathbf{x}')$, $\mathbf{f}_{3}(\mathbf{x}')$ and $\mathbf{f}_{4}(\mathbf{x}')$ are defined in \eqref{eq:10}. Suppose that $\mathbf{f}_{\ell}(\mathbf{x}') \in C^\alpha(\overline{S}_h)$ ($\ell=2,3,4$), $0<\alpha<1$. Recall that $  \widetilde{I}_{1}$ is defined in \eqref{eq:10},  then the following integral estimates hold
\begin{eqnarray}
\left | \widetilde{I}_{1} \right |
&\leq &  6 (\mathbf{f}_{1j}(\mathbf{0})+\mathbf{f}_{2}(\mathbf{0})+\mathbf{f}_{3}(\mathbf{0})+\mathbf{f}_{4}(\mathbf{0}))        |e^{-2\theta_{M}{\rm i}}-e^{-2\theta_{m}\rm i}|s^{-4} \label{eq:I1}
\\&&+\sum_{\ell=2}^4\| \mathbf{f}_{\ell }\|_{C^{\alpha}(S_{h})^{2}}\frac{2\sqrt{2}(\theta_{M}-\theta_{m})\Gamma(2\alpha+4)}{\delta_{W}^{2\alpha+4}}s^{-2\alpha-4} \notag
\\&&+
4 L\omega^2 {\rm diam}(S_h)^{1-\alpha } \sqrt{\pi}\|\phi\|_{L^\infty}
\sum_{\beta=p,\, s}(1+k_\beta)\bigg(\|\mathbf{g}_{\sfr,j\beta}^{(1,2)}\|_{L^2({\mathbb S}^{2})}
+\|\mathbf{g}_{\sfi,j\beta}^{(1,2)}\|_{L^2({\mathbb S}^{2})} \bigg), \notag
 \end{eqnarray}
where $\mathbf{g}_{\sfr ,j\beta}^{(1,2)}  $ and $\mathbf{g}_{\sfi ,j\beta}^{(1,2)}$  are defined in  \eqref{g3} ($\beta=p,\ s$), $  \delta_{W}$ is defined in \eqref{eq:lame7}, $\theta_{m}$ and  $\theta_{M}$  are defined in \eqref{eq:lame1}.
\end{lem}
\begin{proof}
Due to $\mathbf{f}_{\ell}(\mathbf{x}') \in C^\alpha(\overline{S}_h)$ ($\ell=2,3,4$), one has
 \begin{equation}\label{eq:28}
 \begin{aligned}
&\mathbf{f}_{2}(\mathbf{x}')=\mathbf{f}_{2}(\mathbf{0})+\delta \mathbf{f}_{2}(\mathbf{x}'), \ |\delta \mathbf{f}_{2}(\mathbf{x}')|\leq\|\mathbf{f}_{2}\|_{C^{\alpha}(S_{h})^{2}}|\mathbf{x}'|^{\alpha},\\
&\mathbf{f}_{3}(\mathbf{x}')=\mathbf{f}_{3}(\mathbf{0})+\delta \mathbf{f}_{3}(\mathbf{x}'), \ |\delta \mathbf{f}_{3}(\mathbf{x}')|\leq\|\mathbf{f}_{3}\|_{C^{\alpha}(S_{h})^{2}}|\mathbf{x}'|^{\alpha},\\
&\mathbf{f}_{4}(\mathbf{x}')=\mathbf{f}_{4}(\mathbf{0})+\delta \mathbf{f}_{4}(\mathbf{x}'), \ |\delta \mathbf{f}_{4}(\mathbf{x}')|\leq\|\mathbf{f}_{4}\|_{C^{\alpha}(S_{h})^{2}}|\mathbf{x}'|^{\alpha}.\\
\end{aligned}
\end{equation}

It is easy to see that $\mathbf{f}_{1j}(\mathbf{x}') \in C^\alpha(\overline{S}_h) $. Hence it yields that
\begin{equation}\label{f1j}
\begin{aligned}
&\mathbf{f}_{1j}(\mathbf{x}')=\mathbf{f}_{1j}(\mathbf{0})+\delta \mathbf{f}_{1j}(\mathbf{x}'), \ |\delta
\mathbf{f}_{1j}(\mathbf{x}')|\leq\|\mathbf{f}_{1j}\|_{C^{\alpha}(S_{h})^{2}}|\mathbf{x}'|^{\alpha}. 
\end{aligned}
\end{equation}
By virtue of \eqref{eq:28} and \eqref{f1j}, we have
\begin{equation}\label{eq:29}
 \begin{aligned}
 \widetilde{I}_{1}=&(\mathbf{f}_{1j}(\mathbf{0})+\mathbf{f}_{2}(\mathbf{0})+\mathbf{f}_{3}(\mathbf{0})+\mathbf{f}_{4}(\mathbf{0}))
\int_{S_{h}}\mathbf{u}(\mathbf{x}'){\rm d}\mathbf{x}'
+\int_{S_{h}}\delta \mathbf{f}_{1j}\cdot\mathbf{u}(\mathbf{x}'){\rm d}\mathbf{x}'
\\&+\int_{S_{h}}\delta \mathbf{f}_{2}\cdot\mathbf{u}(\mathbf{x}'){\rm d}\mathbf{x}'
+\int_{S_{h}}\delta \mathbf{f}_{3}\cdot\mathbf{u}(\mathbf{x}'){\rm d}\mathbf{x}'
+\int_{S_{h}}\delta \mathbf{f}_{4}\cdot\mathbf{u}(\mathbf{x}'){\rm d}\mathbf{x}'.
 \end{aligned}
\end{equation}
Recall that $ \mathbf{f}_{1j}=-\omega^{2}\mathcal{R}(\mathbf{v}_{j}^{(1,2)}) $. Using the property of compact embedding of H{\"o}lder spaces, we can derive that for $0 < \alpha < 1$,
$$
\| \mathbf{f}_{1j}  \|_{C^\alpha } \leq \omega^2 {\rm diam}(S_h)^{1-\alpha } \|{\mathcal R}(\mathbf{v}_j^{(1,2)}) \|_{C^1},
$$
where $ {\rm diam}(S_h)$ is the diameter of $ S_h$.  By the definition of the dimension reduction operator \eqref{rg.},  using  \eqref{g3} and \eqref{eq:vrj 313},  it is easy to see that
 \begin{eqnarray*}
	|{\mathcal R}(\mathbf{v}_j^{(1,2)}) (\mathbf x')|  & \leq& 4 L\sqrt{\pi }\|\phi \|_{L^\infty} \bigg(\|\mathbf{g}_{\sfr,jp}^{(1,2)}\|_{L^2({\mathbb S}^{2})}
+\|\mathbf{g}_{\sfi,jp}^{(1,2)}\|_{L^2({\mathbb S}^{2})}
\\
 &&+\|\mathbf{g}_{\sfr,js}^{(1,2)}\|_{L^2({\mathbb S}^{2})}
+\|\mathbf{g}_{\sfi,js}^{(1,2)}\|_{L^2({\mathbb S}^{2})}\bigg),
\\
|\partial_{x_i}{\mathcal R}(\mathbf{v}_j^{(1,2)}) (\mathbf x')|&  \leq& 4   L\sqrt{\pi }\|\phi \|_{L^\infty}
\bigg(k_p\|\mathbf{g}_{\sfr,jp}^{(1,2)}\|_{L^2({\mathbb S}^{2})}
+k_p\|\mathbf{g}_{\sfi,jp}^{(1,2)}\|_{L^2({\mathbb S}^{2})}
\\ &&+k_s\|\mathbf{g}_{\sfr,js}^{(1,2)}\|_{L^2({\mathbb S}^{2})}
+k_s\|\mathbf{g}_{\sfi,js}^{(1,2)}\|_{L^2({\mathbb S}^{2})}\bigg).
\end{eqnarray*}
Thus we have
\begin{equation}\label{rejc1}
 \begin{aligned}
\|{\mathcal R}( \mathbf{v}_j^{(1,2)}) \|_{C^1} \leq& 4 L\sqrt{\pi}\|\phi\|_{L^\infty}
\bigg((1+k_p)\|\mathbf{g}_{\sfr,jp}^{(1,2)}\|_{L^2({\mathbb S}^{2})}
+(1+k_p)\|\mathbf{g}_{\sfi,jp}^{(1,2)}\|_{L^2({\mathbb S}^{2})}
\\&+(1+k_s)\|\mathbf{g}_{\sfr,js}^{(1,2)}\|_{L^2({\mathbb S}^{2})}
+(1+k_s)\|\mathbf{g}_{\sfi,js}^{(1,2)}\|_{L^2({\mathbb S}^{2})}\bigg).
\end{aligned}
\end{equation}
From \eqref{eq:lame6} and  \eqref{rejc1}, we have
\begin{eqnarray}
&&\left |\int_{S_{h}}\delta \mathbf{f}_{1j}  \cdot\mathbf{u}(\mathbf{x}'){\rm d}\mathbf{x}'\right |\leq \| \mathbf{f}_{1j}\|_{C^{\alpha}(S_{h})^{2}}\left |\int_{W}\mathbf{u}(\mathbf{x}')|\mathbf{x}'|^{\alpha}{\rm d}\mathbf{x}'\right |\notag \\
&&
\leq \|\mathbf{f}_{1j}\|_{C^{\alpha}(S_{h})^{2}}\frac{2\sqrt{2}(\theta_{M}-\theta_{m})\Gamma(2\alpha+4)}{\delta_{W}^{2\alpha+4}}s^{-2\alpha-4}. \notag
\\
&&\leq
 \omega^2 {\rm diam}(S_h)^{1-\alpha } 4 L\sqrt{\pi}\|\phi\|_{L^\infty}
\bigg((1+k_p)\|\mathbf{g}_{\sfr,jp}^{(1,2)}\|_{L^2({\mathbb S}^{2})}
+(1+k_p)\|\mathbf{g}_{\sfi,jp}^{(1,2)}\|_{L^2({\mathbb S}^{2})} \notag
\\
&&+(1+k_s)\|\mathbf{g}_{\sfr,js}^{(1,2)}\|_{L^2({\mathbb S}^{2})}
+(1+k_s)\|\mathbf{g}_{\sfi,js}^{(1,2)}\|_{L^2({\mathbb S}^{2})}\bigg)s^{-2\alpha-4}.\label{eq:30}
\end{eqnarray}
Similarly, we have
\begin{align}
	\left |\int_{S_{h}}\delta \mathbf{f}_{2}  \cdot\mathbf{u}(\mathbf{x}'){\rm d}\mathbf{x}'\right |
&\leq\| \mathbf{f}_{2}\|_{C^{\alpha}(S_{h})^{2}}\left |\int_{W}\mathbf{u}(\mathbf{x}')|\mathbf{x}'|^{\alpha}{\rm d}\mathbf{x}'\right | \notag
\\
&\leq\| \mathbf{f}_{2}\|_{C^{\alpha}(S_{h})^{2}}\frac{2\sqrt{2}(\theta_{M}-\theta_{m})\Gamma(2\alpha+4)}{\delta_{W}^{2\alpha+4}}s^{-2\alpha-4},\label{eq:31}\\
\left |\int_{S_{h}}\delta \mathbf{f}_{3}  \cdot\mathbf{u}(\mathbf{x}'){\rm d}\mathbf{x}'\right |
&\leq\| \mathbf{f}_{3}\|_{C^{\alpha}(S_{h})^{2}}\left |\int_{W}\mathbf{u}(\mathbf{x}')|\mathbf{x}'|^{\alpha}{\rm d}\mathbf{x}'\right | \notag
\\
&\leq\| \mathbf{f}_{3}\|_{C^{\alpha}(S_{h})^{2}}\frac{2\sqrt{2}(\theta_{M}-\theta_{m})\Gamma(2\alpha+4)}{\delta_{W}^{2\alpha+4}}s^{-2\alpha-4},\label{eq:32}\\
\left |\int_{S_{h}}\delta \mathbf{f}_{4}  \cdot\mathbf{u}(\mathbf{x}'){\rm d}\mathbf{x}'\right |
&\leq\| \mathbf{f}_{4}\|_{C^{\alpha}(S_{h})^{2}}\left |\int_{W}\mathbf{u}(\mathbf{x}')|  \mathbf{x}'|^{\alpha}{\rm d}\mathbf{x}'\right | \notag
\\
&\leq\| \mathbf{f}_{4}\|_{C^{\alpha}(S_{h})^{2}}\frac{2\sqrt{2}(\theta_{M}-\theta_{m})\Gamma(2\alpha+4)}{\delta_{W}^{2\alpha+4}}s^{-2\alpha-4}.\label{eq33}
\end{align}
%
%
%
%
%
%
%
From \eqref{eq:29}, \eqref{eq:30}, \eqref{eq:31}, \eqref{eq:32} and \eqref{eq33}, we can derive \eqref{eq:I1}.
\end{proof}

\begin{lem}
\label{lem37}
	Let $j_\ell (t)$ be the $\ell$-th spherical Bessel function with the form
	  \begin{equation}\notag
	 	 j_\ell (t)=\frac{t^\ell }{ (2\ell+1)!!}\left  (1-\sum_{l=1}^\infty  \frac{(-1)^l t^{2l }}{ 2^l l! N_{\ell,l} }\right  ),
	 \end{equation}
	 where $N_{\ell,l}=(2\ell+3)\cdots  (2\ell+2l+1)$ and $\mathcal R$ be the dimension reduction operator defined in Definition \ref{Def}. Then
\begin{subequations}\label{rjx}
 \begin{align}
 {\mathcal R} (j_{0}(k_\beta |\mathbf x|))(\mathbf{x}')&= C(\phi )+C_1(\phi )\sum^{\infty}_{l=1}\frac{(-1)^lk_\beta ^{2l}(|\mathbf{x}'|^2+a_{0,l}^2)^{l-\frac{1}{2}}}{2^ll!(2l+1)!!}
   |\mathbf{x}'|^2, \label{rjx} \\
{\mathcal R} (j_{2\ell}(k_\beta |\mathbf x|))(\mathbf{x}')	
&=C_{1}(\phi )|\mathbf{x}'|^2\frac{k_{\beta}^{2\ell}(|\mathbf{x}'|^2+a_{\ell,1}^2)^{\ell -\frac{1}{2}} }{ (4\ell+1)!!}   \left  [1-\sum_{l=1}^\infty  \frac{(-1)^l k_{\beta}^{2l } (|\mathbf{x}'|^2+a_{\ell,l,1}^2)^{l }}{ 2^l l! N_{\ell,l,1} }\right  ], \label{eq:350b} \\
{\mathcal R} (j_{2\ell-1}(k_\beta |\mathbf x|))(\mathbf{x}')	&=C_{1}(\phi )|\mathbf{x}'|^2\frac{k_{\beta}^{2\ell-1}(|\mathbf{x}'|^2+a_{\ell,1}^2)^{\ell -1} }{ (4\ell-1)!!}   \left  [1-\sum_{l=1}^\infty  \frac{(-1)^l k_{\beta}^{2l } (|\mathbf{x}'|^2+a_{\ell,l,1}^2)^{l }}{ 2^l l! N_{\ell,l,2} }\right  ], \label{eq:350c}
 \end{align}
\end{subequations}
	 where $\beta=p,s$, $\ell \in  \mathbb N  $, $a_{0,l},\xi, a_{\ell,1}, \, a_{\ell,l,1} \in (-L,L)$,   $N_{\ell,l,1}=(4\ell+3)\cdots  (4\ell+2l+1)$, $N_{\ell,l,2}=(4\ell+1)\cdots (4\ell+2l-1)$ and
	 \begin{equation}\label{C}
	 	C(\phi ) = \int_{-L}^L \phi(x_3) {\rm d} x_3>0, \quad C_1(\phi )= \int_{-\arctan L/|\mathbf{x}'|}^{\arctan L/|\mathbf{x}'|}   \phi(|\mathbf{x}'| \tan\omega) \sec^3\omega  {\rm d}   \omega.
	 \end{equation}
	 Furthermore, it holds that
	 \begin{equation}\label{CE}
	0<C_1(\phi) <\sec^3\varpi C(\phi ).
\end{equation}
where $ \varpi \in (-\arctan L/|\mathbf{x}'|, \arctan L/|\mathbf{x}'|)$.
\end{lem}
\begin{proof}
By virtue of \eqref{rg.}, we have
\begin{equation}\label{eq:353}
 \begin{aligned}
  {\mathcal R} (j_{0})(\mathbf{x}')=&\int^{L}_{-L}\phi(x_3)\bigg(1-\sum^{\infty}_{l=1}\frac{(-1)^lk_\beta ^{2l}(|\mathbf{x}'|^2+x_3^2)^l}{2^ll!(2l+1)!!}\bigg){\rm d}x_3
 \\=&\int^{L}_{-L}\phi(x_3){\rm d}x_3+\int^{L}_{-L}\phi(x_3)\sum^{\infty}_{l=1}\frac{(-1)^lk_\beta ^{2l}(|\mathbf{x}'|^2+x_3^2)^l}{2^ll!(2l+1)!!}{\rm d}x_3.
 \end{aligned}
\end{equation}
For $ \int^{L}_{-L}\phi(x_3)\sum^{\infty}_{l=1}\frac{(-1)^lk_\beta ^{2l}(|\mathbf{x}'|^2+x_3^2)^l}{2^ll!(2l+1)!!}{\rm d}x_3$,  using integral mean value theorem, we have
\begin{equation}\label{eq:354 int}
 \begin{aligned}
  &\int^{L}_{-L}\phi(x_3)\sum^{\infty}_{l=1}\frac{(-1)^lk_\beta ^{2l}(|\mathbf{x}'|^2+x_3^2)^l}{2^ll!(2l+1)!!}{\rm d}x_3
 \\=&\sum^{\infty}_{l=1}\frac{(-1)^lk_\beta ^{2l}(|\mathbf{x}'|^2+a_{0,l}^2)^{l-\frac{1}{2}}}{2^ll!(2l+1)!!}\int^{L}_{-L}\phi(x_3)(|\mathbf{x}'|^2+x_3^2)^\frac{1}{2}{\rm d}x_3
\end{aligned}
\end{equation}
 where  $a_{0,l} \in (-L,L)$. By change of variables, where $ x_3=|x'|\tan \omega$, we have
\begin{equation}\label{eq:355 int}
 \begin{aligned}
 &\int^{L}_{-L}\phi(x_3)(|\mathbf{x}'|^2+x_3^2)^\frac{1}{2}{\rm d}x_3=\int_{-\arctan \frac{L}{|\mathbf{x}'|}}^{\arctan \frac{L} {|\mathbf{x}'|}}  \phi(|\mathbf{x}'| \tan \omega )|\mathbf{x}'| ^2 \sec^3 \omega{\rm d}\omega
=|\mathbf{x}'| ^2 C_1(\phi)
 .
\end{aligned}
\end{equation}
Plugging \eqref{eq:354 int} and \eqref{eq:355 int} into \eqref{eq:353}, it yields \eqref{rjx}.

We proceed to prove  \eqref{CE}.  Using the integral mean value theorem and variable substitution $ \omega=\arctan\frac{x_3}{|\mathbf{x}'|}$, we have
\begin{equation}\notag
 \begin{aligned}
 C_1(\phi)=\sec^3 \varpi\int_{-\arctan L/|\mathbf{x}'|}^{\arctan L/|\mathbf{x}'|}  \phi(|\mathbf{x}'| \tan \omega ) {\rm d}\omega=\sec^3 \varpi\int_{-L}^{L}  \phi(x_3) \frac{|\mathbf{x}'| ^2}{|\mathbf{x}'| ^2+x_3^2}{\rm d}x_3,
\end{aligned}
\end{equation}
where $ \varpi \in (-\arctan L/|\mathbf{x}'|, \arctan L/|\mathbf{x}'|)$. Using the integral mean value theorem again, we have
\begin{equation}\notag
 \begin{aligned}
 C_1(\phi)=\frac{|\mathbf{x}'| ^2}{|\mathbf{x}'| ^2+\xi^2}\sec^3 \varpi\int_{-L}^{L}  \phi(x_3) {\rm d}x_3\leq\sec^3 \varpi C(\phi) ,
\end{aligned}
\end{equation}
where $ \xi \in (-L,L) $.

The series expansion of $ {\mathcal R} (j_{2\ell})(\mathbf{x}')$ and $ {\mathcal R} (j_{2\ell-1})(\mathbf{x}')$ in \eqref{eq:350b} and \eqref{eq:350c}  can be obtained by a similar argument. \end{proof}
\begin{rem}\label{rem:32}
	We should emphasize that $C_1(\phi)$ is uniformly bounded with respect to $|\mathbf x'|$ since $ \varpi \in (-\arctan \frac{L}{|\mathbf{x}'|}, \arctan \frac{L} {|\mathbf{x}'|})$ is fixed, which shall be used in the proof of Theorem \ref{thm:31} in what follows. 
	\end{rem}

\begin{lem}\label{lem38}
Consider the same setup in Lemma \ref{lem34} and recall that $\widetilde{I}_{\pm}$ is defined \eqref{eq:10}. Suppose that $\eta \in C^\alpha (\overline{\Gamma_{h}^{\pm}}\times [-M,M] ) $ is independent of $x_3$ and has the expansion \eqref{etae}, then it holds that
\begin{equation}\label{int23}
\widetilde{I}_{\pm}=\mathcal{\widetilde{I}}^{\pm}_{1}+\eta(\mathbf{0})\mathcal{\widetilde{I}}^{\pm}_{2},
\end{equation}
where $\widetilde{I}_{\pm}$ are defined in \eqref{eq:10} and
$$
\mathcal{\widetilde{I}}^{\pm}_{1}=\int_{\Gamma^{\pm}_{h}}\delta\eta(\mathbf{x}')\mathcal{R}(\mathbf{v}_{j}^{(1,2)})\cdot\mathbf{u}(\mathbf{x}'){\rm d}\sigma,\quad
\mathcal{\widetilde{I}}^{\pm}_{2}=\int_{\Gamma^{\pm}_{h}}\mathcal{R}(\mathbf{v}_{j}^{(1,2)} )\cdot\mathbf{u}(\mathbf{x}'){\rm d}\sigma.
$$
Furthermore, the following estimate holds as $s\rightarrow+\infty$:
\begin{equation}\label{eq:deuvj1 3d}
 \begin{aligned}
\left|\mathcal{\widetilde{I}}^{\pm}_{1}\right|
\leq & O(s^{-2\alpha-6})+
\sum_{\alpha=\sfr, \sfi \atop {\beta=p, s}}\bigg(\| \mathbf{g}_{\alpha, j\beta}^{(1,2)} \|_{L^{2}(\mathbb{S}^{2})^{2}}
+\| \mathbf{g}_{\alpha, j\beta}^{(1,2)} \|_{L^{2}(\mathbb{S}^{2})^{2}}\bigg)\times O(s^{-2\alpha-6}).
 \end{aligned}
\end{equation}
where $\mathbf{g}_{\sfr, j\beta }^{(1,2)} $ and $\mathbf{g}_{\sfi, j\beta }^{(1,2)} $ are defined in \eqref{g3}.
\end{lem}

\begin{proof} One can immediately deduce \eqref{int23} by \eqref{etae}.
By virtue of \eqref{eq:vj12 357}
and the reduction operator $\mathcal R$ defined in \eqref{rg.},  we deduce that
\begin{equation}\label{eq:rv}
 \begin{aligned}
 \mathcal{R}(\mathbf{v}^{(1,2)}_j(\mathbf{x}))&=\mathbf{v}_{jp}^{(1,2)}(\mathbf{0})\mathcal{R}(j_{0}(k_{p}|\mathbf{x}|))+\mathbf{v}_{js}^{(1,2)}(\mathbf{0})\mathcal{R}(j_{0}(k_{s}|\mathbf{x}|))\\
 &
+2\sum^{+\infty}_{\ell =1} (-1)^\ell(4\ell+1) \mathcal{R}( j_{2\ell}(k_p |\mathbf{x}|))\mathbf{B}_{jp,1}^{(\ell)}+2\sum^{+\infty}_{\ell =1} (-1)^\ell (4\ell+1)\\
&\times \mathcal{R}( j_{2\ell}(k_s|\mathbf{x}|))\mathbf{B}_{js,1}^{(\ell)}
+2\sum^{+\infty}_{\ell =1}  (-1)^\ell (4\ell-1)\mathcal{R}( j_{2\ell-1}(k_p |\mathbf{x}|))\mathbf{B}_{jp,2}^{(\ell)}
\\&+2\sum^{+\infty}_{\ell =1}  (-1)^\ell (4\ell-1) \mathcal{R}(j_{2\ell-1}(k_s |\mathbf{x}|))\mathbf{B}_{js,2}^{(\ell)}.
\end{aligned}
\end{equation}
Plugging \eqref{eq:rv} into  $\mathcal{\widetilde{I}}^{\pm}_{1}$, together with the use of \eqref{eq:lame3},  one can directly verify that
\begin{equation}\label{eq:40}
 \begin{aligned}
 \mathcal{\widetilde{I}}^{\pm}_{1} =\mathcal{\widetilde{I}}^{\pm}_{11}
 +\mathcal{\widetilde{I}}^{\pm}_{12} + \mathcal{\widetilde{I}}^{\pm}_{13}
  + \mathcal{\widetilde{I}}^{\pm}_{14} + \mathcal{\widetilde{I}}^{\pm}_{15}
   +\mathcal{\widetilde{I}}^{\pm}_{16},
 \end{aligned}
\end{equation}
where
 \begin{align*}
&\mathcal{\widetilde{I}}^{\pm}_{11}=\int_{\Gamma^{\pm}_{h}}
\begin{pmatrix}
    1\\
   {\rm i}
\end{pmatrix}
\cdot\mathbf{v}_{jp}^{(1,2)}(\mathbf{0})\delta\eta(\mathbf{x}')\mathcal{R}(j_{0}(k_{p}|\mathbf{x}|))u_{1}
{\rm d}\sigma,
\\&\mathcal{\widetilde{I}}^{\pm}_{12}=
\int_{\Gamma^{\pm}_{h}}
\begin{pmatrix}
    1\\
   {\rm i}
\end{pmatrix}
\cdot\mathbf{v}_{js}^{(1,2)}(\mathbf{0})\delta\eta(\mathbf{x}')\mathcal{R}(j_{0}(k_{s}|\mathbf{x}|))u_{1}
{\rm d}\sigma,
\\&\mathcal{\widetilde{I}}^{\pm}_{13}=
\sum^{\infty}_{\ell=1}\int_{\Gamma^{\pm}_{h}}
\begin{pmatrix}
    1\\
   {\rm i}
\end{pmatrix}
\cdot\mathbf{B}_{jp,1}^{(\ell)}(-1)^\ell(4\ell+1)\delta\eta(\mathbf{x}')\mathcal{R}(j_{2\ell}(k_{p}|\mathbf{x}|))u_{1}
{\rm d}\sigma,
\\&\mathcal{\widetilde{I}}^{\pm}_{14}=
\sum^{\infty}_{\ell=1}\int_{\Gamma^{\pm}_{h}}
\begin{pmatrix}
    1\\
   {\rm i}
\end{pmatrix}
\cdot\mathbf{B}_{js,1}^{(\ell)}(-1)^\ell(4\ell+1)\delta\eta(\mathbf{x}')\mathcal{R}(j_{2\ell}(k_{s}|\mathbf{x}|))u_{1}
{\rm d}\sigma,
\\&\mathcal{\widetilde{I}}^{\pm}_{15}=
\sum^{\infty}_{\ell=1}\int_{\Gamma^{\pm}_{h}}
\begin{pmatrix}
    1\\
   {\rm i}
\end{pmatrix}
\cdot\mathbf{B}_{jp,2}^{(\ell)}(-1)^\ell(4\ell-1)\delta\eta(\mathbf{x}')\mathcal{R}(j_{2\ell-1}(k_{p}|\mathbf{x}|))u_{1}
{\rm d}\sigma,
\\&\mathcal{\widetilde{I}}^{\pm}_{16}=
\sum^{\infty}_{\ell=1}\int_{\Gamma^{\pm}_{h}}
\begin{pmatrix}
    1\\
   {\rm i}
\end{pmatrix}
\cdot\mathbf{B}_{js,2}^{(\ell)}(-1)^\ell(4\ell-1)\delta\eta(\mathbf{x}')\mathcal{R}(j_{2\ell-1}(k_{s}|\mathbf{x}|))u_{1}
{\rm d}\sigma
.
\end{align*}

By virtue of \eqref{eq:zeta}, \eqref{etae},  \eqref{rjx} and \eqref{CE}, using the integral mean value theorem, one has as $s\rightarrow+\infty$,
 \begin{align}
\left| \mathcal{\widetilde{I}}^{-}_{11} \right|
&=
\left|\int_{\Gamma^{-}_{h}}
\begin{pmatrix}
    1\\
   {\rm i}
\end{pmatrix}
\cdot\mathbf{v}_{jp}^{(1,2)}(\mathbf{0})\delta\eta(\mathbf{x}')
\Bigg[C(\phi )+C_1(\phi )\sum^{\infty}_{l=1}\frac{(-1)^lk_p ^{2l}(|\mathbf{x}'|^2+a_{0,l}^2)^{l-\frac{1}{2}}}{2^ll!(2l+1)!!}
   |\mathbf{x}'|^2 \Bigg]u_{1}
{\rm d}\sigma\right| \notag
\\
&\leq \sqrt{2}|C(\phi )| \|\eta\|_{C^{\alpha}}
\left|
\mathbf{v}_{jp}^{(1,2)}(\mathbf{0})\right|
\left|\int_{0}^{h}r^{\alpha}\Bigg[1+\sec^3\varpi \sum^{\infty}_{l=1}\frac{(-1)^lk_p ^{2l}(|\mathbf{x}'|^2+a_{0,l}^2)^{l-\frac{1}{2}}}{2^ll!(2l+1)!!}
   |\mathbf{x}'|^2 \Bigg]
u_{1}{\rm d}r\right| \notag
\\
&\leq \sqrt{2}|C(\phi )| \|\eta\|_{C^{\alpha}}
\left|
\mathbf{v}_{jp}^{(1,2)}(\mathbf{0})\right|
\int_{0}^{h}r^{\alpha}e^{-s\sqrt{r}\cos\frac{\theta}{2}}{\rm d}r
+
\sqrt{2}|C(\phi )| \|\eta\|_{C^{\alpha}} \notag
\\
&\quad  \times \left|
\mathbf{v}_{jp}^{(1,2)}(\mathbf{0})\right|
\left|
\sec^3\varpi \sum^{\infty}_{l=1}\frac{(-1)^lk_p ^{2l}(\beta_{0l}^2+a_{0,l}^2)^{l-\frac{1}{2}}}{2^ll!(2l+1)!!}
\right|
\int_{0}^{h}r^{\alpha+2}e^{-s\sqrt{r}\cos\frac{\theta}{2}}{\rm d}r \notag
\\
&= O(s^{-2\alpha-6}), \label{eq:41}
\end{align}
where $\beta_{0l}\in (0,h)$ and $a_{0,l}\in (-L,L)$. Similarly, we can derive that
\begin{equation}\label{eq:42}
 \begin{aligned}
\left| \mathcal{\widetilde{I}}^{+}_{11} \right|\leq  O(s^{-2\alpha-6}), \quad
\left| \mathcal{\widetilde{I}}^{\pm}_{12} \right|\leq O(s^{-2\alpha-6})\quad\mbox{as $s\rightarrow+\infty$.}
\end{aligned}
\end{equation}

From \eqref{eq:vrj 313}, using Cauchy-Schwarz inequality, we have
\begin{equation}\label{eq:361 cs}
 \begin{aligned}
 &\left|\mathbf{B}_{jp,1}^{(\ell)}\right|\leq2\sqrt{\pi}\| \mathbf{g}_{\sfr, jp}^{(1,2)}\|_{L^2(\mathbb{S}^{2})^{2}}, \
 \left|\mathbf{B}_{js,1}^{(\ell)}\right|\leq2\sqrt{\pi}\| \mathbf{g}_{\sfr, js}^{(1,2)}\|_{L^2(\mathbb{S}^{2})^{2}},
 \\&\left|\mathbf{B}_{jp,2}^{(\ell)}\right|\leq2\sqrt{\pi}\| \mathbf{g}_{\sfi, jp}^{(1,2)}\|_{L^2(\mathbb{S}^{2})^{2}}, \
 \left|\mathbf{B}_{js,2}^{(\ell)}\right|\leq2\sqrt{\pi}\| \mathbf{g}_{\sfi, js}^{(1,2)}\|_{L^2(\mathbb{S}^{2})^{2}}.
  \end{aligned}
\end{equation}
From  \eqref{eq:zeta}, \eqref{rjx} and \eqref{eq:361 cs}, we have
 \begin{align}
\left|\mathcal{\widetilde{I}}^{-}_{13}
\right|
&\leq
|\sec^3\varpi|\left|C(\phi )\right| \|\eta\|_{C^{\alpha}}\sum^{\infty}_{\ell=1}\left|4\ell+1\right|
\left|
\begin{pmatrix}
    1\\
   {\rm i}
\end{pmatrix}
\cdot\mathbf{B}_{jp,1}\right| \notag
\\
&\quad \times
\left|\int_{0}^{h}r^{\alpha+2}{ \frac{k^{2\ell}_{p}   (r^2+a_{\ell }^2)^{\ell-\frac{1}{2} } }{ (4\ell+1)!!}\Bigg[ 1-  \sum_{l=1}^\infty  \frac{ k_{p}^{2l} (r^2+a_{\ell,l }^2)^{l }}{ 2^l l! N_{\ell,l,1}  }   \Bigg] e^{-s\sqrt{r}\mu(\theta_m) }}{\rm d}r
\right| \notag
\\
&\leq
\sqrt{2}|\sec^3\varpi|\left|C(\phi )\right| \|\eta\|_{C^{\alpha}}\sum^{\infty}_{\ell=1}\left|4\ell+1\right|
\left|\mathbf{B}_{jp,1}\right| \notag
\\
&\quad \times
\frac{k^{2\ell}_{p}   (\beta_{\ell }^2+\alpha_{\ell }^2)^{\ell-\frac{1}{2} } }{ (2\ell+1)!!} \left| 1-  \sum_{l=1}^\infty  \frac{ k_{p}^{2l} (\beta_{\ell,l }^2+\alpha_{\ell,l }^2)^{l }}{ 2^l l! N_{\ell,l}  }   \right|\int_{0}^{h}r^{\alpha+2}e^{-s\sqrt{r}\cos\frac{\theta}{2}}{\rm d}r \notag
\\
&\leq\|\mathbf{ g}_{\sfr, jp}^{(1,2)}\|_{L^{2}(\mathbb{S}^{2})^{2}}\times  O(s^{-2\alpha-6}),\label{eq:45}
\end{align}
as $s\rightarrow+\infty$,
where $\beta_{\ell l},\, \beta_{\ell}\in(0,h)$ and $a_{\ell,l}\in (-L,L)$. Similarly, we can derive that
\begin{equation}\label{eq:46}
 \begin{aligned}
\left|\mathcal{\widetilde{I}}^{+}_{13}\right|
\leq \| \mathbf{g}_{\sfr, jp}^{(1,2)}\|_{L^{2}(\mathbb{S}^{2})^{2}}\times  O(s^{-2\alpha-6}),\quad
\left|\mathcal{\widetilde{I}}^{\pm}_{14}\right|
\leq \| \mathbf{g}_{\sfr, js}^{(1,2)}\|_{L^{2}(\mathbb{S}^{2})^{2}}\times  O(s^{-2\alpha-6}),
\end{aligned}
\end{equation}
as $s\rightarrow+\infty$.
By \eqref{eq:zeta} and \eqref{rjx}, we have
 \begin{align}
 \left|
\mathcal{\widetilde{I}}^{-}_{15}
\right|
\leq&
|C(\phi )| \|\eta\|_{C^{\alpha}}\sum^{\infty}_{\ell=1}\left|4\ell-1\right|
\left|
\begin{pmatrix}
    1\\
   {\rm i}
\end{pmatrix}
\cdot\mathbf{B}_{jp,2}^{(\ell)}\right| \notag
\\
&\times
\left|\int_{0}^{h}r^{\alpha}{ \frac{k^{2\ell-1} _{p}  (r^2+a_{\ell }^2)^{\ell -1} }{ (4\ell-1)!!}\Bigg[ 1-  \sum_{l=1}^\infty  \frac{ k^{2l}_{p} (r^2+a_{\ell,l }^2)^{l }}{ 2^l l! N_{\ell,l,1}  }   \Bigg] e^{-s\sqrt{r}\mu(\theta_m) } }{\rm d}r
\right| \notag
\\
\leq&
\sqrt{2}|C(\phi )| \|\eta\|_{C^{\alpha}}\sum^{\infty}_{\ell=1}\left|4\ell-1\right|
\left|\mathbf{B}_{jp,2}^{(\ell)}\right| \notag
\\
&\times
\frac{k^{2\ell-1} _{p}  (\beta_{\ell }^2+\alpha_{\ell }^2)^{\ell-1 } }{ (2\ell+1)!!} \left| 1-  \sum_{l=1}^\infty  \frac{ k^{2l}_{p} (\beta_{\ell,l }^2+a_{\ell,l }^2)^{l }}{ 2^l l! N_{\ell,l}  }   \right|\int_{0}^{h}r^{\alpha+2}e^{-s\sqrt{r}\cos\frac{\theta}{2}}{\rm d}r \notag
\\
\leq&\|\mathbf{ g}_{\sfi,jp}^{(1,2)} \|_{L^{2}(\mathbb{S}^{2})^{2}}\times  O(s^{-2\alpha-6}), \label{eq:488}
\end{align}
as $s\rightarrow+\infty$,
where $\beta_{\ell l},\, \beta_{\ell}\in(0,h)$ and $a_{\ell,l }\in (-L,L)$. Similarly, we can derive that
\begin{equation}\label{eq:489}
 \begin{aligned}
 \left|\mathcal{\widetilde{I}}^{+}_{15}\right|
\leq \| \mathbf{g}_{\sfi,jp}^{(1,2)} \|_{L^{2}(\mathbb{S}^{2})^{2}}\times O( s^{-2\alpha-6}),\quad
  \left|\mathcal{\widetilde{I}}^{\pm}_{16}\right|
\leq \| \mathbf{g}_{\sfi,js}^{(1,2)} \|_{L^{2}(\mathbb{S}^{2})^{2}}\times O( s^{-2\alpha-6}),
 \end{aligned}
\end{equation}
as $s\rightarrow+\infty$.
From  \eqref{eq:40},  \eqref{eq:41},  \eqref{eq:42},  \eqref{eq:45}, \eqref{eq:46}, \eqref{eq:488} and \eqref{eq:489},  we obtain \eqref{eq:deuvj1 3d}.
\end{proof}

\begin{lem}
Consider the same setup in Lemma \ref{lem34} and let $\mathcal{\widetilde{I}}^{\pm}_{2}$ be defined in \eqref{int23}. Let
 \begin{align*}
&\widetilde{I}_{21,\beta }^{\pm}=
\int_{\Gamma^{\pm}_{h}}
\begin{pmatrix}
    1\\
   {\rm i}
\end{pmatrix}
\cdot\mathbf{v}_{j\beta }^{(1,2)}(\mathbf{0})\mathcal{R}(j_{0}(k_{\beta}|\mathbf{x}|))u_{1}
{\rm d}\sigma,
\\&\widetilde{I}_{23,\beta }^{\pm}=
\sum^{\infty}_{\ell=1}\int_{\Gamma^{\pm}_{h}}
\begin{pmatrix}
    1\\
   {\rm i}
\end{pmatrix}
\cdot\mathbf{B}_{j\beta,1}^{(\ell)} (-1)^\ell(4\ell+1)\mathcal{R}(j_{2\ell}(k_{\beta}|\mathbf{x}|))u_{1}
{\rm d}\sigma,
\\&\widetilde{I}_{25,\beta}^{\pm}=
\sum^{\infty}_{\ell=1}\int_{\Gamma^{\pm}_{h}}
\begin{pmatrix}
    1\\
   {\rm i}
\end{pmatrix}
\cdot\mathbf{B}_{j\beta,2}^{(\ell)} (-1)^\ell(4\ell-1)\mathcal{R}(j_{2\ell-1}(k_{\beta}|\mathbf{x}|))u_{1}
{\rm d}\sigma,
\end{align*}
where $\mathbf{B}_{j\beta,i}^{(\ell)}$ ($\beta=p,\,s$, $i=1,\,2$) are defined in \eqref{eq:vj12 357}.
It holds that
\begin{equation}\label{eq:491}
 \begin{aligned}
\widetilde{I}_{2}^{\pm}=\widetilde{I}_{21,p}^{\pm}
+\widetilde{I}_{21,s}^{\pm}+\widetilde{I}_{23,p}^{\pm}+\widetilde{I}_{23,s}^{\pm}
+\widetilde{I}_{25,p}^{\pm}+\widetilde{I}_{25,s}^{\pm}.
\end{aligned}
\end{equation} 
\end{lem}

\begin{proof}
	Plugging \eqref{eq:rv} into $\widetilde{I}_{2}^{\pm}$, one can readily obtain \eqref{eq:491}.
\end{proof}

\begin{lem}\label{lem39}
Consider the same setup in Lemma \ref{lem34} and recall that $\widetilde{I}_{2}^{\pm}$ is defined \eqref{int23}. Denote
 \begin{align*}
&\widetilde{I}_{211}^{\pm}=C(\phi )
\begin{pmatrix}
    1\\
   {\rm i}
\end{pmatrix}
\cdot\mathbf{v}_{jp}^{(1,2)}(\mathbf{0})\int_{\Gamma^{\pm}_{h}}u_{1}
{\rm d}\sigma,
\quad \widetilde{I}_{221}^{\pm}=C(\phi )
\begin{pmatrix}
    1\\
   {\rm i}
\end{pmatrix}
\cdot\mathbf{v}_{js}^{(1,2)}(\mathbf{0})\int_{\Gamma^{\pm}_{h}}u_{1}
{\rm d}\sigma,
\\&\widetilde{I}_{212}^{\pm}=
\begin{pmatrix}
    1\\
   {\rm i}
\end{pmatrix}
\cdot\mathbf{v}_{jp}^{(1,2)}(\mathbf{0}) \int_{\Gamma_{h}^{\pm}}C_1(\phi )\sum^{\infty}_{l=1}\frac{(-1)^lk_p ^{2l}(|\mathbf{x}'|^2+a_{0,l}^2)^{l-\frac{1}{2}}}{2^ll!(2l+1)!!}
   |\mathbf{x}'|^2
 u_{1}
{\rm d}\sigma,
\\&\widetilde{I}_{222}^{\pm}=
\begin{pmatrix}
    1\\
   {\rm i}
\end{pmatrix}
\cdot\mathbf{v}_{js}^{(1,2)}(\mathbf{0}) \int_{\Gamma_{h}^{\pm}}C_1(\phi )\sum^{\infty}_{l=1}\frac{(-1)^lk_s ^{2l}(|\mathbf{x}'|^2+a_{0,l}^2)^{l-\frac{1}{2}}}{2^ll!(2l+1)!!}
   |\mathbf{x}'|^2
 u_{1}
{\rm d}\sigma,
  \end{align*}
where $C(\phi )$, $C_1(\phi ) $ are defined in \eqref{rjx}, and $a_{0,l}\in (-L,L)$.
Then it holds that
 \begin{align}
  \widetilde{I}_{211}^{-}+\widetilde{I}_{211}^{+}+\widetilde{I}_{221}^{-}+\widetilde{I}_{221}^{+}
 &=2s^{-2}\bigg(C(\phi)\begin{pmatrix}
    1\\
   {\rm i}
\end{pmatrix}
\cdot\mathbf{v}_{jp}^{(1,2)}(\mathbf{0})
+
C(\phi)\begin{pmatrix}
    1\\
   {\rm i}
\end{pmatrix}
\cdot\mathbf{v}_{js}^{(1,2)}(\mathbf{0})\bigg) \label{l39}
\\
 &\ \bigg(-s\sqrt{h}e^{-s\sqrt{h}\mu(\theta_{m})}\mu^{-1}(\theta_{m})-e^{-s\sqrt{h}\mu(\theta_{m})}\mu^{-2}(\theta_{m})+\mu^{-2}(\theta_{m}) \notag
\\
& \ \ \ -s\sqrt{h}e^{-s\sqrt{h}\mu(\theta_{M})}\mu^{-1}(\theta_{M})-e^{-s\sqrt{h}\mu(\theta_{M})}\mu^{-2}(\theta_{M})+\mu^{-2}(\theta_{M})\bigg), \notag
\end{align}
where $\mu(\theta)$ is defined in \eqref{eq:I311}.
Furthermore, the following estimates hold
 \begin{align}
&|\widetilde{I}_{212}^{\pm}|\leq O(s^{-4}), \quad
|\widetilde{I}_{222}^{\pm}|\leq O(s^{-4}),\quad
\left|\widetilde{I}_{23,\beta }^{\pm}\right|\leq \| \mathbf{g}_{\sfr,j\beta }^{(1,2)} \|_{L^{2}(\mathbb{S}^{2})^{2}}\times O(s^{-2\alpha-6}), \notag
\\&
\left|\widetilde{I}_{25,\beta }^{\pm}\right|\leq
\| \mathbf{g}_{\sfi,j\beta}^{(1,2)} \|_{L^{2}(\mathbb{S}^{2})^{2}}\times O(s^{-2\alpha-6}) \quad \mbox{ as $s\rightarrow +\infty$.} \label{eq:369 est}
 \end{align}
\end{lem}

\begin{proof}
%
Substituting \eqref{rjx} into $\widetilde{I}_{21,p}^{\pm}$ in \eqref{eq:491}, one can obtain that
\begin{equation}\label{100}
\begin{aligned}
\widetilde{I}_{21,p}^{\pm}=\widetilde{I}_{211}^{\pm}+\widetilde{I}_{212}^{\pm}
\end{aligned}
\end{equation}
where $\widetilde{I}_{211}^{\pm}$ and $\widetilde{I}_{212}^{\pm}$ are defined in \eqref{l39}.

By \eqref{eq:I311}, it is directly seen that
 \begin{equation}\label{eq:51}
 \begin{aligned}
 \widetilde{I}_{211}^{-}=
&C(\phi )
\begin{pmatrix}
    1\\
   {\rm i}
\end{pmatrix}
\cdot\mathbf{v}_{jp}^{(1,2)}(\mathbf{0})\int_{\Gamma^{-}_{h}}u_{1}
{\rm d}\sigma
\\=&C(\phi )\begin{pmatrix}
    1\\
   {\rm i}
\end{pmatrix}
\cdot\mathbf{v}_{jp}^{(1,2)}(\mathbf{0})
\frac{2}{s^{2}}\big(-s\sqrt{h}\frac{e^{-s\sqrt{h}\mu(\theta_{m})}}{\mu(\theta_{m})}-\frac{e^{-s\sqrt{h}\mu(\theta_{m})}}{\mu^{2}(\theta_{m})}+\frac{1}{\mu^{2}(\theta_{m})}\big).
\end{aligned}
\end{equation}
Similarly, using \eqref{eq:I311}, we have
 \begin{equation}\label{eq:511}
 \begin{aligned}
&\widetilde{I}_{211}^{+}=
C(\phi )\begin{pmatrix}
    1\\
   {\rm i}
\end{pmatrix}
\cdot\mathbf{v}_{jp}^{(1,2)}(\mathbf{0})
\frac{2}{s^{2}}\big(-s\sqrt{h}\frac{e^{-s\sqrt{h}\mu(\theta_{M})}}{\mu(\theta_{M})}-\frac{e^{-s\sqrt{h}\mu(\theta_{M})}}{\mu^{2}(\theta_{M})}+\frac{1}{\mu^{2}(\theta_{M})}\big),
\\&\widetilde{I}_{221}^{-}=
C(\phi )\begin{pmatrix}
    1\\
   {\rm i}
\end{pmatrix}
\cdot\mathbf{v}_{js}^{(1,2)}(\mathbf{0})
\frac{2}{s^{2}}\big(-s\sqrt{h}\frac{e^{-s\sqrt{h}\mu(\theta_{m})}}{\mu(\theta_{m})}-\frac{e^{-s\sqrt{h}\mu(\theta_{m})}}{\mu^{2}(\theta_{m})}+\frac{1}{\mu^{2}(\theta_{m})}\big),
\\&\widetilde{I}_{221}^{+}=
C(\phi )\begin{pmatrix}
    1\\
   {\rm i}
\end{pmatrix}
\cdot\mathbf{v}_{js}^{(1,2)}(\mathbf{0})
\frac{2}{s^{2}}\big(-s\sqrt{h}\frac{e^{-s\sqrt{h}\mu(\theta_{M})}}{\mu(\theta_{M})}-\frac{e^{-s\sqrt{h}\mu(\theta_{M})}}{\mu^{2}(\theta_{M})}+\frac{1}{\mu^{2}(\theta_{M})}\big).
\end{aligned}
\end{equation}
Combining \eqref{eq:51} with \eqref{eq:511}, we obtain \eqref{l39}.

For $\widetilde{I}_{212}^{-} $, using Lemma \ref{lem:24} and the integral mean value theorem, we can deduce that
 \begin{equation}\label{eq:52}
 \begin{aligned}
|\widetilde{I}_{212}^{-}|\leq
&\left|\begin{pmatrix}
    1\\
   {\rm i}
\end{pmatrix}
\cdot\mathbf{v}_{jp}^{(1,2)}(\mathbf{0}) \int_{\Gamma_{h}^{\pm}}C_1(\phi )\sum^{\infty}_{l=1}\frac{(-1)^lk_p ^{2l}(|\mathbf{x}'|^2+a_{0,l}^2)^{l-\frac{1}{2}}}{2^ll!(2l+1)!!}
   |\mathbf{x}'|^2
 u_{1}
{\rm d}\sigma
\right|
\\\leq&
\sqrt{2}|\sec^3\varpi|\left|C(\phi )\right|\left|
\mathbf{v}_{jp}^{(1,2)}(\mathbf{0})\right|  \sum^{\infty}_{l=1}\frac{k_p ^{2l}(|\beta_l|^2+a_{0,l}^2)^{l-\frac{1}{2}}}{2^ll!(2l+1)!!}
\int_{0}^{h}r^{2}e^{-s\sqrt{r}\cos\frac{\theta}{2}}
{\rm d}r
\\\leq&O(s^{-4})\quad\mbox{as $ s\rightarrow+\infty$},
\end{aligned}
\end{equation}
where $ \beta_l \in(0,h)$, $a_{0,l}\in (-L,L)$ and $ \varpi \in (-\arctan L/|\mathbf{x}'|, \arctan L/|\mathbf{x}'|)$ given by \eqref{CE}.

Using a similar argument for $\widetilde{I}_{212}^{-} $, we can derive that
\begin{equation}\label{eq:521}
 \begin{aligned}
|\widetilde{I}_{212}^{+}|\leq O(s^{-4}), \quad
|\widetilde{I}_{222}^{\pm}|\leq O(s^{-4})\quad\mbox{as $ s\rightarrow+\infty$.}
\end{aligned}
\end{equation}
By virtue of \eqref{eq:zeta}, \eqref{eq:350b} and \eqref{eq:361 cs}, one has
\begin{equation}\label{eq:57}
 \begin{aligned}
 \left|\widetilde{I}_{23,p}^{-}\right|
\leq& \sum^{\infty}_{\ell=1} \left|\begin{pmatrix}
    1\\
   {\rm i}
\end{pmatrix}
\cdot\mathbf{B}_{jp}^{(\ell)}\right|\left|4\ell+1\right|
\int_{0}^{h}
 \frac{k_{p}^{2\ell}   (\beta_{\ell}^2+a_{\ell }^2)^{\ell-\frac{1}{2} } }{ (4\ell+1)!!} \\& \times \Bigg[ 1-  \sum_{l=1}^\infty  \frac{(-1)^l k_{p}^{2l} (\beta_{\ell,l }^2+a_{\ell,l }^2)^{l }}{ 2^l l! N_{\ell,l,1}  }   \Bigg]
|\sec^3\varpi||C(\phi )|\int_{0}^{h}r^2
e^{-s\sqrt{r}\cos\frac{\theta_m}{2}}
{\rm d}r
\\\leq&\| \mathbf{g}_{\sfr,jp}^{(1,2)} \|_{L^{2}(\mathbb{S}^{2})^{2}}\times O(s^{-2\alpha-6})\quad\mbox{as $ s\rightarrow+\infty$,}
\end{aligned}
\end{equation}
 where $\beta_{\ell}, \beta_{\ell,l}\in (0,h)$. Please be noted that according to Remark \ref{rem:32},  $|\sec^3\varpi|$ is uniformly bounded respect to $|\mathbf x'|$. Similarly, using \eqref{eq:zeta}, \eqref{eq:350b} and \eqref{eq:361 cs},  we obtain that
\begin{equation}\label{eq:58}
 \begin{aligned}
\left|\widetilde{I}_{23,p}^{+}\right|\leq\| \mathbf{g}_{\sfr,jp}^{(1,2)} \|_{L^{2}(\mathbb{S}^{2})^{2}}\times O(s^{-2\alpha-6}),\
\left|\widetilde{I}_{23,s}^{\pm}\right|\leq \| \mathbf{g}_{\sfr,js}^{(1,2)} \|_{L^{2}(\mathbb{S}^{2})^{2}}\times O(s^{-2\alpha-6})\ \mbox{as $ s\rightarrow+\infty$.}
\end{aligned}
\end{equation}
By virtue of \eqref{eq:zeta}, \eqref{eq:350c} and \eqref{eq:361 cs}, one has
 \begin{align}
  \left|\widetilde{I}_{25,p}^{-}\right|\leq
 &\left|\sum^{\infty}_{\ell=1}\int_{\Gamma^{-}_{h}}
\begin{pmatrix}
    1\\
   {\rm i}
\end{pmatrix}
\cdot\mathbf{B}_{jp,2}^{(\ell)}(-1)^\ell(4\ell-1)\mathcal{R}(j_{\ell}(k_{p}|\mathbf{x}|))u_{1}
{\rm d}\sigma\right| \notag
\\
\leq&
\sum^{\infty}_{\ell=1} \left|\begin{pmatrix}
    1\\
   {\rm i}
\end{pmatrix}
\cdot\mathbf{B}_{jp,2}^{(\ell )}\right|\left|4\ell-1\right|\int_{0}^{h}
 \frac{k_{p}^{2\ell-1}   (\beta_{\ell}^2+a_{\ell }^2)^{\ell-\frac{3}{2} } }{ (4\ell+1)!!} \notag\\
 & \times\Bigg[ 1-  \sum_{l=1}^\infty  \frac{(-1)^l k_{p}^{2l} (\beta_{\ell,l }^2+a_{\ell,l }^2)^{l }}{ 2^l l! N_{\ell,l,2}  }   \Bigg]
|\sec^3\varpi|\left|C(\phi )\right|\int_{0}^{h}r^2
e^{-s\sqrt{r}\cos\frac{\theta_m}{2}}
{\rm d}r \notag
\\
=&\| \mathbf{g}_{\sfi,jp}^{(1,2)} \|_{L^{2}(\mathbb{S}^{2})^{2}}\times O(s^{-2\alpha-6})\ \ \mbox{as $ s\rightarrow+\infty$,} \label{eq:61}
\end{align}
where $\beta_{\ell}, \beta_{\ell,l}\in (0,h)$. Again according to  \eqref{eq:zeta}, \eqref{eq:350c} and \eqref{eq:361 cs}, we can conclude that
\begin{equation}\label{eq:62}
 \begin{aligned}
 \left|\widetilde{I}_{25,p}^{+}\right|\leq
\| \mathbf{g}_{\sfi,jp}^{(1,2)} \|_{L^{2}(\mathbb{S}^{2})^{2}}\times O(s^{-2\alpha-6}),\
\left|\widetilde{I}_{25,s}^{+}\right|\leq
\| \mathbf{g}_{\sfi,js}^{(1,2)} \|_{L^{2}(\mathbb{S}^{2})^{2}}\times O(s^{-2\alpha-6}),
\end{aligned}
\end{equation}
as $ s\rightarrow+\infty$.

In view of \eqref{eq:52}--\eqref{eq:62}, we derive \eqref{eq:369 est}.
\end{proof}

\begin{lem}\label{lem:314 third}
	Consider the same setup in Lemma \ref{lemma3.2} and suppose that $\mathbf{v}_\sfr $ can be approximated by a sequence of the  functions $\{\mathbf{v}^\sfr_j\} _{j=1}^{+\infty} $  defined by \eqref{eq:vj3} in $H^1(S_h \times (-M,M))^{3}$  satisfying \eqref{eq:vj aprrox 3}. If $\lambda \neq 0$, $\eta(\mathbf 0) \neq 0$ and  $-\pi<\theta_{m}<\theta_{M}<\pi $ satisfying  $\theta_M-\theta_m\neq \pi$, then we have
	\begin{equation}\label{eq:vj3 zero}
		\lim_{j \rightarrow +\infty} v_j^{(3)}(\mathbf 0) =0.
	\end{equation}
\end{lem}

\begin{proof}
%
%
%
%
From Lemma \ref{lem:pde decouple}, we know that $(v_3,w_3) \in H^1(S_h \times (-M,M))$ fulfills \eqref{3}. Since \eqref{eq:qw 3} and \eqref{eq:vr-wr} are satisfied, using a similar argument as that for Lemma \ref{lem:c alpha}, we can prove that $\mathbf{G}_{1}^{(3)}(\mathbf{x}')\in C^\alpha (\overline{S}_h)$ and $\mathbf{G}_{2}^{(3)}(\mathbf{x}')\in C^\alpha (\overline{S}_h)$ ($0<\alpha<1$), where $\mathbf{G}_{1}^{(3)}(\mathbf{x}')$ and $\mathbf{G}_{2}^{(3)}(\mathbf{x}') $ are  the RHS terms of \eqref{3}. According to \eqref{eq:vj aprrox 3}, one has
\begin{equation}\label{eq:v3 381}
	\|v_3-v_j^{(3)}\|_{H^1(S_h \times (-M,M))} \leq j^{-\gamma}.
\end{equation}
By virtue of \eqref{eq:v3 381}, under the assumption $\theta_M-\theta_m\neq \pi$ and $\eta(\mathbf 0) \neq 0$, adopting a similar argument for \cite[Theorem 3.1]{DCL}, we can prove \eqref{eq:vj3 zero}.
\end{proof}

\begin{thm}\label{thm:31}

Let $\Omega$, $S_h\times(-M, M)$ be described above and $\alpha\in (0, 1)$. For any fixed  $x_3^c \in (-M,M)$ and $L>0$ defined in Definition \ref{Def}, we suppose that  $L $ is sufficiently small such that $(x_3^c-L,x_3^c+L) \subset (-M,M) $  and
$$
(B_h  \times (-M,M) ) \cap \Omega=S_h \times (-M,M),
$$
where $B_h\Subset \mathbb R^2 $ is  the central ball of radius $h \in \mathbb R_+$. Assume that $\mathbf{v},\mathbf{w}\in H^1(\Omega )^3$ are the generalized elastic transmission eigenfunctions to \eqref{eq:lame1}, where the Lam\'e constant $\lambda \neq 0$. Assume further that $q\mathbf{w} \in C^{\alpha}(\overline {S}_h \times [-M,M] )^{3}$ and  $\mathbf{v}-\mathbf{w} \in C^{1,\alpha}(\overline{S}_h\times [-M,M] )^{3}$,  where $\mathbf{0}$ is the vertex of $S_h$  defined in \eqref{eq:lame2}.  Write  $\mathbf{x}=(\mathbf{x}', x_3) \in \mathbb{R}^{3}, \, \mathbf x'\in \mathbb{R}^{2}$.  If the following conditions are fulfilled:
\begin{itemize}
\item[{\rm (a)}]~$\mathbf{v} $ can be approximated in $H^{1}(S_{h}\times(-M,M))^{3}$ by the Herglotz functions $\mathbf{v}_{j},j=1,2,...,$ with kernels $\mathbf g_{jp}$ and  $\mathbf g_{js}$ satisfying \eqref{eq:vj aprrox 3};
\\
\item[{\rm (b)}]the function $\eta=\eta(\mathbf{x}')$ is independent  of $x_3$ and
\begin{equation}\label{eq:3 eta}
\eta(\mathbf{0})\neq0,
\end{equation}
\item[{\rm (c)}]the angles $\theta_{m}$ and $\theta_{M}$ of the sector $S_h$ satisfy
\begin{equation}\label{eq:4}
-\pi<\theta_{m}<\theta_{M}<\pi \mbox{ and } \theta_{M}-\theta_{m}\neq\pi ,
\end{equation}
\end{itemize}
then for every edge points  $(\mathbf{0},x_3^c) \in \mathbb{R}^3 $ of $S_h \times (-M,M)$, namely ${x}_3^c\in (-M,M)$, one has
\begin{equation}\label{eq:36}
	\lim_{ \rho \rightarrow +0 }\frac{1}{m(B((\mathbf{0},x_3^c), \rho  )\cap\Omega) } \int_{m(B((\mathbf{0},x_3^c), \rho  )\cap\Omega) } |\mathbf{v}(\mathbf{x})| {\rm d} \mathbf{x}=0,
\end{equation}
	where $m(B((\mathbf{0},x_3^c), \rho  )\cap\Omega)$ is the measure of $B((\mathbf{0},x_3^c) ,\rho )\cap\Omega$.
\end{thm}


\begin{rem}\label{rem:22}

Similar to Remark \ref{rem:23} in the 2D case, a more general Fourier extension property (cf. \eqref{eq:cond thm3 gene}) can be proposed study the vanishing property 
in Theorem~\ref{thm:31}. However, this will involve rather lengthy and complicate analysis, and we choose not to explore more along that direction in this paper. 
\end{rem}

\begin{proof}[Proof of Theorem \ref{thm:31}]

As remarked earlier, it is sufficient for us to consider $\mathbf{v}_{\sf R}$ and $\mathbf{w}_{\sf R}$, and prove that \eqref{eq:36} holds for $\mathbf{v}_{\sf R}$. 
Since $q\mathbf{w} \in C^{\alpha }(\overline{S}_h \times [-M,M])^{3}$, it is easy to see that $q\mathbf{w}_\sfr \in C^{\alpha }(\overline{S}_h \times [-M,M])^{3}$. Similarly, we know that  $\mathbf{v}_\sfr-\mathbf{w}_\sfr\in C^{1,\alpha} (\overline{S_h}\times [-M,M] )^3$ under the assumption  $\mathbf{v}-\mathbf{w}\in C^{1,\alpha} (\overline{S_h}\times [-M,M] )^3$.

Since  the Herglotz functions $\mathbf{v}_{j},j=1,2,...,$ defined in \eqref{eq:vj3} with  kernels $\mathbf g_{jp}$ and  $\mathbf g_{js}$   can approximate $\mathbf v$ under the condition \eqref{eq:vj aprrox 3}, we see that $\mathbf v_\sfr \in H^1(S_h \times (-M,M))^3$ can be approximated by $\{\mathbf v_j^\sfr \}_{j=1}^{+\infty}$ defined in \eqref{eq:vr j 3} satisfying \eqref{eq:vj aprrox 3}. Therefore the assumptions in Lemmas \ref{lem34}--\ref{lem39} are fulfilled.

We divide the proof into two parts. 

\medskip
\noindent {\bf Part I.}~First we shall prove that
\begin{equation}\label{eq:vj vanish}
 \begin{aligned}
\lim_{j\rightarrow +\infty}\mathbf{v}_{jp}^{(1,2)}(\mathbf{0})+\mathbf{v}_{js}^{(1,2)}(\mathbf{0})=\mathbf{0},
 \end{aligned}
\end{equation}
where $\mathbf{v}_{jp}^{(1,2)}(\mathbf{0})$ and $\mathbf{v}_{js}^{(1,2)}(\mathbf{0})$ are defined in \eqref{eq:vj12 357}.

By Lemmas \ref{lemma3.2} and \ref{lem:pde decouple}, we consider the PDE system \eqref{eq:2.}. Recall that $C(\phi)$ and $\mu(\theta)$ are defined in \eqref{C} and \eqref{eq:I311} respectively. Substituting  \eqref{int23},  \eqref{eq:491} and  \eqref{100} into  \eqref{4},  after arranging  terms, we have
\begin{equation}
 \begin{aligned}\label{chang}
 &\eta(\mathbf{0})
2s^{-2}\bigg(C(\phi)\begin{pmatrix}
    1\\
   {\rm i}
\end{pmatrix}
\cdot\mathbf{v}_{jp}^{(1,2)}(\mathbf{0})
+
C(\phi)\begin{pmatrix}
    1\\
   {\rm i}
\end{pmatrix}
\cdot\mathbf{v}_{js}^{(1,2)}(\mathbf{0})\bigg)
\\ &\ \bigg(-s\sqrt{h}e^{-s\sqrt{h}\mu(\theta_{m})}\mu^{-1}(\theta_{m})-e^{-s\sqrt{h}\mu(\theta_{m})}\mu^{-2}(\theta_{m})+\mu^{-2}(\theta_{m})
\\& \ \ \ -s\sqrt{h}e^{-s\sqrt{h}\mu(\theta_{M})}\mu^{-1}(\theta_{M})-e^{-s\sqrt{h}\mu(\theta_{M})}\mu^{-2}(\theta_{M})+\mu^{-2}(\theta_{M})\bigg),
\\
 =&-\eta(\mathbf{0})(\mathcal{\widetilde{I}}^{\pm}_{23,p}+\mathcal{\widetilde{I}}^{\pm}_{23,s}+\mathcal{I}^{\pm}_{25,p}
 +\mathcal{\widetilde{I}}^{\pm}_{25,s}
+\mathcal{\widetilde{I}}^{\pm}_{211}+\mathcal{\widetilde{I}}^{\pm}_{212}+\mathcal{\widetilde{I}}^{\pm}_{221}
 +\mathcal{\widetilde{I}}^{\pm}_{222}  )
\\&-\mathcal{\widetilde{I}}^{\pm}_{1}+\widetilde{I}_{\Lambda_h } - \widetilde{I}_1-\widetilde{I}_2-\widetilde{I}_{\pm}^\Delta,
\end{aligned}
\end{equation}
where $\mathcal{\widetilde{I}}^{\pm}_{2\ell,\beta }$ ($\ell=3,5$, $\beta=p,s$), $\mathcal{\widetilde{I}}^{\pm}_{21\ell}$ ($\ell=1,2$), $\mathcal{\widetilde{I}}^{\pm}_{22\ell}$ ($\ell=1,2$), $\mathcal{\widetilde{I}}^{\pm}_{1}$, $\widetilde{I}_{\Lambda_h }$, $ \widetilde{I}_\ell$ ($\ell=1,2$), $\widetilde{I}_{\pm}^\Delta$ are defined in \eqref{l39},  \eqref{eq:491}, \eqref{int23}, \eqref{eq:10}  respectively.

Multiplying $s^{2}$ on the both side of \eqref{chang}, by virtue of \eqref{e3},  \eqref{eq:20}, \eqref{eq:I1},   \eqref{eq:deuvj1 3d} and \eqref{l39}, and letting $s=j^{\varrho/2}$ ($\max\{\beta_{1}, \beta_{2}\}<\varrho<\gamma$) with $j\rightarrow +\infty$, we have
\begin{equation}\notag
 \begin{aligned}
\lim_{j\rightarrow \infty}C(\phi)\eta(\mathbf{0})
\bigg(
\mathbf{v}_{jp}^{(1,2)}(\mathbf{0})
+
\mathbf{v}_{js}^{(1,2)}(\mathbf{0})\bigg)\cdot
\begin{pmatrix}
    1\\
   {\rm i}
\end{pmatrix}
\bigg(\mu^{-2}(\theta_{m})+\mu^{-2}(\theta_{M})\bigg)=0,
 \end{aligned}
\end{equation}
Recall that the definition of $\phi$ in Definition \ref{Def}, we know that $C(\phi)> 0$. Under the condition \eqref{eq:4}, from \cite[Lemma 2.10]{DCL}, we know that
\begin{equation}\notag
\mu^{-2}(\theta_{m})+\mu^{-2}(\theta_{M})\neq 0.
\end{equation}
Under the condition \eqref{eq:3 eta}, since $ \eta$ is real valued function, we have
\begin{equation}\notag
 \begin{aligned}
\lim_{j\rightarrow \infty}
\bigg(
\mathbf{v}_{jp}^{(1,2)}(\mathbf{0})
+
\mathbf{v}_{js}^{(1,2)}(\mathbf{0})\bigg)\cdot
\begin{pmatrix}
    1\\
   {\rm i}
\end{pmatrix}
=0,
 \end{aligned}
\end{equation}
which readily implies \eqref{eq:vj vanish} by noting that $\mathbf{v}_{jp}^{(1,2)}(\mathbf{0})
+
\mathbf{v}_{js}^{(1,2)}(\mathbf{0})$ is a real vector.

\medskip
\noindent {\bf Part II.}~By Lemma \ref{lem:314 third}, we have \eqref{eq:vj3 zero}. Combining \eqref{eq:vj vanish} with \eqref{eq:vj3 zero}, by using the definition of $ \mathbf{v}^\sfr_j(\mathbf{0})$ in \eqref{0}, we can prove that
\begin{equation}\notag
 \begin{aligned}
\lim_{j\rightarrow \infty} \mathbf{v}^\sfr_j(\mathbf{0})=\mathbf{0}.
 \end{aligned}
\end{equation}
Using a similar argument to that for \eqref{finally}, we can finish the proof of this theorem.
\end{proof}

We next consider the degenerate case of Theorem \ref{thm:31} with $\eta\equiv 0$.

%

\begin{cor}\label{cor:3.1}
Consider the same setup in Theorem~\ref{thm:31} but with $\eta\equiv 0$. If the following conditions are fulfilled:
\begin{itemize}
\item[{\rm (a)}]~ for any given constants  $\gamma>\max\{\beta_{1}, \beta_{2}\}>0$, there exits a sequence of Herglotz functions   $\mathbf{v}_{j},j=1,2,...,$ with kernels $\mathbf g_{jp}$ and  $\mathbf g_{js}$ can approximate  $\mathbf v$ in $H^{1}(S_{h} \times (-M,M))$   satisfying
\begin{equation}\label{eq:388 assump}
\| \mathbf{v}-\mathbf{v}_{j}\|_{H^{1}(S_{h} \times (-M,M))^{3}}\leq j^{-\gamma},\| \mathbf{g}_{jp}\|_{L^{2}(\mathbb{S}^{2})^{3}}\leq j^{\beta_{1}} \ and \ \| \mathbf{g}_{js}\|_{L^{2}(\mathbb{S}^{2})^{3}}\leq j^{\beta_{2}};
\end{equation}
\item[{\rm (b)}]the angles $\theta_{m}$ and $\theta_{M}$ of the sector $W$ satisfy
\begin{equation}\label{33eq:ass2 int}
-\pi<\theta_{m}<\theta_{M}<\pi\quad  and \quad \theta_{M}-\theta_{m}\neq\pi;
\end{equation}
\end{itemize}
then for every edge points  $(\mathbf{0},x_3^c) \in \mathbb{R}^3 $ of $S_h \times (-M,M)$, namely ${x}_3^c\in (-M,M)$, one has
\begin{equation}\label{eq:llnn3}
\lim_{\rho\rightarrow+0 }\frac{1}{m(B((\mathbf{0},x_3^c),\rho)\cap P(\Omega))}\int_{B((\mathbf{0},x_3^c),\rho)\cap P(\Omega))}\mathcal{R}(V\mathbf{w})(\mathbf{x}'){\rm d }\mathbf{ x}'=\mathbf{0},
\end{equation}
 where $P(\Omega )$ is the projection set of $\Omega$ on $\mathbb R^2$.

\end{cor}
\begin{proof}
The proof follows from the one of Theorem \ref{thm:31} with some necessary modifications. It is sufficient for us to show that \eqref{eq:llnn3} holds for $\mathbf{w}_{\sf R}$. Clearly, we have
\begin{equation}\label{eq:390 cond}
	 q\mathbf{w}_\sfr\in C^{\alpha}(\overline{S_{h}})^2,\quad \mathbf{v}_\sfr-\mathbf{w}_\sfr\in C^{\alpha}(\overline{S}_{h}\times [-M,M])^3, 
\end{equation}
and the condition \eqref{eq:388 assump} gives
\begin{equation}\notag
\| \mathbf{v}^\sfr-\mathbf{v}^\sfr_{j}\|_{H^{1}(S_{h} \times (-M,M))^{3}}\leq j^{-\gamma},\| \mathbf{g}_{jp}\|_{L^{2}(\mathbb{S}^{2})^{3}}\leq j^{\beta_{1}} \ and \ \| \mathbf{g}_{js}\|_{L^{2}(\mathbb{S}^{2})^{3}}\leq j^{\beta_{2}},
\end{equation}
for any given constants  $\gamma>\max\{\beta_{1}, \beta_{2}\}>0$.

According to the proof of Theorem~\ref{thm:31}, we divide the proof into two parts.  

\medskip
\noindent {\bf Part I.}~First, let  $\mathbf{v}^{(1,2)}$ and $\mathbf{w}^{(1,2)}$ be defined in \eqref{eq:defvw}, which fulfill \eqref{eq:2.} with $\eta \equiv 0$. From \eqref{4}, it follows that
\begin{equation}\label{corint0}
	\widetilde{I}_1+\widetilde{I}_2=\widetilde{I}_{\Lambda_h },
\end{equation}
where $\widetilde{I}_1$, $\widetilde{I}_2$ and $\widetilde{I}_{\Lambda_h } $ are defined in \eqref{eq:10}.
From \eqref{eq:lame5}, one can see that
\begin{equation}\label{eq:78}
 \begin{aligned}
\int_{S_{h}}\mathbf{u}(\mathbf{x}'){\rm d}\mathbf{x'}
&=\int_{W}\mathbf{u}(\mathbf{x}'){\rm d}\mathbf{x'}
+\int_{W\setminus S_{h}}\mathbf{u}(\mathbf{x}'){\rm d}\mathbf{x'}
\\&=6\mathrm i(e^{-2\theta_{M}{\rm i}}-e^{-2\theta_{m}i})s^{-4}
\begin{pmatrix}
    1\\
   {\rm i}
\end{pmatrix}
+ \int_{W\setminus S_{h}}\mathbf{u}(\mathbf{x}'){\rm d}\mathbf{x'}.
\end{aligned}
\end{equation}

Under the condition \eqref{eq:390 cond}, by Lemma \ref{lem:c alpha}, we know that $\mathbf f_2(\mathbf x')$, $\mathbf f_3(\mathbf x')$ and $\mathbf f_4(\mathbf x')$ have the expansion \eqref{eq:28}, where $\mathbf f_\ell(\mathbf x')$ is defined in \eqref{eq:10}. Therefore $\widetilde{I}_1$ can be rewritten as \eqref{eq:29}.

Since $\mathbf{v}_\sfr=\mathbf{w}_\sfr  $ on $  \Gamma^{\pm}_h\times(-M,M)$, recalling  the definitions of $\mathbf{f}_{3}$ and  $\mathbf{f}_{4}$ in \eqref{eq:10},  it is easy to see that
\begin{equation}\label{eq:294 f3f4}
 \begin{aligned}
 &\mathbf{f}_{3}(\mathbf{0})=-\int^{L}_{-L}\phi''(x_{3})
\begin{pmatrix}
    \lambda (v_{1}(\mathbf{0}, x_3)-w_{1}(\mathbf{0}, x_3)) \\
    \lambda (v_{2}(\mathbf{0}, x_3)- w_{2}(\mathbf{0}, x_3))
\end{pmatrix}
(\mathbf{0},x_{3}){\rm d}x_{3}=\mathbf{0},
\\
 &\mathbf{f}_{4}(\mathbf{0})=(\lambda+\mu)\int^{L}_{-L}\phi'(x_{3})
\begin{pmatrix}
    \partial_{1} (v_{3}(\mathbf{0}, x_3)-  w_{3}(\mathbf{0}, x_3))          \\
     \partial_{2}( v_{3}(\mathbf{0}, x_3)   -  w_{3}(\mathbf{0}, x_3) )
\end{pmatrix}
(\mathbf{0},x_{3}){\rm d}x_{3}=\mathbf{0}. 
\end{aligned}
\end{equation}
By virtue of \eqref{eq:294 f3f4}, substituting \eqref{eq:29} and \eqref{eq:78} into \eqref{corint0}, we can deduce that
\begin{equation}\label{eq:79}
 \begin{aligned}
&6\mathrm i(e^{-2\theta_{M}{\rm i}}-e^{-2\theta_{m}i})s^{-4}
\begin{pmatrix}
    1\\
   {\rm i}
\end{pmatrix}\cdot
(\mathbf{f}_{1j}(\mathbf{0})+\mathbf{f}_{2}(\mathbf{0}))
\\=&-\widetilde{I}_2+\widetilde{I}_{\Lambda_h }-(\mathbf{f}_{1j}(\mathbf{0})+\mathbf{f}_{2}(\mathbf{0}))\int_{W\setminus S_{h}}\mathbf{u}(\mathbf{x}'){\rm d}\mathbf{x'}
-\int_{S_{h}}\delta \mathbf{f}_{1j}\cdot\mathbf{u}(\mathbf{x}'){\rm d}\mathbf{x}'
\\&-\int_{S_{h}}\delta \mathbf{f}_{2}\cdot\mathbf{u}(\mathbf{x}'){\rm d}\mathbf{x}'
-\int_{S_{h}}\delta \mathbf{f}_{3}\cdot\mathbf{u}(\mathbf{x}'){\rm d}\mathbf{x}'
-\int_{S_{h}}\delta \mathbf{f}_{4}\cdot\mathbf{u}(\mathbf{x}'){\rm d}\mathbf{x}'.
 \end{aligned}
\end{equation}
In \eqref{eq:79}, we take 
\begin{equation}\label{eq:396 s}
	s=j^{\varrho/4}, \quad \max\{\beta_{1}, \beta_{2}\}<\varrho<\gamma ,
\end{equation}
 and multiply $j^{\varrho}$ on the both sides of \eqref{eq:79}, then it yields that
 \begin{equation}
 \begin{aligned}
&6{\rm i}(e^{-2\theta_{M}{\rm i}}-e^{-2\theta_{m}i})
\begin{pmatrix}
    1\\
   {\rm i}
\end{pmatrix}\cdot
(\mathbf{f}_{1j}(\mathbf{0})+\mathbf{f}_{2}(\mathbf{0}))
\\=&j^{\varrho}\bigg(-\widetilde{I}_2+\widetilde{I}_{\Lambda_h }-(\mathbf{f}_{1j}(\mathbf{0})+\mathbf{f}_{2}(\mathbf{0}))\int_{W\setminus S_{h}}\mathbf{u}(\mathbf{x}'){\rm d}\mathbf{x'}
-\int_{S_{h}}\delta \mathbf{f}_{1j}\cdot\mathbf{u}(\mathbf{x}'){\rm d}\mathbf{x}'
\\&-\int_{S_{h}}\delta \mathbf{f}_{2}\cdot\mathbf{u}(\mathbf{x}'){\rm d}\mathbf{x}'
-\int_{S_{h}}\delta \mathbf{f}_{3}\cdot\mathbf{u}(\mathbf{x}'){\rm d}\mathbf{x}'
-\int_{S_{h}}\delta \mathbf{f}_{4}\cdot\mathbf{u}(\mathbf{x}'){\rm d}\mathbf{x}'\bigg).
 \end{aligned}
\end{equation}
By virtue of \eqref{eq:20} in Lemma \ref{lem35}, we have
\begin{equation}\label{001}
 \begin{aligned}
j^{\varrho}|-\widetilde{I}_2|\leq
 h\sqrt{2L(\theta_{M}-\theta_{m})}e^{-s\sqrt{\Theta}\delta_{W}} \|\phi\|_{L^\infty} j^{-\gamma+\varrho}.
 \end{aligned}
\end{equation}
Under the assumption \eqref{eq:388 assump},  in \eqref{eq:30}, \eqref{eq:31}, \eqref{eq:32} and \eqref{eq33}, in view of \eqref{eq:396 s}, we can obtain the following estimates
\begin{equation}\label{002}
 \begin{aligned}
j^{\varrho} \left |\int_{S_{h}}\delta \mathbf{f}_{1j}  \cdot\mathbf{u}(\mathbf{x}'){\rm d}\mathbf{x}'\right |\leq&
 \omega^2 {\rm diam}(S_h)^{1-\alpha } 4 L\sqrt{\pi}\|\phi\|_{L^\infty}
\bigg(2(1+k_p)j^{\beta_1-\frac{1}{2}\alpha\varrho-\varrho}
\\&+2(1+k_s)j^{\beta_2-\frac{1}{2}\alpha\varrho-\varrho}\bigg), \ \
\\
	j^{\varrho} \left |\int_{S_{h}}\delta \mathbf{f}_{2}  \cdot\mathbf{u}(\mathbf{x}'){\rm d}\mathbf{x}'\right |
&\leq\| \mathbf{f}_{2}\|_{C^{\alpha}(S_{h})^{2}}\frac{2\sqrt{2}(\theta_{M}-\theta_{m})\Gamma(2\alpha+4)}{\delta_{W}^{2\alpha+4}}j^{-\frac{1}{2}\alpha\varrho-\varrho},\\
j^{\varrho}\left |\int_{S_{h}}\delta \mathbf{f}_{3}  \cdot\mathbf{u}(\mathbf{x}'){\rm d}\mathbf{x}'\right |
&\leq\| \mathbf{f}_{3}\|_{C^{\alpha}(S_{h})^{2}}\frac{2\sqrt{2}(\theta_{M}-\theta_{m})\Gamma(2\alpha+4)}{\delta_{W}^{2\alpha+4}}j^{-\frac{1}{2}\alpha\varrho-\varrho},\\
j^{\varrho}\left |\int_{S_{h}}\delta \mathbf{f}_{4}  \cdot\mathbf{u}(\mathbf{x}'){\rm d}\mathbf{x}'\right |
&\leq\| \mathbf{f}_{4}\|_{C^{\alpha}(S_{h})^{2}}\frac{2\sqrt{2}(\theta_{M}-\theta_{m})\Gamma(2\alpha+4)}{\delta_{W}^{2\alpha+4}}j^{-\frac{1}{2}\alpha\varrho-\varrho}.
 \end{aligned}
\end{equation}
Under the  assumption \eqref{33eq:ass2 int}, it is easy to see that
$$
\left| e^{-2\theta_M \mathrm i }-e^{-2\theta_m\mathrm  i }  \right| =\left|1-e^{-2(\theta_M -\theta_m) \mathrm i } \right| \neq 0,
$$
since $\theta_M -\theta_m
\neq \pi$.
Leting $ j\rightarrow \infty$, by virtue of \eqref{eq:lame7},  \eqref{e3}, \eqref{001} and \eqref{002}, we have
$$\lim_{j\rightarrow\infty}\mathbf{f}_{1j}(\mathbf{0})=-\mathbf{f}_{2}(\mathbf{0}).$$
Since $\mathbf{f}_{1j}(\mathbf{x}')=-\omega ^{2}\mathcal{R}(\mathbf{v}_{j}^{(1,2)})(\mathbf{x}') $ and $ \mathbf{f}_{2}(\mathbf{x}')= \omega ^{2}\mathcal{R}(q\mathbf{w}^{(1,2)})(\mathbf{x}')$, we obtain that
\begin{equation}\label{000}
 \begin{aligned}
\lim_{j\rightarrow\infty}\mathcal{R}(\mathbf{v}_{j}^{(1,2)})(\mathbf{0})=\mathcal{R}(q\mathbf{w}^{(1,2)})(\mathbf{0}).
\end{aligned}
\end{equation}
From the boundary condition in \eqref{eq:2.} for $\eta \equiv 0$, we have
\begin{equation}\label{eq:3101 vw}
 \begin{aligned}
&\lim_{ \rho \rightarrow +0 }\frac{1}{m(B(\mathbf{0}, \rho  )\cap P(\Omega))} \int_{B(\mathbf{0}, \rho )\cap P(\Omega)}   \mathcal{R}(\mathbf{v}^{(1,2)}_{j})(\mathbf{x}') {\rm d} \mathbf{x} '\\=& \lim_{ \rho \rightarrow +0 }\frac{1}{m(B(\mathbf{0}, \rho  )\cap P(\Omega))} \int_{B(\mathbf{0}, \rho )\cap P(\Omega)} \mathcal{R}(\mathbf{w}^{(1,2)})(\mathbf{x}')   {\rm d} \mathbf{x}'. 
\end{aligned}
\end{equation}
Furthermore, it yields that
\begin{equation}\label{eq:80}
 \begin{aligned}
\lim_{j \rightarrow \infty} \mathcal{R}(\mathbf{v}_{j}^{(1,2)})(\mathbf{0})&=\lim_{j \rightarrow \infty}  \lim_{ \rho \rightarrow +0 }\frac{1}{m(B(\mathbf{0}, \rho  ))\cap P(\Omega)} \int_{B(\mathbf{0}, \rho )\cap P(\Omega)} \mathcal{R}(\mathbf{v}^{(1,2)}_{j})(\mathbf{x}') {\rm d} \mathbf{x}'\\
	&= \lim_{ \rho \rightarrow +0 }\frac{1}{m(B(\mathbf{0}, \rho  ))\cap P(\Omega)} \int_{B(\mathbf{0}, \rho )\cap P(\Omega)}  \mathcal{R}(\mathbf{v}^{(1,2)})(\mathbf{x}'){\rm d} \mathbf{x}',\\
	\mathcal{R}(q\mathbf{w}^{(1,2)})(\mathbf{0})&= \lim_{ \rho \rightarrow +0 }\frac{1}{m(B(\mathbf{0}, \rho  ))\cap P(\Omega)} \int_{B(\mathbf{0}, \rho )\cap P(\Omega)} \mathcal{R}(q\mathbf{w}^{(1,2)})(\mathbf{x}') {\rm d} \mathbf{x}'.
\end{aligned}
\end{equation}
Due to \eqref{000}, combining \eqref{eq:3101 vw} with \eqref{eq:80}, we can prove that
\begin{equation}\label{eq:3103 c1}
	\lim_{\rho\rightarrow+0 }\frac{1}{m(B(\mathbf{0},\rho)\cap P(\Omega))}\int_{B(\mathbf{0},\rho)\cap P(\Omega)}\mathcal{R}(V\mathbf{w}^{(1,2)})(\mathbf{x}'){\rm d }\mathbf{ x}'=0.
\end{equation}

\medskip
\noindent {\bf Part II.}~ Since the Lam\'e constant $\lambda \neq 0$, similar to Lemma \ref{lem:314 third},  consider the PDE system \eqref{3} for $\eta \equiv 0$, under the conditions 
\eqref{eq:388 assump} and \eqref{33eq:ass2 int},  using a similar argument of \cite[Corollary 3.1]{DCL}, we can prove that
\begin{equation}\label{eq:3103 c2}
	\lim_{\rho\rightarrow+0 }\frac{1}{m(B(\mathbf{0},\rho)\cap P(\Omega))}\int_{B(\mathbf{0},\rho)\cap P(\Omega)}\mathcal{R}(Vw_3)(\mathbf{x}'){\rm d }\mathbf{x}'=0,
\end{equation}
where $w_3$ is the third component of $\mathbf w_\sfr$. According to \eqref{eq:3103 c1} and  \eqref{eq:3103 c2}, we finish the proof of this corollary.
\end{proof}

\begin{rem}\label{rem:34 holder}
According to Corollary \ref{cor:3.1}, the average value of each component of the function $V\mathbf{w}$ over the cylinder centered at the edge point $(\mathbf{0},x_3^c)$ with the height $L$ vanishes in the distribution sense. Moreover, if we assume that $V(\mathbf{x}',x_3)$ is continuous near the edge point $(\mathbf{0},x_3^c)$ where $x_3^c \in (-M,M)$ and $ V(\mathbf{0},x_3^c ) \neq 0$, the by the dominant convergence theorem and the definition of $\mathcal R$, one can show that
\begin{equation}\label{eq:cor31 con}
	\lim_{ \rho \rightarrow +0 }\frac{1}{m(B(\mathbf{0}, \rho  )\cap P(\Omega ) )} \int_{B(\mathbf{0}, \rho ) \cap P(\Omega) }\left( \int_{x_3^c-L}^{x_3^c+L} \phi(x_3)  \mathbf{w } (\mathbf{x}',x_3) {\rm d} x_3  \right) {\rm d} \mathbf{x}'=\mathbf 0
\end{equation}
under the assumptions  in  Corollary \ref{cor:3.1}. Since $\mathbf w \in H^1(\Omega )^3$, it can be readily seen that
$$
	\lim_{ \rho \rightarrow +0 }\frac{1}{m(B(\mathbf{0}, \rho  ) \cap P(\Omega) )} \int_{B(\mathbf{0}, \rho ) \cap P(\Omega) }   \mathbf{w} (\mathbf{x}',x_3)  {\rm d}\mathbf{ x}'=\mathbf 0,\quad \forall x_3 \in (-M,M),
		$$
which also describes the vanishing property of the interior elastic transmission eigenfunctions $\mathbf{v}$ and $\mathbf{w}$ near the edge point in 3D. We would like to point out that $\mathbf{v}$ and $\mathbf{w}$ must vanish at an edger corner point $(\mathbf 0, x_3^c)$ with $x_3^c\in (-M,M)$  if $V$ and either one of $\mathbf v,\, \mathbf w$  are $C^\alpha$ smooth near the corner, $\mathbf v- \mathbf w \in H^2(\Omega )^3$, $V((\mathbf 0, x_3^c) )\neq 0$ (cf. \cite[Theorem 1.5]{EBL}). Indeed, the assumption that $V$ and $\mathbf w$ are $C^\alpha$ smooth near the corner can imply $q\mathbf w$ is $C^\alpha$ smooth near the corner. Compared Corollary \ref{cor:3.1} with \cite[Theorem 1.5]{EBL}, we remove the assumption $\mathbf v- \mathbf w \in H^2(\Omega )^3$ and establish  \eqref{eq:cor31 con}. 
\end{rem}

Similar to Corollary \ref{cor:3.1}, if we assume that $\mathbf v$ has the $C^\alpha$-regularity near a 3D edge corner, we can use a similar argument of Corollary \ref{cor:3.1} to prove that $\mathbf v$  and $\mathbf w$ must vanish at the underlying edge corner point. 
\begin{cor}\label{cor:34 eta}
Consider the same setup in Corollary \ref{cor:3.1}, but assume that $\mathbf{v} \in C^{\alpha}(\overline {S}_h \times [-M,M] )^{3}$,  $q\mathbf{w} \in C^{\alpha}(\overline {S}_h \times [-M,M] )^{3}$ and  $\mathbf{v}-\mathbf{w} \in C^{1,\alpha}(\overline{S}_h\times [-M,M] )^{3}$, for $0<\alpha<1$. If $V(\mathbf{x}',x_3)=q(\mathbf{x}',x_3)-1$ is continuous near the edge point $(\mathbf{0},x_3^c)$ with $ V(\mathbf{0},x_3^c ) \neq 0$ for $x_3^c \in (-M,M)$, then for every edge points  $(\mathbf{0},x_3^c) \in \mathbb{R}^3 $ of $S_h \times (-M,M)$, one has
\begin{equation}\label{eq:cor 33 109}
	\mathbf{v}((\mathbf{0},x_3^c))=\mathbf{w}((\mathbf{0},x_3^c))=\mathbf{0}. 
\end{equation}
\end{cor}

\begin{proof} According to Corollary \ref{cor:3.1}, we know that \eqref{eq:3103 c1} holds. 
	In what follows, let us consider two separate cases.

\medskip
\noindent {\bf Case 1: $\lambda \neq 0$.}~Combining  the argument in {\bf Part II} of the proof of Corollary \ref{cor:3.1} with a similar argument for Remark \ref{rem:34 holder}, one can show that \eqref{eq:cor 33 109} holds by noting that $\mathbf w$ is continuous near the edge point. 

\medskip
\noindent {\bf Case 2: $\lambda = 0$.}~  Since $\mathbf v_\sfr=\mathbf w_\sfr$ on $\Gamma_h^\pm \times (-L,L)$, we have $ \partial_\ell ( v_j- w_j )=0$ for $\ell=1,2$ and $j=1,2$, where  $v_j$ and  $w_j$ the $j-$th components of $\mathbf v_\sfr$ and $\mathbf w_\sfr$ respectively. Therefore, under $\lambda = 0$, subtracting the first equation of \eqref{3} from the second one of \eqref{3}, it yields that 
\begin{equation}\notag
	\mathcal R(qw_3)=\mathcal R(v_3) \mbox{ on } \Gamma_h^\pm,
\end{equation}
which can be used to further deduce that $ \mathcal R( w_3 (\mathbf 0))=0$. Therefore we have $w_3 (\mathbf 0, x_3^c)=0$ for all $x_3^c\in (-L,L)$. Due to \eqref{eq:3103 c1}, we can show that $\mathbf{v}_\sfr((\mathbf{0},x_3^c))=\mathbf{w}_\sfr((\mathbf{0},x_3^c))=\mathbf{0}$. Using a similar argument, we can prove \eqref{eq:cor 33 109} when $\lambda =0$. 

The proof is complete. 
\end{proof}

If we further require that $\mathbf v$ is H\"older continuous at the edge corner,  similar to Theorem \ref{thm:31}, we have
\begin{thm}\label{thm:33}
Consider the same setup in Theorem~\ref{thm:31}, but assume that $\mathbf{v}\in H^1 (\Omega )^{3} \cap C^\alpha(\overline {S}_h   \times [-M,M ] ) ^{3}$,   $q \mathbf{w}\in C^\alpha(\overline {S}_h   \times [-M,M ] ) ^{3}$ and $\eta \in C^\alpha(\overline{\Gamma}_h^\pm  \times [-M,M]  )$  for $0< \alpha  <1$. If $\eta$ is independent of $x_3$ and $\eta(\mathbf{0})\neq 0$, and the corner is non-degenerate, then $\mathbf{v}$ vanishes at the edge point $(\mathbf 0,x_3^c) \in \mathbb{R}^3$ of $S_h \times (-M,M)$, where $x_3^c \in (-M,M)$.
\end{thm}

\begin{proof}
It is sufficient to show the vanishing property for $\mathbf{v}_{\sf R}$. Clearly, one has $q\mathbf{w}_\sfr \in C^{\alpha }(\overline{S}_h \times [-M,M])^{3}$ and $\mathbf{v}_\sfr-\mathbf{w}_\sfr\in C^{1,\alpha} (\overline{S_h}\times [-M,M] )^3$. Therefore the assumptions in Lemmas \ref{lem34}--\ref{lem39} are fulfilled. We next divide the proof into two parts.  

\medskip

	\noindent {\bf Part I.}~Let $\mathbf{v}^{(1,2)})$ be defined in \eqref{eq:defvw}.  First we shall prove
\begin{align}\label{eq:thm 32 con1}
	\mathbf{v}^{(1,2)}(\mathbf{0})=\mathbf{0}. 
\end{align}
From \eqref{eq:green}, \eqref{5} and \eqref{6}, it follows that
\begin{equation}\label{09}
 \begin{aligned}
\int_{S_{h}}(\mathbf{f}_{1}+\mathbf{f}_{2}+\mathbf{f}_{3}+\mathbf{f}_{4}) \cdot\mathbf{u}(\mathbf{x}'){\rm d}\mathbf{x}'
=\widetilde{I}_{\Lambda_h}-\int_{\Gamma^{\pm}_{h}}\eta(\mathbf{x}')\mathcal{R}(\mathbf{v}^{(1,2)})\cdot\mathbf{u}(\mathbf{x}'){\rm d}\sigma
 \end{aligned}
\end{equation}
where $\widetilde{I}_{\Lambda_h }$ is defined in \eqref{eq:10}. 

Using the conditions  $q\mathbf{w}_\sfr \in C^{\alpha }(\overline{S}_h \times [-M,M])^{3}$ and $\mathbf{v}_\sfr-\mathbf{w}_\sfr\in C^{1,\alpha} ({S_h}\times (-M,M) )^3$, by Lemma \ref{lem:c alpha}, we know that $\mathbf f_\ell(\mathbf x') \in C^\alpha(\overline{S}_h )$, where $\mathbf f_\ell(\mathbf x')$ is defined in \eqref{eq:10}.  Since  $\eta \in C^\alpha\left(\overline{\Gamma}_h^\pm \right)$,    we know that  $\eta$, $\mathbf{f}_{2}, \mathbf{f}_{3}$, and $ \mathbf{f}_{4}$ have the expansions \eqref{etae} and \eqref{eq:28} around the origin. Recall the definition of $ \mathbf{f}_{1}$, i.e. $\mathbf{f}_{1}=-\omega^{2}\mathcal{R}(\mathbf{v}^{(1,2)})$. Due to the fact that  $\mathbf{v}^{(1,2)}\in C^\alpha(\overline { S}_h )$, we have the following expansions:
\begin{equation}\label{vre1}
 \begin{aligned}
&\mathcal{R}(\mathbf{v}^{(1,2)})(\mathbf{x}')=\mathcal{R}(\mathbf{v}^{(1,2)})(\mathbf{0})+\delta\mathcal{R}(\mathbf{v}^{(1,2)})(\mathbf{x}'),\ |\delta\mathcal{R}(\mathbf{v}^{(1,2)})|\leq \|\mathcal{R}(\mathbf{v}^{(1,2)})(\mathbf{x}')\|_{C^{\alpha}(S_h)^{2}}| \mathbf{x}'|^{\alpha}
.
 \end{aligned}
\end{equation}
Substituting \eqref{etae}, \eqref{eq:28} and \eqref{vre1} into \eqref{09}, we have
\begin{equation}\label{a}
 \begin{aligned}
&\bigg(-\omega^{2}\mathcal{R}(\mathbf{v}^{(1,2)})(\mathbf{0})+\mathbf {f}_2 (\mathbf{0})+\mathbf {f}_3 (\mathbf{0})+\mathbf {f}_4 (\mathbf{0})\bigg)\int_{S_{h}}\mathbf{u}(\mathbf{x}'){\rm d}\mathbf{x}'
\\&+\int_{S_{h}}\bigg(-\omega^{2}\delta\mathcal{R}(\mathbf{v}^{(1,2)})(\mathbf{x}')+\delta\mathbf {f}_2 (\mathbf{x}')+\delta\mathbf {f}_3 (\mathbf{x}')+\delta\mathbf {f}_4 (\mathbf{x}')\bigg)\cdot\mathbf{u}(\mathbf{x}'){\rm d}\mathbf{x}'
\\=&\widetilde{I}_{\Lambda_h }
-\eta(\mathbf{0})\mathcal{R}(\mathbf{v}^{(1,2)})(\mathbf{0})\int_{\Gamma_{h}^{\pm}}\mathbf{u}(\mathbf{x}'){\rm d}\sigma-\eta(\mathbf{0})\int_{\Gamma_{h}^{\pm}}\mathbf{u}(\mathbf{x}')\cdot\delta\mathcal{R}(\mathbf{v}^{(1,2)})(\mathbf{x}'){\rm d}\sigma
\\&-\mathcal{R}(\mathbf{v}^{(1,2)})(\mathbf{0})\int_{\Gamma_{h}^{\pm}}\delta\eta(\mathbf{x}')\mathbf{u}(\mathbf{x}'){\rm d}\sigma-\int_{\Gamma_{h}^{\pm}}\delta\eta(\mathbf{x}')\mathbf{u}(\mathbf{x}')\cdot\delta\mathcal{R}(\mathbf{v}^{(1,2)})(\mathbf{x}'){\rm d}\sigma.
\end{aligned}
\end{equation}
Combining \eqref{eq:lame6} with \eqref{eq:zeta} , we  have the following integral estimates:
\begin{align}
&\left|\int_{\Gamma_{h}^{\pm}}\delta\eta(\mathbf{x}')\mathbf{u}(\mathbf{x}'){\rm d}\sigma\right|
\leq\sqrt{2}
\|\eta\|_{C^{\alpha}}
\int_{0}^{h}r^{\alpha}e^{-s\sqrt{r}\cos\frac{\theta_{m}}{2}}{\rm d}r=O(s^{-2\alpha-2}), \label{a1} \\
 & \left|\int_{\Gamma_{h}^{\pm}}\mathbf{u}(\mathbf{x}')\cdot\delta\mathcal{R}(\mathbf{v}^{(1,2)})(\mathbf{x}'){\rm d}\sigma\right|
\leq\sqrt{2}
\|\mathcal{R}(\mathbf{v}^{(1,2)})\|_{C^{\alpha}}
\int_{0}^{h}r^{\alpha}e^{-s\sqrt{r}\cos\frac{\theta_{m}}{2}}{\rm d}r\notag \\
&\quad =O(s^{-2\alpha-2}), \label{a2} \\
&\left|\int_{\Gamma_{h}^{\pm}}\delta\eta(\mathbf{x}')\mathbf{u}(\mathbf{x}')\cdot\delta\mathcal{R}(\mathbf{v}^{(1,2)})(\mathbf{x}'){\rm d}\sigma\right|\notag
\leq \sqrt{2}\|\eta\|_{C^{\alpha}}
\|\mathcal{R}(\mathbf{v}^{(1,2)})\|_{C^{\alpha}}
\int_{0}^{h}r^{2\alpha}e^{-s\sqrt{r}\cos\frac{\theta_{m}}{2}}{\rm d}r\\
&\quad =O(s^{-4\alpha-2}),\label{a3}\\
&\left|\int_{S_{h}}\delta\mathcal{R}(\mathbf{v}^{(1,2)})(\mathbf{x}')\cdot\mathbf{u}(\mathbf{x}'){\rm d}\mathbf{x}'\right|\label{a4}
\leq\sqrt{2}\|\mathcal{R}(\mathbf{v}^{(1,2)})\|_{C^{\alpha}}
\int_{W} |u_{1}( \mathbf{x }')||\mathbf{x}|^{\alpha}{\rm d}\mathbf{x}'
\\
&\quad \leq2\sqrt{2}\|\mathcal{R}(\mathbf{v}^{(1,2)})\|_{C^{\alpha}}\frac{(\theta_{M}-\theta_{m})\Gamma(2\alpha+4)}{\delta_{W}^{2\alpha+4}}s^{-2\alpha-4},\notag \\
&\left|\int_{S_{h}}\delta \mathbf{f}_2(\mathbf{x}')\cdot\mathbf{u}(\mathbf{x}'){\rm d}\mathbf{x}'\right|
\leq \sqrt{2}\|\mathbf{f}_2\|_{C^{\alpha}}
\int_{W} |u_{1}( \mathbf{x }')||\mathbf{x}|^{\alpha}{\rm d}\mathbf{x}' \notag
\\
&\quad \leq 2\sqrt{2}\|\mathbf{f}_2\|_{C^{\alpha}}\frac{(\theta_{M}-\theta_{m})\Gamma(2\alpha+4)}{\delta_{W}^{2\alpha+4}}s^{-2\alpha-4}, \label{a5} \\
&\left|\int_{S_{h}}\delta\mathbf{f}_3(\mathbf{x}')\cdot\mathbf{u}(\mathbf{x}'){\rm d}\mathbf{x}'\right|
\leq\sqrt{2}\|\mathbf{f}_3\|_{C^{\alpha}}
\int_{W} |u_{1}( \mathbf{x }')||\mathbf{x}|^{\alpha}{\rm d}\mathbf{x}' \notag
\\
&\quad \leq 2\sqrt{2}\|\mathbf{f}_3\|_{C^{\alpha}}\frac{(\theta_{M}-\theta_{m})\Gamma(2\alpha+4)}{\delta_{W}^{2\alpha+4}}s^{-2\alpha-4}, \label{a6}\\
&\left|\int_{S_{h}}\delta\mathbf{f}_4(\mathbf{x}')\cdot\mathbf{u}(\mathbf{x}'){\rm d}\mathbf{x}'\right|
\leq \sqrt{2}\|\mathbf{f}_4\|_{C^{\alpha}}
\int_{W} |u_{1}( \mathbf{x }')||\mathbf{x}|^{\alpha}{\rm d}\mathbf{x}' \notag
\\
& \quad \leq 2\sqrt{2}\|\mathbf{f}_4\|_{C^{\alpha}}\frac{(\theta_{M}-\theta_{m})\Gamma(2\alpha+4)}{\delta_{W}^{2\alpha+4}}s^{-2\alpha-4},\label{a7}
\end{align}
as $s\rightarrow +\infty$.
From \eqref{a}, after rearranging terms, we have
\begin{equation}\label{a0}
 \begin{aligned}
&\eta(\mathbf{0})\mathcal{R}(\mathbf{v}^{(1,2)})(\mathbf{0})\int_{\Gamma_{h}^{\pm}}\mathbf{u}(\mathbf{x}'){\rm d}\sigma
\\=&\widetilde{I}_{\Lambda_h }-\bigg(-\omega^{2}\mathcal{R}(\mathbf{v}^{(1,2)})(\mathbf{0})+\mathbf {f}_2 (\mathbf{0})+\mathbf {f}_3 (\mathbf{0})+\mathbf {f}_4 (\mathbf{0})\bigg)\int_{S_{h}}\mathbf{u}(\mathbf{x}'){\rm d}\mathbf{x}'
\\&-\int_{S_{h}}\bigg(-\omega^{2}\delta\mathcal{R}(\mathbf{v}^{(1,2)})(\mathbf{x}')+\delta\mathbf {f}_2 (\mathbf{x}')+\delta\mathbf {f}_3 (\mathbf{x}')+\delta\mathbf {f}_4 (\mathbf{x}')\bigg)\cdot\mathbf{u}(\mathbf{x}'){\rm d}\mathbf{x}'
\\&-\eta(\mathbf{0})\int_{\Gamma_{h}^{\pm}}\mathbf{u}(\mathbf{x}')\cdot\delta\mathcal{R}(\mathbf{v}^{(1,2)})(\mathbf{x}'){\rm d}\sigma
-\mathcal{R}(\mathbf{v}^{(1,2)})(\mathbf{0})\int_{\Gamma_{h}^{\pm}}\delta\eta(\mathbf{x}')\mathbf{u}(\mathbf{x}'){\rm d}\sigma\\&-\int_{\Gamma_{h}^{\pm}}\delta\eta(\mathbf{x}')\mathbf{u}(\mathbf{x}')\cdot\delta\mathcal{R}(\mathbf{v}^{(1,2)})(\mathbf{x}'){\rm d}\sigma.
\end{aligned}
\end{equation}

Multiplying $s^{2}$ on the both side of \eqref{a0}, by virtue of \eqref{e}, \eqref{2.86}, \eqref{2.87}, \eqref{a1}, \eqref{a2}, \eqref{a3}, \eqref{a4}, \eqref{a5}, \eqref{a6} and  \eqref{a7}, and letting $s\rightarrow \infty$, we have
\begin{equation}
 \begin{aligned}
\eta(\mathbf{0})\mathcal{R}(\mathbf{v}^{(1,2)})(\mathbf{0})
\cdot
\begin{pmatrix}
    1\\
   {\rm i}
\end{pmatrix}
\bigg(\mu(\theta_M )^{-2}+   \mu(\theta_m )^{-2}\bigg)=0.
 \end{aligned}
\end{equation}
Since the corner is non-degenerate, we know that
\begin{equation}\notag
\mu^{-2}(\theta_{m})+\mu^{-2}(\theta_{M})\neq 0.
\end{equation}
Hence, we can obtain \eqref{eq:thm 32 con1}, which implies that $\mathbf{v}^{(1,2)}((\mathbf{0},x_3^c))=\mathbf{0}$. 

\medskip
\noindent {\bf Part II.}~ Consider the PDE system \eqref{3} when $\lambda \neq 0$. According to  \cite[Theorem 3.2]{DCL}, under the condition $q\mathbf{w}_\sfr\in C^{\alpha}(\overline{S_{h}})^2$ and $\mathbf{v}_\sfr-\mathbf{w}_\sfr\in C^{\alpha}(\overline{S}_{h}\times [-M,M])^3$ we know that $v_3((\mathbf{0},x_3^c))=0$ for all $x_3^c\in (-L,L)$. For $\lambda =0$, we can use a similar argument in the proof of Corollary \ref{cor:34 eta} for proving the 2nd case to establish the vanishing property of $v_3$ at the edge point. 

The proof is complete. 
\end{proof}

In Corollary \ref{cor:3 added}, if  global H\"older continuous regularities for $\mathbf v$ and $\mathbf w$ are fulfilled, similar to Theorem \ref{thm:33}, we can prove that $\mathbf v$ and $\mathbf w$ must vanish at the edge corner point by removing the assumption that $\mathbf v-\mathbf w$ is $C^{1,\alpha}$-continuous at the edge corner. Before  Corollary  \ref{cor:3 added},  let us recall a Schauder estimate for the Lam\'e operator, which is a special case of \cite[Theorem 5.2]{Giaquinta}. 

\begin{lem}\cite[Proposition 2.7]{BL2021}\label{lem:35 sha}
	Let $D\subset \mathbb R^3$ be a bounded Lipschitz domain. Let $\mathbf  U \in H^1(D)^3$ solve 
	\begin{equation}\label{eq:382 U}
		\begin{cases}
			\mathcal L \mathbf U=\mathbf f \mbox{ in } D,\\ 
			\mathbf U-\mathbf g \in H^1_0(D
			)^3
		\end{cases}
	\end{equation}
	for some $\mathbf f \in C^\alpha(D )$, $\mathbf g \in C^{1,\alpha}(\overline{D} )^3$ and $\alpha \in (0,1)$. Then it holds that $\mathbf U \in C^{1,\alpha} (\overline{D} )$. 
\end{lem}

\begin{cor}\label{cor:3 added}
Consider the same setup in Theorem \ref{thm:33}. Assume that $\mathbf{v},\mathbf{w}\in H^1(\Omega )^3 \cap  C^{\alpha}(\overline {\Omega} )^3$ ($0<\alpha<1$) are the generalized elastic transmission eigenfunctions to \eqref{eq:lame1}. Suppose further that $\eta \in C^\alpha(\overline{\Gamma}_h^\pm  \times [-M,M]  )$ with $\eta(\mathbf{0})\neq 0$, and $q \mathbf{w}\in C^\alpha(\overline{\Omega })$ for $0< \alpha  <1$. If the corner is non-degenerate, then $\mathbf{v}$ and $\mathbf{w}$ vanish at the edge point $(\mathbf 0,x_3^c) \in \mathbb{R}^3$ of $S_h \times (-M,M)$, where $x_3^c \in (-M,M)$.
	
	\end{cor}

\begin{proof}
%
%
	Let $\mathbf U=\mathbf v-\mathbf w$. It can be verified that $\mathbf U$ solves \eqref{eq:382 U} for $\mathbf f=\omega^2 q\mathbf w -\omega^2 \mathbf v  $, $\mathbf g=\mathbf 0$ and $D=\Omega$. Due to $q\mathbf{w} \in C^{\alpha }(\overline{\Omega } )^{3}$ and $\mathbf{v} \in C^{\alpha }(\overline{\Omega } )^{3}$,  we have $\mathbf f\in C^\alpha(\overline \Omega ) $. Hence by Lemma \ref{lem:35 sha}, we further have $\mathbf v-\mathbf w\in C^{1,\alpha} (\overline \Omega )$. Therefore the assumptions in Theorem \ref{thm:33} are fulfilled. We readily finish the proof of this corollary by Theorem \ref{thm:33}. 
\end{proof}

\section{unique recovery results for the inverse elastic problem }\label{sect:4}

In this section, we apply the geometric property of generalized elastic transmission eigenfunctions established in the previous sections to the study of the unique recovery for the inverse elastic problem \eqref{eq:so2}. 

We first introduce a more general formulation of the inverse elastic problem. Let $\Omega$, $\mathbf{u}^i$, $V$ and $q=1+V$ be those introduced in Section 1.1. Let $\eta\in L^\infty(\partial\Omega)$ with $\Im\eta\geq 0$. Consider the following elastic scattering system
\begin{equation}\label{eq:model}
	\begin{cases}
	{\mathcal L} \mathbf u^-+\omega ^2 q \mathbf u^-=\mathbf 0 & \mbox{ in }\ \Omega, \\[5pt]
{\mathcal L} \mathbf u^+ +\omega^2  \mathbf u^+=\mathbf 0 & \mbox{ in }\ \mathbb R^n \backslash \Omega, \\[5pt]
	\mathbf u^+= \mathbf u^-,\quad T_\nu \mathbf u^+ + \eta  \mathbf u^+=T_\nu \mathbf u^- & \mbox{ on }\ \partial \Omega, \\[5pt]
	\mathbf u^+=\mathbf u^i+\mathbf u^{\rm sc} & \mbox{ in }\ \mathbb R^n \backslash \Omega, \\[5pt]
	\displaystyle{ \lim_{r\rightarrow\infty}r^{\frac{n-1}{2}}\left(\frac{\partial \bmf{u}_\beta^{\mathrm{sc} }}{\partial r}-\mathrm{i}k_\beta \bmf{u}_\beta^{\mathrm{sc} }\right) =\,\mathbf 0,} & \beta=p,s. 
	\end{cases}
	\end{equation}
If $\eta\equiv 0$, \eqref{eq:model} is reduced to be \eqref{eq:scat1}. The well-posedness of \eqref{eq:model} for the case $\eta\equiv 0$ was investigated in \cite{Hahner93, Hahner98}. By following a standard variational argument in \cite{Hahner93, Hahner98}, one can show the unique existence of a solution $\mathbf{u}=\mathbf{u}^-\chi_\Omega+\mathbf{u}^+\chi_{\mathbb{R}^n\backslash\overline{\Omega}}\in H_{loc}^1(\mathbb{R}^n)^n$ to \eqref{eq:model}. However, it is not the focus of this article and in what follows, we always assume the well-posedness of the system \eqref{eq:model}. We write $\bmf{u}_\beta^\infty(\hat{\mathbf{ x}}; \mathbf{u}^i),\, \beta=t,p, \mbox{ or } s$ to signify the far-field patterns associated with \eqref{eq:model} and consider the following inverse problem
\begin{equation}\label{eq:ip1}
\mathcal{F}(\Omega; q, \eta )=\bmf{u}_\beta^\infty(\hat{\mathbf{ x}};\mathbf{u}^i),\quad \beta=t,p, \mbox{ or } s,
\end{equation}
where $\mathcal{F}$ is implicitly defined by the scattering system \eqref{eq:model}. We are particularly interested in the geometrical inverse problem of recovering $\Omega$ independent of $q$ and $\eta$. On the other hand, it is pointed out that if $\Omega$ can be recovered, the boundary parameter $\eta$ can be recovered as well by a standard argument. To the best of our knowledge, the inverse problem \eqref{eq:ip1} is new to the literature, in particular for the case $\eta\equiv\hspace*{-3.5mm}\backslash\, 0$. If the far-field pattern is given associated with a single incident wave, then it is referred to as a single far-field measurement for the inverse problem \eqref{eq:ip1}, otherwise it is referred as multiple measurements. It can be directly verified that with a single far-field measurement, the inverse problem \eqref{eq:ip1} is formally determined since both $\mathbb{S}^{n-1}$ (on which the far-field pattern is given) and $\partial\Omega$ (it completely determines the shape of $\Omega$) are $(n-1)$-dimensional manifolds. The inverse shape problem associated with a single far-field measurement constitutes a longstanding challenging problem in the inverse scattering theory \cite{CKreview}, a fortiori the one described above for the elastic scattering. In the rest of the paper, we shall apply the geometric results derived in the previous sections to establish several novel unique identifiability results to the inverse shape problem \eqref{eq:ip1} associated with a single far-field measurement within a certain generic (though still specific) scenario. Before that, we would like to mention in passing some related results in the literature \cite{Hahner93,Hahner02} on the unique identifiability for the inverse problem \eqref{eq:ip1} associated with multiple far-field measurements and $\eta\equiv 0$.

First, we introduce the admissible class of elastic scatterers in our study.  Let $W_{{\mathbf x}_c}(\theta_W)$ be an open sector in $\mathbb R^2$ with the vertex ${\mathbf x}_c$ and an open angle $\theta_W $. Denote
	\begin{equation}\label{eq:thm21}
	\begin{split}
		\Gamma_h^\pm({\mathbf x}_c)&: =\partial W_{{\mathbf x}_c} (\theta_W )  \cap B_h({\mathbf x}_c),\quad S_{h} (\mathbf x_c):= 
		B_{h} (\mathbf x_c) \cap W_{\mathbf x_c} (\theta_W ),
	\end{split}
	\end{equation}
	where $B_{h} (\mathbf x_c)$ is an open disk  centered at $\mathbf x_c$ with the radius $h\in \mathbb R_+.$

\begin{defn}~\label{def:adm}
Let {$(\Omega; q, \eta)$}  be an elastic scatterer. Consider the scattering problem \eqref{eq:model} and $\mathbf{u}^i$ is the incident wave field therein.
The scattering configuration is said to be be admissible if it fulfils the following conditions:
\begin{itemize}
\item[(a)]  $\Omega$ is a bounded simply connected  Lipschitz domain in $\mathbb{R}^{n}$ with a connected complement, and $q\in L^\infty(\Omega)$, $\eta\in L^\infty(\partial\Omega)$ are real valued functions.

\item[(b)] Following the notations in Theorem \ref{thm:23}, if $\Omega\Subset \mathbb R^2 $ possesses a  planar corner  $B_h({\mathbf x}_c) \cap \Omega= \Omega\cap W_{{\mathbf x}_c}(\theta_W )$ where ${\mathbf x}_c$ is the vertex of the sector $W_{{\mathbf x}_c}(\theta_W )$ and the open angle $\theta_W$ of $  W_{\mathbf x_c}(\theta_W )  $ satisfies $\theta_W \in (0,\pi)  $,  then $q  \in C^\alpha(\overline{ S}_h (\mathbf x_c ))  $ and $\eta\in C^\alpha(\overline{\Gamma_h^\pm(\mathbf x_c}))$ for $\alpha\in (0,1)$ with $q(\mathbf x_c)\neq 1$ and $\eta(\mathbf x_c)\neq 0$, where $S_{h}(\mathbf x_c) $ and $\Gamma_h^\pm (\mathbf x_c)$ are defined in \eqref{eq:thm21}. Similarly, following the notations in Theorem \ref{thm:31}, if  $\Omega\Subset \mathbb R^3 $ possesses a 3D edge corner $(B_h ({\mathbf x}_c')  \times (-M,M) ) \cap \Omega=S_h ({\mathbf x}_c') \times (-M,M)$,  where ${\mathbf x}_c$ is the vertex of $S_h(\mathbf x_c)$ contained in the sector $W_{{\mathbf x}_c}(\theta_W )$ and the open angle $\theta_W$ of $  W_{\mathbf x_c}(\theta_W )  $ satisfies $\theta_W \in (0,2\pi) \backslash\{\pi\} $, then $q\in C^\alpha(\overline{S_h } (\mathbf {x}'_c)\times [-M,M] )$, $\eta 
\equiv 0$ on $\partial \Omega$.  


\item[(c)] The total wave field $\mathbf u$ is non-vanishing everywhere in the sense that for any $\mathbf x\in\mathbb{R}^n$,
\begin{equation}\label{eq:nn2}
	\lim_{ \rho \rightarrow +0 }\frac{1}{m(B(\mathbf x, \rho  ))} \int_{B(\mathbf x, \rho )} |\mathbf u(\mathbf x)|  {\rm d} x\neq 0.
	\end{equation}
\end{itemize}

\end{defn}

\begin{rem}
	The assumption \eqref{eq:nn2} can be fulfilled in certain generic scenario. For an illustration, let us consider a specific case by requiring the angular frequency $\omega\in\mathbb{R}_+$ sufficiently small. In the physical scenario, this is also equivalent to requiring that the size of the scatterer, namely $\mathrm{diam}(\Omega)$ is  sufficiently small (compared to the operating wavelength). In such a case, from a physical point of view, the interruption of the incident field due to the scatterer should be small, i.e. $\mathbf{u}^{\rm sc}$ should be small compared to $\mathbf{u}^i$. Hence, if $\mathbf{u}^i$ is non-vanishing everywhere (say, $\mathbf{u}^i$ is a plane wave, namely, the Herglotz wave \eqref{eq:h2d} or \eqref{eq:h3d} with the densities being delta-distributions), then $\mathbf{u}=\mathbf{u}^i+\mathbf{u}^{\rm sc}$ should be non-vanishing everywhere. However, a rigorous justification of such a physical intuition will cost lengthy arguments and we choose not to explore more about this point. 
	\end{rem}
%

Next, we present a technical lemma concerning the regularity of the solution to the Lam\'e system around a corner (cf.  \cite{Nicaise, Rossle}). We would also like to refer interested readers to \cite{Dauge88,Grisvard,Cos} on classical results of decomposing solutions to elliptic PDEs in corner domains.


\begin{lem}\cite[Theorem 2.3]{Nicaise}\label{lem:lem reg}
	Let $\Omega$ be a bounded open connected subset of $\mathbb R^2$, where the boundary $\partial \Omega$ of $\Omega$ is the union of a finite number of line segment $\overline{\Gamma}_\ell$, $\ell \in \Xi  $. Fix a partition of $\Xi$ into $\mathcal D \cup \mathcal N$, where $\mathcal D$ and $\mathcal N$ correspond to Dirichlet and Neumann boundary conditions respectively. Given a vector field $\mathbf f \in L^2(\Omega )^2$ and $\mathbf g^{(\ell)} \in H^{1/2} (\Gamma_\ell )^2$ for all $\ell \in \mathcal N$, consider the weak solution $\mathbf u\in H^1(\Omega )^2$ of the Lam\'e system
	\begin{equation}\label{eq: L op}
		\mathcal L \mathbf u=\mathbf f \mbox{ in } \Omega
	\end{equation}
	with mixed boundary conditions
	\begin{equation*}
		\begin{cases}
			\mathbf u=\mathbf 0 & \mbox{ on } \Gamma_\ell, \quad \ell \in \mathcal D,\\
			T_\nu \mathbf u=\mathbf g^{(\ell)}& \mbox{ on } \Gamma_\ell, \quad \ell \in \mathcal N. 
		\end{cases}
	\end{equation*}
	If $\Omega$ satisfies the assumption: $\forall \ell_1, \ell_2 \in \Xi $ such that $\overline{\Gamma }_{\ell_1} \cap \overline{\Gamma }_{\ell_2} \neq \emptyset  $, the interior angle $\angle( \Gamma _{\ell_1}, \Gamma _{\ell_2})$ fulfills $\angle( \Gamma _{\ell_1}, \Gamma _{\ell_2})< 2\pi$ and moreover, if $\ell_1 \in \mathcal D$ and $\ell_2 \in \mathcal N$,  $\angle( \Gamma _{\ell_1}, \Gamma _{\ell_2})< \pi$, then a solution $\mathbf u$ of \eqref{eq: L op} with the data $\mathbf f \in L^2(\Omega )^2$ and $\mathbf g^{(\ell)} \in H^{1/2} (\Gamma_\ell )^2$,   $\forall \ell \in \mathcal N$ satisfies $\mathbf u \in H^{3/2+\varepsilon }(\Omega )^2$ for some $\varepsilon >0 $. 
\end{lem}

\begin{lem} \label{lem:41}
Let $S_{2h}=W\cap B_{2h}$ and $\Gamma_{2h}^\pm =\partial S_{2h}\backslash \partial B_{2h} $, where $W$ is the infinite sector defined in \eqref{eq:lame2}  with the opening angle $\theta_W \in (0,\pi)$. Suppose that $\mathbf u\in H^1(B_{2h})^2$ satisfies 
	\begin{equation}\label{eq:lem23}
\begin{cases}
\mathcal L \mathbf u^-+\omega^2 q_- \mathbf u^-=\mathbf 0 & \mbox{ in }\ S_{2h}, \\[5pt] 
\mathcal L\mathbf u^+ +\omega^2  \mathbf  u^+=\mathbf 0 & \mbox{ in }\ B_{2h} \backslash {\overline{ S_{2h}}}, \\[5pt] 
\mathbf u^+=\mathbf  u^- & \mbox{ on }\ \Gamma_{2h}^\pm, 
\end{cases}
\end{equation} 
where $\mathbf u^+=\mathbf u|_{{B_{2h}\backslash {\overline{ S_{2h}}}}}$, $\mathbf u^-=\mathbf u|_{S_{2h}}$,  $\omega$ is a positive constant and $q_{-}\in L^{\infty}(S_{2h})$. Assume that $\mathbf u^+$ is  real analytic in ${B_{2h}\backslash {\overline{ S_{2h}}}}$. There exists $\alpha\in (0, 1)$ such that $\mathbf u^- \in C^\alpha(\overline{S_h})$. 
\end{lem}

\begin{proof}

Since $\mathbf u^+$ is real analytic in $B_{2h}\backslash\overline{S_{2h}}$,  we let $\mathbf w$ be the analytic extension of $\mathbf u^+|_{B_h\backslash\overline{S_h}}$ in $B_h$. By using the transmission condition on $\Gamma_h^\pm$, one clearly has that $\mathbf u^-=\mathbf u^+=\mathbf w$ on $\Gamma_h^\pm$. Set $\mathbf v=\mathbf u^- -\mathbf w$. Set $ \boldsymbol{
l}_h$ to denote the line segment with the staring point $h(\cos\theta_m,\sin \theta_m )\in \Gamma_h^-$ and the ending point $h(\cos\theta_M,\sin \theta_M )\in \Gamma_h^+$.   It can be directly verified that
	\begin{equation}\notag 
\mathcal L\mathbf  v=\mathbf f  \mbox{ in }\ \mathcal T_h; \ \ T_\nu \mathbf v= \mathbf g  \mbox{ on }\ \boldsymbol{l}_h;\ \ 
\mathbf v=\mathbf  0  \mbox{ on }\ \Gamma_h^\pm, 
\end{equation}
where $\mathcal T_h$ is the open triangle formed by $\Gamma_h^\pm$ and $\boldsymbol{l}_h$, $\mathbf f=-\omega^2 (\mathbf w+ q_-\mathbf  u^-)\in L^2(S_h)^2$ and  $\mathbf g\in C^{\infty}(\overline {\boldsymbol{l}_h})^2$. By virtue of Lemma \ref{lem:lem reg}, one has that $\mathbf v \in H^{3/2+\varepsilon }(T_{h})^2$ with $\varepsilon\in \mathbb R_+$. Therefore by the Sobolev embedding theorem, it is clear that there exists $\alpha\in (0, 1)$ such that $\mathbf v\in C^\alpha(\overline{ S_h} )^2$. Hence, we readily have that $\mathbf u^-\in C^\alpha ( \overline{ S_h} )^2$. 

The proof is complete. 
\end{proof}

\begin{rem}
	We would like to point out that the regularity result in Lemma \ref{lem:41} in general does not for the three dimensional case. This is mainly due to the fact that the corresponding regularity result in Lemma \ref{lem:lem reg} is generically not true around a general polyhedral corner in $\mathbb R^3$. Hence, we exclude the generalized transmission condition in \eqref{eq:lame1} for an admissible  elastic scatterer in Definition \ref{def:adm}. That is, we only consider the case $\eta \equiv 0$ on the boundary  of an admissible elastic scatterer in $\mathbb R^3$, which is different from the two-dimensional case. 
	\end{rem}

\begin{thm}\label{thm:41}
Consider the elastic scattering problem \eqref{eq:model} associated  the incident elastic wave field $\mathbf u^i$ and  two elastic scatterers $(\Omega_j; q, \eta_j)$ being admissible scattering configuration. Let $\bmf{u}_\beta^{j,\infty}(\hat{\mathbf{ x}}; \mathbf{u}^i )$ be the far-field pattern associated with the scatterer $(\Omega_j; q_j, \eta_j)$ and the incident field $\mathbf u^i$, $\beta=t,\,p$ or $s$. If
\begin{equation}\label{eq:nn1}
\bmf{u}_\beta^{1,\infty}(\hat{\mathbf{ x}}; \mathbf{u}^i )={\bmf{u}}_\beta^{2,\infty}(\hat{\mathbf{ x}}; \mathbf{u}^i), \ \ \hat{\mathbf x}\in\mathbb{S}^{n-1},
\end{equation}
for a fixed incident wave $\mathbf u^i$, then one has that
\begin{equation}\label{eq:nn3}
\Omega_1\Delta\Omega_2:=\big(\Omega_1\backslash\Omega_2\big)\cup \big(\Omega_2\backslash\Omega_1\big)
\end{equation}
cannot possess a corner. Hence, if $\Omega_1$ and $\Omega_2$ are convex polygons in $\mathbb{R}^2$ or convex polyhedra in $\mathbb{R}^3$, one must have
\begin{equation}\label{eq:nn4}
\Omega_1=\Omega_2.
\end{equation}

\end{thm}
\begin{proof} 
  We prove  \eqref{eq:nn3} by contradiction. Suppose  that there is a corner contained in $\Omega_1\Delta\Omega_2$. Without loss of generality,  we may assume that the vertex $O$ of the corner $\Omega_2 \cap W_{\mathbf x_c}(\theta_W)$ is such that $O \in \partial \Omega_2$ and $ O \notin \overline{\Omega}_1$. Furthermore, one may assume that in two dimensions, $O$ is the origin, whereas in three dimensions, the edge corner point $O=(\mathbf x_c', x_3^c)$ of the 3D edge corner $(B_h ({\mathbf x}_c')  \times (-M,M) ) \cap \Omega_2=S_h ({\mathbf x}_c') \times (-M,M)$ fulfils that $\mathbf x_c'$ is the origin of $\mathbb R^2$.
 
	
	Due to \eqref{eq:nn1}, applying Rellich's Theorem (see \cite{Hahner98,Kupradze}),  we know that $\mathbf u_1^{\rm sc}=\mathbf u_2^{\rm sc}$ in $\mathbb R^n \backslash (  \overline{\Omega}_1 \cup \overline{\Omega}_2 )$. Thus
	\begin{equation}\label{eq:u1u2}
		\mathbf u_1(\mathbf x)=\mathbf u_2(\mathbf x)
	\end{equation}
	for all $\mathbf x \in \mathbb R^n \backslash (  \overline{\Omega}_1 \cup \overline{\Omega}_2 )$.  In what follows, we consider two separate cases.

\medskip

\noindent {\bf Case 1 ($n=2$):}~Following the notations in \eqref{eq:SIGN} and the setup of Theorem \ref{thm:23},  we have from \eqref{eq:u1u2} that
	\begin{equation}\label{eq:413}
		\mathbf 	u_2^-=\mathbf u_2^+=\mathbf u_1^+,\quad T_\nu \mathbf u_2^- = T_\nu  \mathbf u_2^+ + \eta_2 \mathbf u_2^+=T_\nu  \mathbf u_1^+ + \eta_2 \mathbf u_1^+ \mbox{ on } \Gamma_h^\pm,
	\end{equation}
	where the superscripts $(\cdot)^-, (\cdot)^+$ stand for the limits taken from $\Omega_2$ and $\mathbb{R}^2 \backslash\overline{\Omega_2}$ respectively. Moreover, we take $h\in\mathbb{R}_+$ sufficient small such that
\begin{equation}\label{eq:414}
	\mathcal L \mathbf u_1^+ + \omega^2 \mathbf u_1^+ =\mathbf 0 \mbox{ in } B_h, \quad \mathcal L \mathbf u_2^-  +\omega^2 q_2 \mathbf u_2^-=\mathbf 0 \mbox{ in } S_h.
	\end{equation}
	
Since  $(\Omega_j; q_j, \eta_j)$, $j=1,2$, are admissible, we know that $q_2 \in C^\alpha(\overline{S_h})$ and $\eta_j\in C^\alpha(\overline{\Gamma}_h^\pm)$. Clearly $\mathbf u_2^-\in H^1(S_h)^2$ and  $\mathbf u_1^+$ is  real analytic in ${B_{2h}\backslash {\overline{ S_{2h}}}}$. According to Lemma \ref{lem:41},  we know that $\mathbf u_2^-\in C^{\alpha}(S_{h})^2 $, which implies that $q_2 \mathbf u_2^- \in C^\alpha(\overline{S_h})$.   Using the admissibility condition (b) in Definition \ref{def:adm}, by \eqref{eq:414} and applying Theorem  \ref{thm:23}, and also utilizing the fact that $\mathbf u_1$ is continuous at the vertex $\mathbf 0$, we have
	$$
	\mathbf u_1(\mathbf 0)=\mathbf 0,
	$$
	which  contradicts to the admissibility condition (c) in Definition \ref{def:adm}.

\medskip

\noindent {\bf Case 2 ($n=3$):} Since  $(\Omega_j; q_j, \eta_j)$, $j=1,2$, are admissible, we know that $\eta_2\equiv 0$ on $\partial \Omega_2$. Therefore, from \eqref{eq:u1u2}, following the setup of Theorem \ref{thm:31}, it yields that 
\begin{equation}\label{eq:413 3d}
		\mathbf 	u_2^-=\mathbf u_2^+=\mathbf u_1^+,\quad T_\nu \mathbf u_2^- = T_\nu  \mathbf u_2^+ =T_\nu  \mathbf u_1^+ \mbox{ on } \Gamma_h^\pm \times (-M,M). 
	\end{equation}
	Moreover, we take $h\in\mathbb{R}_+$ sufficient small such that
\begin{equation}\label{eq:414}
	\mathcal L \mathbf u_1^+ + \omega^2 \mathbf u_1^+ =\mathbf 0 \mbox{ in } B_h, \quad \mathcal L \mathbf u_2^-  +\omega^2 q_2 \mathbf u_2^-=\mathbf 0 \mbox{ in } S_h \times (-M,M). 
	\end{equation}

	  By the well-posedness of the direct problem \eqref{eq:model} with $\eta \equiv 0$, we know that $\mathbf u_2 \in H^1(B_{R})^3$  where $B_{R}$ is a ball centered at the origin with the radius $R \in \mathbb R_+$ such that $\overline{S_h} \times [-M,M] \Subset B_R $ and $B_R \Subset \mathbb R^3 \backslash \overline{\Omega_1}$. Let  $\widetilde{ q}_2=q_2 \chi_{\Omega_2 }+1 \chi_{B_R\backslash \overline \Omega_2 } $. Then 
	$$
	\mathcal L\mathbf u_1 +\omega^2 \mathbf u_1=\mathbf 0 \mbox{ in } B_R ,\quad  \mathcal L\mathbf u_2 +\omega^2 \widetilde{q}_2\mathbf u_2=\mathbf 0  \mbox{ in } B_R .
	$$  
	Since  $\widetilde{ q}_2\in L^\infty(B_R)$, by the interior elliptic regularity estimate \cite{McLean}, we have  $\mathbf u_2^- \in H^2(B_{R'})^3 $, where $B_{R'}\Subset B_R \backslash (\overline{S_h} \times [-M,M]) $.  Again using the interior elliptic regularity estimate, we have $\mathbf u_2^- \in W^{2,4}(\overline{S_h} \times [-M,M])^3$. Using the Sobolev embedding theorem, we have $\mathbf u_2^- \in C^{1,1/4}(\overline{S_h} \times [-M,M])^3$. Clearly, $\mathbf u_1^+$ is real analytic in $\overline{S_h} \times [-M,M]$. Therefore one has $\mathbf u_1^+-\mathbf u_1^+  \in C^{1,1/4}(\overline{S_h} \times [-M,M])^3$. Since  $(\Omega_j;q_j, \eta_j)$, $j=1,2$, are admissible, we know that $q_2 \in C^\alpha(\overline{S_h} \times [-M,M] )$.  Using the admissibility condition (b) in Definition \ref{def:adm}, by \eqref{eq:414} and applying Corollary   \ref{cor:34 eta}, we have
	$$
	\mathbf u_1(\mathbf 0)=\mathbf 0,
	$$
	which  contradicts to the admissibility condition (c) in Definition \ref{def:adm}.

	
\medskip

The conclusion \eqref{eq:nn4} can be immediately obtained by using the contradiction argument and \eqref{eq:nn3}. 

	The proof is complete.
\end{proof}



{{Based on Definition \ref{def:adm}, if we further assume that the surface parameter $\eta$ is constant, we can recover $\eta$ simultaneously in $\mathbb R^2$ once the shape of the scatterer, namely $\Omega$ is determined. However, in determining the surface conductive parameter, we need to assume that $q_1=q_2:=q$ are known.
		
		\begin{thm}\label{Th: unique eta}
			Consider the elastic scattering problem \eqref{eq:model} in $\mathbb R^2$ associated with  the incident elastic wave field $\mathbf u^i$ and  two elastic scatterers $(\Omega_j; q, \eta_j)$ being admissible scattering configuration, where $\Omega_j=\Omega$ for $j=1,2$ and $\eta_j\neq0$, $j=1,2$, are two constants. Let $\bmf{u}_\beta^{j,\infty}(\hat{\mathbf{ x}}; \mathbf{u}^i )$ be the far-field pattern associated with the scatterer $(\Omega_j; q_j, \eta_j)$ and the incident field $\mathbf u^i$, $\beta=t,\,p$ or $s$. Suppose that
			\begin{equation}\label{eq:far}
			\bmf{u}_\beta^{1,\infty}(\hat{\mathbf{ x}}; \mathbf{u}^i )={\bmf{u}}_\beta^{2,\infty}(\hat{\mathbf{ x}}; \mathbf{u}^i), \ \ \hat{\mathbf x}\in\mathbb{S}^{1},
			\end{equation}			
				for a fixed incident wave $\mathbf u^i$. Then if $\omega $ is not an eigenvalue of the partial differential operator $\mathcal L+\omega ^2q$ in $H_0^1(\Omega)$, we have $\eta_1=\eta_2$.
		\end{thm}
		
		 \begin{proof}	
	  Due to \eqref{eq:far},  we have $\mathbf u_1^+=\mathbf u_2^+$ for all $\mathbf x\in\mathbb{R}^2\backslash\overline{\Omega}$ and thus $T_\nu \mathbf u_1^+=T_\nu \mathbf u_2^+$ on $\partial\Omega$. Combining with the transmission condition in the scattering problem \eqref{eq:model}, we deduce that
			\begin{equation}\notag
			\mathbf u_1^-=\mathbf u_1^+=\mathbf u_2^+=\mathbf u_2^-\ \mbox{ on }\ \partial\Omega. 
			\end{equation}
			Thus, we have
			\begin{equation}\notag
			T_\nu(\mathbf u_1^- -\mathbf u_2^-)=T_\nu(\mathbf u_1^+-\mathbf u_2^+)+\eta_1 \mathbf u_1^+-\eta_2\mathbf u_2^+=(\eta_1-\eta_2)\mathbf u_1^- \ \mbox{ on }\ \partial\Omega.
			\end{equation}
			Set $\mathbf v:=\mathbf u_1^- -\mathbf u_2^-$. Then $\mathbf v$ fulfills
			\begin{equation}\label{eq:v}
			\begin{cases}
			(\mathcal L +\omega ^2q) \mathbf v=\mathbf 0 & \mbox{ in }\ \Omega,\\
			\mathbf v=\mathbf 0 &\mbox{ on }\ \partial\Omega,\\
			T_\nu \mathbf v=(\eta_1-\eta_2)\mathbf u_1^- &\mbox{ on }\ \partial\Omega.
			\end{cases}	
			\end{equation}
			Since $\omega$ is not an eigenvalue of the operator $\mathcal L+\omega ^2q$ in $H_0^1(\Omega)$,
			 one must have $\mathbf v=0$ to \eqref{eq:v}. Substituting this into the Neumann boundary condition of \eqref{eq:v}, we know that $(\eta_1-\eta_2)\mathbf u_1^-=T_\nu \mathbf v=\mathbf 0$ on $\partial\Omega$.
			
			Next, we prove the uniqueness of $\eta$ by contradiction. Assume that $\eta_1\neq\eta_2$. Since $(\eta_1-\eta_2)\mathbf u_1^-=\mathbf 0$ on $\partial\Omega$ and $\eta_j$, $j=1,2$ are constants, we can deduce that $\mathbf u_1^-=\mathbf 0$ on $\partial\Omega$. Then $u_1^-$ satisfies
			\begin{equation}\notag
			\begin{cases}
			(\mathcal L +\omega ^2q)\mathbf u_1^-=\mathbf 0 & \mbox{ in }\ \Omega,\\
			\mathbf u_1^-=\mathbf 0 & \mbox{ on }\ \partial\Omega.
			\end{cases}
			\end{equation}
		    Similar to \eqref{eq:v}, this Dirichlet problem also only has a trivial solution $\mathbf u_1^-=\mathbf 0$ in $\Omega$ due to that $\omega $ is not an eigenvalue of $\mathcal L +\omega ^2q$ in $H_0^1(\Omega)$. Hence, we can derive  $\mathbf u_1^+=\mathbf u_1^-=0$ and
			\begin{equation}\notag
			T_\nu \mathbf u_1^-=T_\nu \mathbf u_1^++\eta_1\mathbf u_1^+=T_\nu \mathbf  u_1^+=\mathbf 0 \mbox{ on } \partial\Omega,
			\end{equation}
			which implies that $\mathbf u_1\equiv \mathbf 0$ in $\mathbb{R}^n$ and thus $\mathbf u_1^{\rm sc}=-\mathbf u^i$. This contradicts to the fact that $\mathbf u^{\rm sc}_1$ satisfies the Kupradze radiation condition.
			
			The proof is complete.
		\end{proof}

\begin{rem}
In Theorem \ref{Th: unique eta}, it is required that $\omega$ is not an eigenvalue of $\mathcal L+\omega ^2q$ in $H_0^1(\Omega)$. Clearly, if $q$ is negative-valued in $\Omega$, this condition is obviously fulfilled. On the other hand, if $q$ is positive-valued in $\Omega$, then this condition can be readily fulfilled when $\omega \in\mathbb{R}_+$ is sufficiently small.
\end{rem}

\vspace*{-.3cm}

\section*{Acknowledgement}

The work  of H Diao was supported in part by the National Natural Science Foundation of China under grant 11001045 and by National Key Research and Development Program of China (No. 2020YFA0714102).  The work of H Liu was supported by the startup fund from City University of Hong Kong and the Hong Kong RGC General Research Fund (projects 12301420, 12302919 and 12301218).

\end{document}